\documentclass[a4paper,reqno,11pt]{amsart}
\usepackage[left=2.8cm,right=2.8cm,top=2.8cm,bottom=2.8cm]{geometry}
\usepackage[utf8]{inputenc}	
\usepackage{indentfirst} 
\usepackage{amsmath,amsfonts,amssymb,,mathrsfs,mathtools}
\usepackage{booktabs}
\usepackage{esint} 
\usepackage{bbm}
\usepackage[colorlinks=true]{hyperref}
\usepackage{enumerate}
\usepackage{graphicx}
\usepackage{xcolor}
\usepackage{stmaryrd}
\usepackage{todonotes}
\usepackage[shortlabels]{enumitem}
\usepackage{setspace}
\newcommand{\restr}{%
  \,\raisebox{-.127ex}{\reflectbox{\rotatebox[origin=br]{-90}{$\lnot$}}}\,%
}

\theoremstyle{plain}
\newtheorem{definition}{Definition}[section]
\theoremstyle{plain}
\newtheorem{remark}[definition]{Remark}
\theoremstyle{plain}
\newtheorem{theorem}[definition]{Theorem}
\theoremstyle{plain}
\newtheorem{lemma}[definition]{Lemma}
\theoremstyle{remark}

\theoremstyle{plain}
\newtheorem{corollary}[definition]{Corollary}
\theoremstyle{plain}
\newtheorem{proposition}[definition]{Proposition}
\theoremstyle{plain}

\theoremstyle{plain}

\numberwithin{equation}{section}
\allowdisplaybreaks

\def\XXint#1#2#3{{\setbox0=\hbox{$#1{#2#3}{\int}$}
      \vcenter{\hbox{$#2#3$}}\kern-.5\wd0}}

\definecolor{ao}{rgb}{0.0, 0.5, 0.0}

\def\namedlabel#1#2{\begingroup
    #2%
    \def\@currentlabel{#2}%
    \phantomsection\label{#1}\endgroup
}

\makeatletter
\@namedef{subjclassname@2020}{%
  \textup{2020} Mathematics Subject Classification}
\makeatother

\title[Besicovitch-Federer projection theorem for measures]{Besicovitch-Federer projection theorem for measures}

\author[Emanuele Tasso]{Emanuele Tasso}

\AtEndDocument{
	\bigskip
	\footnotesize{
		\noindent
		 E. Tasso.
		
		\noindent
		Institute of Analysis and Scientific Computing,
		TU Wien,
		
		\noindent
		Wiedner Hauptstra\ss e 8-10, 1040 Vienna, Austria.
		
		\noindent
		E-mails:
		\href{mailto:emanuele.tasso@tuwien.ac.at}{\tt emanuele.tasso@tuwien.ac.at}, 
}}

\begin{document}

\begin{abstract}
 	 In this paper we prove a Besicovitch-Federer type projection theorem
for general finite Borel measures in $\mathbb{R}^n$. Assuming only that
the conditional measures along typical affine $(n-m)$-planes are
purely atomic, we characterize concentration on a purely
$m$-unrectifiable set in terms of the singularity of the projected
measure and the $\mu$-almost everywhere injectivity of typical
orthogonal projections. No regularity or absolute-continuity
assumption on the projected measures is imposed, and the result is new
even for measure of the form $\mu=\mathcal H^m\restr E$.

As a consequence, we obtain a rectifiability criterion for Radon
measures in terms of atomic disintegrations. We then apply this
criterion to currents. In particular, we prove that a finite-mass
classical Federer-Fleming current satisfying the intrinsic
integer-valued push-forward condition is rectifiable, and hence a real
flat chain, without assuming any a priori flat-chain or metric-current
structure. We also show that a Euclidean metric current is rectifiable
whenever its slices are atomic along a set of projections of positive
Grassmannian measure, with no assumption on the mass of its boundary.
Finally, the projection theorem extends to separable locally compact
metric spaces.

		\medskip 
  


	\end{abstract}

 \maketitle

\section{Introduction}
A central theme in geometric measure theory is the relationship between geometric irregularity and the behavior of orthogonal projections. 
In this framework, irregularity is captured by the notion of \emph{purely unrectifiable} sets, which stand in sharp contrast to \emph{rectifiable} (or regular) ones. 
Since these two classes capture complementary geometric behavior, it is natural to seek criteria that characterize the purely unrectifiable part of a set, or, more genrally, of a measure.

For planar sets of finite length, Besicovitch established a fundamental link between unrectifiability and projections, showing that rectifiability can be detected from the size of one-dimensional projections: a set $E \subset \mathbb{R}^2$ with $\mathcal{H}^1(E) < \infty$ is purely $1$-unrectifiable if and only if
\[
\mathcal{H}^1(\pi_\ell(E))=0, \ \ \text{ for almost every line $\ell$ through the origin,}
\]
where $\pi_\ell \colon \mathbb{R}^2 \to \ell$ denotes the orthogonal projection and the almost every condition refers to the orthogonally invariant measure on the sets of all $1$-dimensional subspaces of $\mathbb{R}^2$.

This landmark result was extended by Federer to subsets of $\mathbb{R}^n$ with finite $m$-dimensional Hausdorff measure \cite{fed1}, and, more recently, by White \cite{whi2} through a completely different approach which consists of a dimension inductive argument assuming the validity of Besicovitch's result. Subsequent developments have broadened the framework to encompass more general families of projections. Hovila et al. \cite{hov} pursued this direction using the notion of transversality, while Pugh \cite{pug} investigated Lipschitz perturbations uniformly close to the identity (see also \cite{gal}). The analysis has further been extended to non-Euclidean settings: De Pauw \cite{depa2} and, independently, Bate, Csörnyei, and Wilson \cite{bate2} obtained results in Banach spaces, both in the negative direction. In contrast, Bate \cite{bate1} strengthened and refined Pugh’s theorem in the positive direction by considering $1$-Lipschitz perturbations and quantifying exceptionality within this class via residual sets with respect to the supremum norm. His result applies in the broad context of complete metric spaces.

All these results, in both Euclidean and metric contexts, rely on the assumption of finite Hausdorff measure. The present work asks whether a characterization of unrectifiability remains possible for general finite Borel measures $\mu$ in $\mathbb{R}^n$, via orthogonal projections. 

When working with orthogonal projections, it is natural to consider the disintegration of $\mu$ with respect to the projection $\pi_V \colon \mathbb{R}^n \to V$ onto the $m$-plane $V$ of the Grasmannian $\mathrm{Gr}(n,m)$, which takes the form 
\[
\mu = \mu_{V,y} \otimes \pi_V \mu, 
\]
for a uniquely determined family of probability measures $(\mu_{V,y})$, each concentrated on its corresponding fiber $\pi_V^{-1}(y)$, and where $\pi_V \mu$ denotes the measure on $V$ obtained as the pushforward of $\mu$ under $\pi_V$.

We denote by $\gamma_{n,m}$ the unique probability measure on $\text{Gr}(n,m)$ which is invariant under rotations.

\subsection{Main result.}
\label{ss:mainresult}
We now state the main result of the paper for the family of orthogonal projections. A more general formulation, valid for separable locally compact metric spaces, will be given at the end of Section~\ref{s:rvs}. We also refer the reader to Subsection~\ref{ss:rectcrit} for a
rectifiability criterion for measures and to
Subsection~\ref{s:conncurrentsintr} for its applications to currents.

For the sake of completeness we recall that a subset $E \subset \mathbb{R}^n$ is purely $m$-unrectifiable if and only if
\[
\mathcal{H}^m(E \cap f(U))=0,
\]
for every bounded set $U \subset \mathbb{R}^m$ and every Lipschitz function $f \colon U \to \mathbb{R}^n$.

\begin{theorem}
\label{t:mainthm1}
Let $\mu$ be a finite Borel measure on $\mathbb{R}^n$ and let $\Sigma \subset \emph{Gr}(n,m)$ be a Borel set with $\gamma_{n,m}(\Sigma)>0$. Assume that for $\gamma_{n,m}$-a.e. $V \in \Sigma$ it holds
\begin{equation}
\label{e:mainthm1.1}
\mu_{V,y} \text{ is purely atomic for $\mathcal{H}^m$-a.e.\ $y \in V$}.
\end{equation}
Then, for $\gamma_{n,m}$-a.e.\ $V \in \Sigma$,
\begin{equation}
\label{e:mainthm1.2}
\pi_V \mu \perp \mathcal{H}^m
\quad \text{and} \quad
\pi_V\colon \mathbb{R}^n \to V
\text{ is injective on a set of full $\mu$-measure}
\end{equation}
if and only if $\mu$ is concentrated on a purely $m$-unrectifiable set.
\end{theorem}

Theorem~\ref{t:mainthm1} asserts that measures concentrated on purely
$m$-unrectifiable sets are graphical with respect to almost every
$m$-plane $V$, in the sense that they are supported on the graph of a
measurable $V^\perp$-valued function defined over an
$\mathcal H^m$-negligible set $B_V\subset V$. Combined with the injectivity criterion proved in
Subsection~\ref{ss:injproj}, Theorem~\ref{t:mainthm1} yields a complete characterization of those measures whose typical orthogonal projections
are injective on sets of full measure
(Corollary~\ref{c:mainthmn2}).

The theorem recovers the classical Besicovitch-Federer projection
theorem by taking
$\mu=\mathcal H^m\restr E$, where
$E\subset\mathbb R^n$ satisfies
$\mathcal H^m(E)<\infty$.
Indeed, in this case the first condition in
\eqref{e:mainthm1.2} is \emph{equivalent} to $\mathcal H^m(\pi_V(E))=0$, since orthogonal projections map
$\mathcal H^m$-negligible sets into
$\mathcal H^m$-negligible sets.
Hence Theorem~\ref{t:mainthm1} recovers exactly the classical
characterization of pure unrectifiability while providing the additional
conclusion that almost every projection is injective on a set of full
$\mathcal H^m\restr E$-measure.

The main interest of Theorem~\ref{t:mainthm1}, however, lies beyond the
$\sigma$-finite Hausdorff setting.
The atomicity assumption is also
motivated by the theory of currents, where it naturally arises as a
slicing property. In Subsection~\ref{ss:rectcrit} we prove a general
rectifiability criterion for measures, and in
Subsection~\ref{s:conncurrentsintr} we apply it to obtain new
rectifiability results for classical and metric currents, without any
assumption on the mass of the boundary.

\subsection{Injectivity of projections}
\label{ss:injproj}
While the study of dimension of sets under orthogonal projections has been extensively developed since the seminal works of Marstrand \cite{mar} (see, \cite{mat7} and the references therein), another closely related question concerns the injectivity of the projections $\pi_V$ on a given set $E$, at least for a typical subspace $V$. This problem is of particular interest in dimension theory, since whenever this injectivity property holds, the map $\pi_V$ provides a topological embedding of $E$ into the $m$-dimensional linear subspace $V$. 

It is known that if $K \subset \mathbb{R}^n$ is compact and $2\dim_B K < m$, where $\dim_B$ denotes the upper box (Minkowski) dimension, then for a generic $m$-plane $V \subset \mathbb{R}^n$, the orthogonal projection $\pi_V$ is injective on $K$. This statement, often referred to as the \emph{Mañé projection theorem}, was first established by Mañé~\cite{Mn81} for topologically generic projections, and later extended to hold for almost every projection in~\cite{SYC91, Rob11}. 
In the smooth category, the corresponding statement is given by the \emph{Whitney embedding theorem}~\cite{Whi36}, asserting that every $m$-dimensional $C^r$-manifold can be generically embedded into $\mathbb{R}^{2m+1}$ by a $C^r$-map ($r \geq 1$ finite or infinite).

It is known that, in general, the bound $2\dim_B K < m$ appearing in the Mañé projection theorem is optimal and cannot be improved (see~\cite[Example~V.3]{HW41}), and moreover, the upper box dimension $\dim_B$ cannot be replaced by the Hausdorff dimension $\dim_{\mathcal{H}}$ 

The situation changes substantially when one replaces strict injectivity of $\pi_V$ by \emph{almost sure injectivity}, that is, injectivity on a set of full $\mu$-measure. The strongest result known in this direction is due to Barański-Gutman-Śpiewak~\cite{bara}. They proved that if $E \subset \mathbb{R}^n$ is a Borel set equipped with a Borel $\sigma$-finite measure $\mu$, and if $\mathcal{H}^m(E)=0$, then $\pi_V$ is injective on a Borel subset $E_V \subset X$ of full $\mu$-measure, for almost every $m$-plane $V$. 
Thus, the minimal embedding dimension required for a typical projection can be reduced by a factor of two.

The border case $\dim_{\mathcal{H}} E = m$ generally exhibits different behavior. For instance, when $\mu$ coincides with the restriction of $\mathcal{H}^1$ to a segment in $\mathbb{R}^2$, the projection onto a typical line is injective, whereas for $\mu = \mathcal{H}^1 \restr \mathbb{S}^1$, clearly no projection onto a line can be injective on a set of full $\mu$-measure. Since in both examples the measure $\mu$ is rectifiable, it is natural to investigate the situation in the purely unrectifiable case.

In this context, our main contribution is the following novel result.
\begin{theorem}[Injectivity of projections]
\label{t:mainthm2}
Let $\mu$ be a finite Borel measure on $\mathbb{R}^n$ and let $\Sigma \subset \emph{Gr}(n,m)$ be a Borel set.  
Then, one of the following two alternatives holds:
\begin{align}
\label{e:alter1}
    \text{$\pi_V\colon \mathbb{R}^n \to V$ is injective } &\text{on a set of full $\mu$-measure for $\gamma_{n,m}$-a.e.\ $V \in \Sigma$} \\
    \label{e:alter2}
     &\! \! \! \! \! \! \! \!\text{$\gamma_{n,m}\bigl(\{V \in \Sigma \ | \  (\pi_V\mu)^a \neq 0\}\bigr) > 0$},
\end{align}
where $(\pi_V \mu)^a$ denotes the absolutely continuous part of $\pi_V\mu$ with respect to $\mathcal{H}^m$.
\end{theorem}

Theorem~\ref{t:mainthm2} connects the injectivity of typical orthogonal projections to the regularity of projected measures. Precisely, it shows that the \emph{singular projection property}
\[
\pi_V \mu \, \bot \, \mathcal{H}^m,
\]
ensures almost sure injectivity of typical projections. 

As a first application, if $E \subset \mathbb{R}^n$ is purely $m$-unrectifiable with $\mathcal{H}^m(E)<\infty$, then~\eqref{e:alter2} cannot hold for $\mu=\mathcal{H}^m\restr E$. Hence, condition~\eqref{e:alter1} necessarily follows. This extends the result in~\cite[Corollary 1.5]{bara} from the case $\mathcal{H}^m(E)=0$ to the entire class of purely $m$-unrectifiable sets of finite $\mathcal{H}^m$-measure. Furthermore, Theorem~\ref{t:mainthm2} is not tied to any assumption involving Hausdorff measures. Indeed, the singular projection property is formulated entirely in terms of the projected measures and therefore applies to a substantially broader class of measures. As discussed in Subsection~\ref{s:difficultiesproof}, this includes measures supported on sets that are not even $\mathcal{H}^m$-$\sigma$-finite, such as those arising from sets of finite integral geometric measure, for which no analogue of the previous injectivity results was available.

Combining our main result Theorem~\ref{t:mainthm1} with Theorem~\ref{t:mainthm2}, we obtain the following characterization of almost sure injectivity.

\begin{corollary}[Characterization of injectivity]
\label{c:mainthmn2}
Let $\mu$ be a finite Borel measure on $\mathbb{R}^n$ and let $\Sigma \subset \emph{Gr}(n,m)$ be a Borel set. If $\mu \, \bot \, \mathcal{H}^m \restr W$ for every affine $m$-plane $W \subset \mathbb{R}^n$, then $\pi_V$ is injective on a set of full $\mu$-measure for $\gamma_{n,m}$-a.e. $V \in \Sigma$ if and only if $\pi_V\mu \, \bot \, \mathcal{H}^m$ for $\gamma_{n,m}$-a.e. $V \in \Sigma$. 
\end{corollary} 

In particular, apart from the rigid case of $\mathcal{H}^m$-absolutely continuous measures living on a single affine $m$-plane, the singular projection property is the only mechanism responsible for almost sure injectivity of typical orthogonal projections.

From a geometric perspective, almost sure injectivity can be interpreted as a weak incidence property: for a typical $m$-plane $V$, the affine $(n-m)$-planes orthogonal to $V$ ``meet the measure in a single point'' for $(\pi_V\mu)$-almost every fiber, in the sense of disintegration. As we discuss in Subsection~\ref{ss:thirdalt}, this viewpoint is naturally connected with the study of radial projections of purely $m$-unrectifiable sets.

\subsection{Rectifiability via slicing}
\label{ss:rectcrit}
 In this subsection we state our main rectifiability criterion for measures. Although it is inspired by the ideas underlying the proof of Theorem~\ref{t:mainthm1}, it does not follow directly from that result, but rather from a combination of the techniques developed in its proof (see Section \ref{s:recslice}).

\begin{theorem}[Rectifiability via slicing]
   \label{t:rectcri}
    Let $\Sigma \subset \emph{Gr}(n,m)$ be a Borel set and let $(\mu_V)_{V \in \Sigma}$ be a family of finite Borel measures in $\mathbb{R}^n$ such that
    \begin{equation}
    \label{e:rectcri}
        \pi_V \mu_V \ll \mathcal{H}^m, \ \ \text{ for $V \in \Sigma$}.
    \end{equation}
    Assume that there exists a finite Borel measure $\eta$ in $\mathbb{R}^n$ such that 
    \[
    \mu_V \ll \eta, \ \ \text{ for $V \in \Sigma$}.
    \]
     Then, for $\gamma_{n,m}$-almost every $ V \in \Sigma$ and for every Borel measurable function $f \colon V \to V^\bot$, the restricted measure $\mu_V \restr \text{\emph{graph}}(f)$ is $m$-rectifiable.
\end{theorem}
 Clearly, the statement of Theorem~\ref{t:rectcri} is non-trivial only when $\mu_V(\text{graph}(f))>0$.
By the disintegration theorem, this is equivalent to requiring that
$(y,f(y))$ is a non-trivial atom of the conditional measure
$(\mu_V)_{V,y}$ for a set of $y\in V$ of positive
$\mathcal H^m$-measure. The assumption $\mu_V\ll\eta$ provides a common reference measure for
the family $(\mu_V)_{V\in\Sigma}$ and is the key ingredient allowing one
to combine the information coming from different projections.

Moreover, it is natural to ask whether the hypothesis $\gamma_{n,m}(\Sigma)>0$
is optimal. Although we do not settle this question completely, in
Section~\ref{s:recslice} we construct a measure $\mu$ satisfying
\eqref{e:mainthm1.1} and \eqref{e:rectcri} for a dense family of
projections, while containing no nontrivial rectifiable part. Thus, at
least, requiring the slicing condition only along a dense family of
projections is not sufficient to guarantee rectifiability.

\subsection{Rectifiability of currents from atomic slices}
\label{s:conncurrentsintr}

A second main contribution of the paper concerns the rectifiability of a
$m$-current $T$ in $\mathbb{R}^n$. We show that the atomicity of suitable slices forces
rectifiability without any assumption on the mass of the boundary.
This principle is established in two complementary settings. First, we
consider classical Federer-Fleming currents of finite mass satisfying
a discrete pushforward constraint, but \emph{not assumed} to be metric
currents. Second, we consider Euclidean metric currents whose slices
are atomic along a family of projections of positive Grassmannian
measure. In both cases, rectifiability follows without assuming
normality, namely, without assuming that the boundary $\partial T$ has finite mass 
\[
\mathbf M(\partial T)<\infty.
\]

These results are closely related to the classical
\emph{rectifiable slices theorem}, proved by White for flat chains in
Euclidean space~\cite{whi1} and, independently, by Ambrosio and
Kirchheim for metric currents~\cite{ak1}. Roughly speaking, these
theorems characterize the rectifiability of a current $T$ in terms
of the atomicity of its slices
$\langle T,\pi,y\rangle$.

There is, however, an important distinction with White's theorem. White works within the class of finite-mass flat chains: flatness is therefore part of the initial structure, and his proof exploits both approximation by polyhedral chains in the flat norm and the deformation theorem. Within this class, atomicity of the slices associated with the coordinate projections is enough to imply rectifiability. In Theorem~\ref{t:Grect}, by contrast, no flat-chain structure is assumed. We start from an arbitrary finite-mass Federer-Fleming current satisfying only the intrinsic discrete push-forward condition \eqref{e:G-valuedintr}, and prove both its rectifiability and, as a consequence, its flatness. Thus, in our setting, belonging to the class to which White's theorem applies is itself part of the conclusion.

The argument of Ambrosio and Kirchheim follows a different route and
makes essential use of the boundary operator. Under the assumption
$\mathbf M(\partial T)<\infty$, they prove that, for every Lipschitz
map $\pi:\mathbb R^n\to\mathbb R^m$, the slicing map
\begin{equation}
\label{e:jerrardbv}
y\longmapsto \langle T,\pi,y\rangle
\end{equation}
has metric bounded variation when the space of $0$-currents is
endowed with the flat distance; see~\cite[Section~7]{ak1}. Within the
existing metric-current theory, the implication
\[
\text{atomic slices}
\quad\Longrightarrow\quad
\text{rectifiability}
\]
is therefore obtained through a quantitative estimate that requires
finite boundary mass.

In the metric case, the possibility of removing this assumption is closely connected with
the \emph{flat chain conjecture} introduced by Ambrosio and
Kirchheim~\cite{ak1}. The conjecture is presently known for
$m=1$ and in the top-dimensional case $m=n$; see~\cite{marc}.
Indeed, if the Federer-Fleming current $\widetilde T$ associated
with every $T\in\mathbf M_m(\mathbb R^n)$ by~\cite[Theorem~11.1]{ak1} were known to be a real flat
$m$-chain, then White's rectifiable slices theorem would immediately
show that atomicity of the slices alone implies the rectifiability of
$T$.

Theorem~\ref{t:rectcri} provides a different approach. Its key feature
is that it separates two distinct issues:

\begin{enumerate}
    \item the rectifiability of the mass measure of the current;
    \item the identification of the polar of the current with the
    approximate tangent $m$-planes of that measure.
\end{enumerate}

The first issue can be addressed through the measure-theoretic slicing
criterion of Theorem~\ref{t:rectcri}. The second requires an additional
argument and becomes especially delicate for classical
Federer-Fleming currents, where no locality axiom~\cite[Definition 3.1 (iii)]{ak1} is available. In the
presence of a discrete coefficient constraint, however, this
separation leads to a rectifiability theorem formulated entirely
within the classical theory of currents.

\subsubsection{Integer-coefficient Federer-Fleming currents}
For a classical Federer-Fleming current of finite mass, the natural intrinsic way to express that it has integer coefficients is to require that all its localized pushforwards under orthogonal projections have integer multiplicity. This condition is formulated directly in terms of the action of the current and does not presuppose any further structure. It is precisely the condition used by Ambrosio and Kirchheim in their definition of integer-rectifiable currents~\cite{ak1}; moreover, once rectifiability is known, it is equivalent to the usual requirement that the multiplicity be integer-valued. This leads to the following definition.

\begin{definition}
\label{d:Gconstrained}
A classical Federer-Fleming $m$-current $T$ of finite mass in $\mathbb R^n$ is an integer-coefficient current if, for every Borel set
$B\subset\mathbb R^n$ and every
$V\in\emph{Gr}(n,m)$, there exists a function
\[
\theta_{\pi_V,B}\in L^1(\mathbb R^m;\mathbb Z)
\]
such that
\begin{equation}
\label{e:G-valuedintr}
\pi_{V\#}(T\restr B)
=
\llbracket \theta_{\pi_V,B}\rrbracket .
\end{equation}
\end{definition}

Here and in what follows, the right-hand side in \eqref{e:G-valuedintr} denotes the
top-dimensional current in $\mathbb R^m$ induced by the function
$\theta_{\pi_V,B}$.

In Euclidean space, condition~\eqref{e:G-valuedintr} is precisely the
integrality condition of Ambrosio and
Kirchheim~\cite[Definition~4.2]{ak1}, restricted to the family of
orthogonal projections. As already mentioned, it also agrees with the usual notion of
integer coefficients for rectifiable currents. More precisely, if
\[
T=\llbracket R,\tau,\theta\rrbracket
\]
is an $m$-dimensional rectifiable current, then $T$ satisfies
\eqref{e:G-valuedintr} if and only if
\[
\theta(x)\in\mathbb Z
\qquad
\text{for }\mathcal H^m\text{-a.e. }x\in R.
\]

Our first result shows that the discrete pushforward condition alone
forces rectifiability.

\begin{theorem}
\label{t:Grect}
Let $T$ be an integer-coefficient $m$-current of finite mass in
$\mathbb R^n$. Then $T$ is rectifiable. More precisely, there
exist a countably $m$-rectifiable set
$R\subset\mathbb R^n$, an orientation $\tau$, and a multiplicity
$\theta$ such that
\[
T=\llbracket R,\tau,\theta\rrbracket .
\]
In particular, $T$ is a real flat chain.
\end{theorem}

The significance of Theorem~\ref{t:Grect} is that $T$ is merely
assumed to be a classical Federer-Fleming current of finite mass.
It is not assumed to be a metric current, no locality axiom is
available, and no condition is imposed on
$\mathbf M(\partial T)$. In particular, flatness is not assumed at any stage, but follows from the rectifiability conclusion.

Theorem~\ref{t:Grect} should be compared with
\cite[Theorem~8.8]{ak1}. The latter concerns normal metric currents
and therefore assumes, in particular, finite boundary mass. This
assumption is used to obtain the $BV$-estimate for the slicing map
in~\eqref{e:jerrardbv}, from which the rectifiability of the mass
measure is deduced. The locality axiom then yields the parametric
representation of the current; see~\cite[Theorem~4.5]{ak1}. 

By contrast, Theorem~\ref{t:Grect} is proved directly in the
Federer-Fleming setting. This creates an additional difficulty:
rectifiability of the mass measure $\|T\|$ does not by itself imply
rectifiability of $T$. Indeed, even if $\|T\|$ is concentrated on
a countably $m$-rectifiable set, the polar of $T$ need not be
aligned with the approximate tangent $m$-planes of $\|T\|$.
Proving this alignment is an essential part of the argument.

The proof of Theorem~\ref{t:Grect} proceeds in two steps. First,
condition~\eqref{e:G-valuedintr} gives rise, for each
$V\in\text{Gr}(n,m)$, to a positive finite measure $\mu_V:=|T\restr \pi_V|$ satisfying the assumptions of Theorem~\ref{t:rectcri}. The discreteness
of $\mathbb Z$ forces the conditional measures associated with the
disintegration of $\mu_V$ through $\pi_V\colon \mathbb R^n\to V$ to be atomic. Theorem~\ref{t:rectcri} then yields the
$m$-rectifiability of the mass measure $\|T\|$.

The second step identifies the polar of $T$ with the approximate
tangent planes of $\|T\|$. Here we exploit once more the discreteness of the coefficient
group. Assume that $\tau' \colon \mathbb{R}^n \to \Lambda_m(\mathbb{R}^n)$ is a Borel orientation of $\|T\|$, namely, that $\tau'(x)$ is a simple and unitary $m$-vector spanning the tangent space of $\|T\|$ at $\|T\|$-almost every point $x \in \mathbb{R}^n$. At a point where both the polar $\tau$ of $T$ and $\tau'$ are
defined, the localized pushforward condition forces, for almost every
projection $V$, the quotient
\begin{equation}
\label{e:realanalytic}
V\longmapsto
\frac{\langle\tau,d\pi_V\rangle}
     {\langle\tau',d\pi_V\rangle}
\end{equation}
to take values in a discrete, and therefore countable, subset of
$\mathbb R$. On the other hand, the map in
\eqref{e:realanalytic} is real analytic on its domain of definition.
Such a range restriction is impossible unless $\tau$ and $\tau'$
are aligned. Consequently, $\tau=\pm\tau'$, which gives the required tangentiality and completes the proof of
rectifiability; see Section~\ref{s:recslice}.

Thus, the discrete pushforward constraint performs two distinct
geometric tasks: it forces atomicity of the relevant conditional
measures, and it rigidly aligns the polar of the current with the
geometry of its mass measure.

\subsubsection{Metric currents without finite boundary mass}

We next turn to metric currents $T\in\mathbf M_m(\mathbb R^n)$. In this setting, once the rectifiability of the mass measure
$\|T\|$ is known, the locality axiom and the parametric
representation theorem of Ambrosio and Kirchheim provide the remaining
geometric structure. The main problem is therefore to deduce the
rectifiability of $\|T\|$ directly from atomic slicing information,
without assuming that $T$ is normal.

A preliminary issue is that the required slices are not directly
provided by the classical Ambrosio-Kirchheim slicing theory, since no
assumption is made on $\mathbf M(\partial T)$. For
$V\in\text{Gr}(n,m)$, we denote by
\[
\bigl(\langle T,\pi_V,y\rangle\bigr)_{y\in V}
\]
the family of signed measures obtained by disintegrating, with respect
to $\pi_V \colon \mathbb R^n\to V$, the coefficient measure $T \restr \pi_V$, namely, the restriction of $T$ to the projection
$\pi_V$. We refer to Section~\ref{s:recslice} for the precise
construction. Whenever the classical Ambrosio-Kirchheim slicing
theory applies, these measures coincide with the usual slices.

The next theorem shows that atomicity of these slices is, by itself,
sufficient to force rectifiability.

\begin{theorem}
\label{t:rectmetric}
Let
$\Sigma\subset\emph{Gr}(n,m)$ be a Borel set satisfying $\gamma_{n,m}(\Sigma)>0$, and let
$T\in\mathbf M_m(\mathbb R^n)$ be a metric current. Assume that,
for every $V\in\Sigma$,
\begin{equation}
\label{e:rectmetric}
\langle T,\pi_V,y\rangle
\text{ is atomic for }
\mathcal H^m\text{-almost every }y\in V.
\end{equation}
Then $T$ is a rectifiable metric current. In particular, the
associated Federer-Fleming current $\widetilde T$ is a real flat
chain.
\end{theorem}

Two features of Theorem~\ref{t:rectmetric} are worth emphasizing.
First, atomicity is required only along a family of projections of
positive Grassmannian measure. Neither all projections nor almost
every projection in the whole Grassmannian are needed. Second, no
assumption is imposed on the mass of the boundary. Thus, once
atomicity of the slices is given, the boundary-mass hypothesis used in
the classical metric-current argument is no longer required.

The proof combines Theorem~\ref{t:rectcri} with the disintegration
construction described above. The slicing criterion first yields the
rectifiability of the mass measure $\|T\|$. The result in~\cite[Theorem~4.5]{ak1} then implies that the associated
Federer-Fleming current $\widetilde T$ is $m$-rectifiable.

Taken together, Theorems~\ref{t:Grect} and~\ref{t:rectmetric} are somewhat surprising: they show that purely slicing information can force the rectifiability of a current, even in the absence of any control on the mass of its boundary and, in the case of Theorem~\ref{t:Grect}, without assuming any of the additional structure available for metric currents.

Finally, the ideas developed here suggest a possible extension of the
classical Federer-Fleming closure and compactness theory beyond
families of currents with equibounded boundary mass. In forthcoming
work, we make use of the above slicing criterion to prove closure and
compactness results for rectifiable currents with coefficients in a
group, under uniform bounds on suitable energies associated with their
boundaries.

\subsection{Obstructions to the existing approaches}
\label{s:difficultiesproof}
 
The main conceptual difficulty behind Theorem~\ref{t:mainthm1} lies in the generality of hypothesis~\eqref{e:mainthm1.1} and extends far beyond the $\sigma$-finite setting with respect to the Hausdorff measure. Indeed, when translated into a statement about supports, hypothesis~\eqref{e:mainthm1.1} shifts the focus to sets whose typical affine $(n-m)$-dimensional slices are merely countable. Such sets are, in general, far from being $\mathcal H^m$-sigma-finite. This phenomenon naturally appears in the framework of the integral geometric measure \cite[2.10.5]{fed1}. First, it follows from a result of Mattila \cite{mat8} that the class of negligible sets with respect to the $m$-dimensional integral geometric measure $\mathcal I^m$ on $\mathbb{R}^n$ is not preserved under Lipschitz mappings. Consequently, there exist Borel sets $E$ whose orthogonal projections are $\mathcal H^m$-negligible in almost every direction, namely, $\mathcal{I}^m(E)=0$, and which cannot be $\mathcal H^m$-$\sigma$-finite. Such planar sets $E$ can be constructed to satisfy $0 < \mathcal{H}^{\varphi}(E) < \infty$ where $\mathcal{H}^{\varphi}$ denotes the generalized Hausdorff measures associated to a positive and weakly increasing gauge function $\varphi$. As shown by Joyce and M\"orters in \cite{joymor}, $\varphi$ can be chosen in such a way that 
\[
\int_0^1 \frac{\varphi(r)^\alpha}{r^{\alpha+1}} \, dr <\infty, \ \ \text{for every $\alpha >1$.}
\]
Second, Mattila constructs in \cite{mat2} a purely $1$-unrectifiable compact set $K \subset \mathbb{R}^2$ such that
\begin{equation}
\label{e:intgeoce}
K\subset\mathbb R^2,
\qquad
0<\mathcal I^1(K)<\infty.
\end{equation}
By the classical Besicovitch-Federer projection theorem, such a set cannot be $\mathcal H^1$-$\sigma$-finite, due to the fact that $\mathcal{I}^1(K)>0$. On the other hand, since $\mathcal{I}^1(K) < \infty$ implies that  $\ell \cap K$ is finite for a typical line $\ell \subset \mathbb{R}^2$, one readily checks that any measures $\mu$ supported in $K$ must satisfies condition \eqref{e:mainthm1.1}. This example also explains why the conclusion of Theorem~\ref{t:mainthm1} cannot, in general, be strengthened to the null projection property of the classical Besicovitch-Federer theorem for sets of finite $\mathcal H^m$-measure. Another important feature is that the class of measures $\mu$ satisfying~\eqref{e:mainthm1.1} lies genuinely outside the framework of generalized Hausdorff measures. Indeed, regardless of the choice of the gauge function~$\varphi$, the class of $\mathcal H^\varphi$-null sets is clearly stable under Lipschitz mappings, whereas this property generally fails for $\mu$.

More recently, the problem of characterizing pure unrectifiabilty within the class of sets with finite integral geometric measure via orthogonal projections has been solved in \cite{tas1}. Indeed, any purely $m$-unrectifiable set $E \subset \mathbb{R}^n$ with $\mathcal{I}^m(E) < \infty$ admits the following characterization: if $\mu$ is any finite Borel measure concentrated on such a set, then it satisfies the first condition in~\eqref{e:mainthm1.2}.

The proof in~\cite{tas1} relies crucially on the fact that finite integral geometric measure forces almost every affine $(n-m)$-dimensional plane to intersect $E \subset \mathbb{R}^n$ in a finite, and hence discrete, set. One may nevertheless ask whether this argument can be adapted to prove the first condition of~\eqref{e:mainthm1.2}. The most natural strategy would be to proceed as follows. For almost every $m$-plane V, define
\[
E_V
:=
\bigcup_{y\in V}
\{\text{atoms of }\mu_{V,y}\},
\]
and then try to combine the sets $E_V$ into a single set $E$ supporting $\mu$, so that the machinery of~\cite{tas1} could be applied. This strategy breaks down for three independent reasons.
\vspace{3mm}

 {(1)} Hypothesis~\eqref{e:mainthm1.1} only requires the conditional measures $\mu_{V,y}$ to be purely atomic for $\mathcal H^m$-almost every $y\in V$. No information is available on the conditional measures corresponding to fibers lying above the singular part of $\pi_V\mu$. Consequently $\mu(\mathbb R^n\setminus E_V)$
may be strictly positive, and therefore $\mu$ is, in general, not concentrated on any union of the sets $E_V$.
\vspace{2mm}

 {(2)} Even if we assume from the beginning that $\pi_V\mu\ll\mathcal H^m$ for almost every $V$, so that $\mu(\mathbb R^n\setminus E_V)=0$, by construction we can infer
\[
\text{card}
\bigl(E_V\cap\pi_V^{-1}(y)\bigr)
\le\aleph_0,
\qquad
\text{for $\mathcal H^m$-almost every }y\in V,
\]
while nothing is known about the slices of $E_V$ with respect to any other projection. Although $\mu\bigl(\mathbb R^n\setminus(E_V\cap E_{V'})\bigr)=0$
for every pair of $m$-planes $V,V'$, this property is only stable under countable intersections. Hence there is no evident way, such as an essential-union over the family $(E_V)$, of constructing a common set $E$ with controlled slicing behaviour in almost every directions. In fact, it would not even suffice to find a set $E$ satisfying $\mu(E\setminus E_V)=0$ for almost every $V$, since it is not known whether $\mathcal H^m\bigl(\pi_V(E\setminus E_V)\bigr)=0$. Indeed, this is generally false; by inspecting the proof of \cite[Theorem 3.1]{tas1}, it turns out that the compact set $K \subset \mathbb{R}^2$ in \eqref{e:intgeoce} must decompose as $K = K_\ell \cup (K \setminus K_\ell)$ where $\mathcal{I}^1(K_\ell)=0$ and $\mathcal{H}^1(\pi_\ell(K_V))>0$, for almost every line $\ell$ through the origin. 
\vspace{2mm}

 {(3)} Even under the considerably stronger assumption that there exists a single set $E$ satisfying $E=E_V$ for almost every $V$, condition~\eqref{e:mainthm1.1} still allows the typical slice $E\cap\pi_V^{-1}(y)$ to be countably infinite and even dense in the corresponding affine $(n-m)$-plane. This is precisely where the third alternative in~\cite[Theorem~3.3.4]{fed1} ceases to apply, as discreteness of typical slices is no longer available. A more detailed discussion of this obstruction is postponed to Subsection~\ref{ss:thirdalt}.

\vspace{1mm}
These obstructions explain why the proof of
Theorem~\ref{t:mainthm1} requires a substantially different approach (see Subsection \ref{ss:ideaproof}).



\subsection{On the necessity of Besicovitch's third alternative}
\label{ss:thirdalt}
The so-called Besicovitch's three key alternatives~\cite[Theorem~3.3.4]{fed1} constitute a key ingredient in the proof of the Besicovitch-Federer projection theorem. Informally, these alternatives assert that if the asymptotic behaviour of certain maximal conical densities of $\mu$ around an $m$-plane $V$ at a point $x$ is finite and non-zero, then the following geometric condition must hold:
\begin{equation}
\label{e:alternative(iii)}
x\text{ is an accumulation point of }(E\setminus\{x\})\cap(V^\bot+x).
\end{equation}
When $\mathcal I^m(E)<\infty$, condition~\eqref{e:alternative(iii)} is incompatible with the fact that the typical affine $(n-m)$-dimensional slice of $E$ consists of finitely many points. This incompatibility is precisely what allows one to rule out~\eqref{e:alternative(iii)}, a crucial step in the argument developed in~\cite{tas1}.

A central difficulty in extending this argument under the hypotheses of Theorem~\ref{t:mainthm1} is the lack of a set $E$ supporting $\mu$ with sufficiently good slicing properties, as explained in points~(1)-(3) of Subsection~\ref{s:difficultiesproof}. 

Mattila observed in~\cite[Subsection~4.3]{mat4} that, even for sets of
finite Hausdorff measure, it is not known whether condition
\eqref{e:alternative(iii)} is genuinely necessary. Indeed, an old
problem posed by Marstrand~\cite{mar} (originally formulated for planar sets) asks whether a purely
$m$-unrectifiable set $E\subset\mathbb R^n$ with
$\mathcal H^m(E)<\infty$ satisfies, for $\mathcal H^m$-almost every
$x\in E$,
\[
E\cap (V^\bot+x)=\{x\},
\qquad
\text{for $\gamma_{n,m}$-a.e. }V\in\text{Gr}(n,m).
\]
In other words, for almost every point $x\in E$, a typical
$(n-m)$-dimensional plane through $x$ intersects $E$ only at $x$.

This problem remains widely open; a complete answer is known only for
planar self-similar sets satisfying the open set condition~\cite{simsol}.
If the answer were positive, then condition
\eqref{e:alternative(iii)} would merely be an artifact of the proof
and, in particular, would not be needed in Besicovitch's three key
alternatives.

It is worth pointing out that, after an application of Fubini's
theorem, Marstrand's problem can be reformulated in terms of
projections. Namely, it asks whether for
$\gamma_{n,m}$-almost every $V\in\text{Gr}(n,m)$,
\[
E\cap \pi_V^{-1}(\pi_V(x)) = \{x\},
\qquad
\text{for $\mathcal H^m$-a.e. }x\in E.
\]
Thus, for almost every projection, the fibers of $\pi_V$ are
essentially singletons on $E$.

This comparison also illustrates why assumption~\eqref{e:mainthm1.1}
is substantially weaker than a geometric assumption on the support of
the measure. Indeed, when applied to the measure
$\mathcal H^m\restr E$, the injective projection property established in Theorem~\ref{t:mainthm2} yields that, for
almost every $V$, there exists a graph $G_V$ over $V$ carrying the whole
measure. However, this graph is allowed to depend on the $m$-plane
$V$. A proof that these sets $G_V$ could be
chosen independently of $V$ would immediately imply
the conclusion of Marstrand's conjecture. A detailed analysis of this connection, as well as its implication to Marstrand's problem, will be the subject of forthcoming work.

Finally, one of the advantages of the general framework developed in
this paper is that condition~\eqref{e:alternative(iii)} can be
eliminated altogether. More precisely,
conditions~\eqref{e:(1)}-\eqref{e:(2)} in the proof of Lemma~\ref{l:keyalternative} replace the three alternatives by a
dichotomy which rules out \eqref{e:alternative(iii)}. We also refer the reader to Subsection~\ref{ss:ideaproof} for a heuristic discussion of the proof.

\subsection{Transversal family of maps}

The characterization of pure unrectifiability given by Theorem~\ref{t:mainthm1} is in fact established in this paper for any separable locally compact metric space $(X,d)$ and for a transversal family of maps $\Pi_\lambda \colon X \to \mathbb{R}^m$, parametrized by an open set $\Lambda \subset \mathbb{R}^l$ with $m \le l$. We refer to Theorem~\ref{p:finite1} and Theorem~\ref{t:final2} for the analogues of Theorem~\ref{t:mainthm2} and the first condition of~\eqref{e:mainthm1.1} of~Theorem~\ref{t:mainthm1}, respectively.

 The family of orthogonal projections may be locally regarded as a transversal family of maps (see, for instance, \cite[Appendix A]{tas1}). In particular, any statement holding for $\mathcal{L}^\ell$-almost every $\lambda \in \Lambda$ translates into the corresponding statement for $\gamma_{n,m}$-almost every $V \in \mathrm{Gr}(n,m)$.

The notion of transversality considered here is that introduced by Peres and Schlag in~\cite{per}, and we refer to Subsection~\ref{ss:transversal} for a precise definition. Examples include, but are not limited to, horizontal projections in Heisenberg groups~\cite{hov1}, closest-point projections~\cite{bal}, geodesic-based projections on Riemannian surfaces or manifolds~\cite{hov,almtas}, and hyperbolic spaces~\cite{ise}.

\subsection{Idea of the proof}
\label{ss:ideaproof}
For clarity of exposition, we shall present the argument in the setting of orthogonal projections, where the geometric structure is more transparent.

The main line of argument in the proof of our principal result is to first establish the validity of the two alternatives \eqref{e:alter1} and \eqref{e:alter2}, and then to show that whenever $\mu$ fails to satisfy the singular projection property, it must necessarily contain a nontrivial rectifiable component. The fundamental quantity to be controlled is the \emph{maximal conical $m$-density truncated at scale} $\delta>0$, namely
\begin{equation}
\label{e:maxcon}
\sup_{0<\rho<\delta} \frac{\mu(X(x,V,s,\rho))}{(s\rho)^m},
\end{equation}
where $X(x,V,s,\rho)$ denotes the intersection of the ball $B_\rho(x)$ with the cone of aperture $s$ around the $m$-plane $V$ centered at $x$.

When the measure $\mu$ is supported on a set whose intersection with a generic $(n-m)$-plane consists of finitely many points, the classical argument in the proof of the Besicovitch-Federer projection theorem implies that, as $s \to 0^+$, the quantity in \eqref{e:maxcon} can asymptotically attain only two possible values: either $0$ for some $\delta>0$, or $+\infty$ for every $\delta>0$. The delicate point in our setting is that there is no measure-supporting set exhibiting good behavior under slicing; and even if such a set were to exist, it would not, in any case, possess the property of finite slicing.

The novel feature of our approach lies in the construction of a disintegration of $\mu$ along all possible $(n-m)$-planes through $x$, that is,
\begin{equation}
\label{e:disidea}
\mu = \int \mathring{\mu}_{V,x}\, d\gamma(V),
\end{equation}
for measures $\mathring{\mu}_{V,x}$ supported on $V^\bot$, and for a suitable finite $m$-dimensional Borel measure $\gamma$ on the Grassmannian. In codimension one, such a disintegration can be realized directly via the \emph{radial projection map}
\[
r_x \colon \mathbb{R}^n \setminus \{x\} \to \mathbb{S}^{n-1}, \qquad r_x(x') = \frac{x'-x}{|x'-x|}.
\]
In this case, the relation between the disintegration \eqref{e:disidea} and the maximal conical density in \eqref{e:maxcon} becomes transparent, since cones of aperture $s$ can be identified with the preimages of spherical caps of radius $s$ in $\mathbb{S}^{n-1}$, namely
\[
X(x,V_\xi,s) \simeq r_x^{-1}(\mathbb{B}_s(\xi)), \qquad V_\xi = \xi^\perp.
\]

In contrast, in the general higher-codimension case one cannot hope to realize the disintegration \eqref{e:maxcon} through a single map. Here, additional work is required to account for the geometric complexity of the Grassmannian. We show that, in a neighborhood of any point $x$, the space can be covered by finitely many open sets $U_i$, with constants depending only on the codimension $n-m$, such that on each $U_i$ there exists a map
\[
\Phi_i \colon U_i \to M_i
\]
taking values in an $m$-dimensional submanifold $M_i$ of the Grassmannian, which plays, in a suitable sense, the role of a radial projection: cones of aperture $s$ at $x \in U_i$ correspond to the preimages under $\Phi_i$ of balls of radius $s$ centered at $\Phi_i(x)$ in $M_i$. This construction relies crucially on transversality, and is established for a generic transversal family in Proposition~\ref{p:equivcone}. The disintegration \eqref{e:disidea} is then obtained in Proposition~\ref{p:fcv1} by means of all such maps $\Phi_i$.
 
At this stage, the key observation is of a geometric nature. 
The truncated cone $X(x,V,s,\rho)$ can be compared, through simple inclusions, 
with the cylindrical sector obtained as the preimage under $\pi_V$ of the ball $B_{\rho s}(\pi_V(x)) \subset V$. 
This observation allows one to relate the auxiliary measures $\mathring{\mu}_{V,x}$ in \eqref{e:disidea} 
to the measures $\mu_{V,y}$ arising from the disintegration of $\mu$ with respect to orthogonal projection $\pi_V$. 
It follows that these two families of measures are essentially equivalent, 
up to a Jacobian factor given by the $m$-th power of the distance from the center $x$.
 More precisely, we show in Proposition~\ref{p:disretr1} that for almost every $m$-plane $V$ there exists a \emph{locally finite} measure $\mu_{V,x}$ on $(V^\perp +x) \setminus \{x\}$ satisfying
\begin{equation}\label{e:intrjac}
\mathring{\mu}_{V,x} \simeq \int |x'-x|^m \, d\mu_{V,x}.
\end{equation}
The relation between $\mu_{V,x}$ and $\mu_{V,y}$ is further clarified in Proposition~\ref{p:finite}, where it is shown, in brief, that if $\mu$ admits absolutely continuous projections, then not only is $\mu_{V,x}$ finite, but in fact the identity $\mu_{V,x} = \mu_{V,\pi_V(x)}$ holds for almost every $m$-plane $V$.

Once the relation~\eqref{e:intrjac} has been established, the proof of the injectivity property in Theorem~\ref{t:mainthm2} relies on the observation that, if $x$ is a point such that $\pi_V(x)$ has infinite $m$-dimensional density for the projected measure $\pi_V\mu$, and if the measure $\mu_{V,x}$ exists and satisfies~\eqref{e:intrjac}, then no other point $x' \in V^\bot +x$ can enjoy the same property. The rigorous proof of this idea is contained in Theorem~\ref{p:finite1}.

\medskip

The role of relation~\eqref{e:intrjac} is also crucial in the proof of Theorem~\ref{t:mainthm1}. In order to keep the exposition as transparent as possible, we discuss the codimension-one case. Our strategy proceeds as follows. The measure $\gamma$ can be decomposed into its absolutely continuous and singular parts with respect to $\mathcal{H}^{n-1}$, denoted by $\gamma^a$ and $\gamma^s$, respectively. This induces the corresponding orthogonal decomposition of $\mu$ in a neighborhood of $x$:
\[
\mu = \mu^a + \mu^s := \int \mathring{\mu}_{V,x} \, d\gamma^a(V) + \int \mathring{\mu}_{V,x} \, d\gamma^s(V).
\]

While the maximal conical $(n-1)$-density associated with $\mu^s$ can be estimated analogously to~\cite[Theorem 3.3.4]{fed1} by invoking the density properties of outer measures in Lemma~\ref{l:density}, one might be tempted to estimate the corresponding quantity for $\mu^a$ by exploiting the equivalence~\eqref{e:intrjac}. Indeed, the presence of the Jacobian factor $|x-x'|^{n-1}$ naturally suggests estimating~\eqref{e:maxcon} by means of the quantity
\[
s^{1-n} \int_{\mathbb{B}_s(V_\xi)} \mu_{V,x}(B_\rho(x)) \, d\gamma^a(V),
\]
and then appealing to the Lebesgue differentiation theorem to deduce that, as $s \to 0$, the above expression is asymptotically bounded from above by $\mu_{V,x}(B_\rho(x))$.

However, the map $V \longmapsto \mu_{V,x}(\mathbb{R}^n)$ is not, in general, integrable. Indeed, such integrability would fail even under the stronger assumption $\pi_V\mu \ll \mathcal{H}^m$ for almost every $V$, since it would imply higher integrability of the projected measures $\pi_V\mu$ (see, for instance,~\cite{chatol}). The technical argument needed to circumvent this difficulty is developed in Lemma~\ref{l:keyalternative}.

\medskip

This approach is particularly powerful, as it allows us to reduce the analysis of the three Besicovitch alternatives to a single novel one. Specifically, it suffices to consider the asymptotic behavior
\begin{equation}
\label{e:singlealt}
\lim_{\delta \to 0} \limsup_{s \to 0} \sup_{0<\rho<\delta} 
\frac{\mu(X(x,V,s,\rho))}{(s\rho)^m} = 0,
\end{equation}
which is precisely the content of Lemma~\ref{l:keyalternative}. Compared with one of the key alternatives in the classical setting, the novelty here lies in the fact that the parameter $\delta$ is also sent to zero. In the original formulation, the corresponding vanishing condition occurs at a fixed finite scale $\delta>0$, which may depend on the point. That stronger, finite-scale condition allows one to deduce that if it holds at almost every point of a purely unrectifiable set, then such a set must necessarily be $\mu$-negligible.

In our case, the vanishing of $\delta$ forces us to restrict attention to purely unrectifiable sets which, in addition, arise as graphs of measurable functions 
$g \colon V \to V^\perp$ (see Lemma~\ref{l:graphical}). 
This graphical condition is entirely consistent with the atomicity of the conditional measures $\mu_{V,y}$. 
By a measurable selection principle, if $\mu$ charges a purely unrectifiable set, then the set necessarily contains a nontrivial measurable graph. 
Such a graph must, however, be $\mu$-negligible in light of the preceding argument. This is established in Theorem~\ref{t:final}.

\subsection{Organization of the paper}

After this introduction, the paper is organized into four main sections and a final appendix.
Section \ref{s:pr} fixes notation and collects several preliminary results, including the construction of the maps $\Phi$ in Proposition \ref{p:equivcone}, which will be used to disintegrate $\mu$ around any point of the space. In Section \ref{s:rp}, we establish a number of technical propositions describing the relation between the families of measures obtained by disintegrating $\mu$ with respect to all $\Phi$ and those arising from disintegration with respect to our generalized projections. The section concludes with the proof of Theorem \ref{p:finite1}, which establishes the injectivity of the projection maps.
Section \ref{s:rvs} is devoted to the derivation of the single alternative stated in \eqref{e:singlealt}, as well as to the proof of the first condition of~\eqref{e:mainthm1.1} of Theorem \ref{t:mainthm1} in the setting of separable locally compact metric spaces and for generalized families of projections (Theorems \ref{t:final}, \ref{t:final2}, and \ref{t:mainthm1.}). The final section \ref{s:recslice} presents the rectifiability via slicing for general Radon measures as well as the proofs of the Theorems \ref{t:Grect} and \ref{t:rectmetric}. Two appendices complete the paper: Appendix A contains the proof of Proposition \ref{p:equivcone}, and Appendix B establishes certain measurability properties used throughout.

\section{Notation and Preliminaries}
\label{s:pr}
We work in a separable locally compact metric space $(X,d)$. 
Throughout the paper, we let $m$ and $l$ be positive integers satisfying $m \le l$. 
The symbol $\Lambda$ denotes an open and bounded subset of $\mathbb{R}^l$. We denote by $B_\rho(x)$, $B^m_\rho(y)$, and $\mathbb{B}_s(\lambda)$ the open balls in $X$, $\mathbb{R}^m$, and $\mathbb{R}^\ell$, respectively.

For $A \subset X$, $\mathcal{H}^m(A)$ denotes the $m$-dimensional Hausdorff measure. 
A Borel set $R \subset X$ is said to be \emph{countably $m$-rectifiable} if 
\[
R = \bigcup_{i=1}^\infty f_i(E_i),
\]
for Lipschitz functions $f_i \colon E_i \to X$ defined on some bounded set $E_i \subset \mathbb{R}^m$.
A finite Borel measure $\mu$ in $X$ is \emph{$m$-rectifiable} if it is absolutely continuous with respect to $\mathcal{H}^m$ and $\mu(X\setminus R)=0$ for some countable $m$-rectifiable set $R \subset X$. Finally, a set $E \subset X$ is \emph{purely $m$-unrectifiable} if 
\[
\mathcal{H}^m(E \cap R)=0, \ \ \text{ for every countably $m$-rectifiable set $R$.}
\]

Given a Borel measurable map $\Pi \colon X \to \mathbb{R}^m$ and a finite Borel measure $\mu$ on $X$ (see~\cite[Theorem 5.3.1]{ags}), 
we can consider a family of probability measures $(\mu_y)_{y \in \mathbb{R}^m}$ obtained by disintegrating $\mu$ with respect to the map $\Pi$. Presicely,
\begin{align}
\label{e:disintegration00}
&y \longmapsto \mu_y(B) \text{ is Borel measurable}, 
\qquad \text{for every Borel set } B \subset X\\
\label{e:disintegration0}
&\mu(B) = \int_{\mathbb{R}^m} \mu_y(B) \, d(\Pi\mu)(y), 
\qquad \text{for every Borel set } B \subset X,
\end{align}
where $\Pi\mu$ denotes the pushforward of $\mu$ under $\Pi$, namely,
\[
\Pi\mu(B):= \mu(\Pi^{-1}(B)), \ \ \text{ for every Borel set $B \subset \mathbb{R}^m$}.
\]

Recall that, by the standard properties of disintegration, 
the family $(\mu_y)_{y \in \mathbb{R}^m}$ is $\Pi\mu$-a.e.\ uniquely determined, and that
\[
\mu_y \restr \Pi^{-1}(y) = \mu_y, 
\qquad \text{for $\Pi\mu$-a.e.\ } y \in \mathbb{R}^m,
\]
where $\mu_y \restr \Pi^{-1}(y)$ denotes the restriction of $\mu_y$ to the fiber $\Pi^{-1}(y)$. 
We will use the compact notation $\mu = \mu_y \otimes \Pi\mu$ to denote the disintegration formula~\eqref{e:disintegration0}.

\subsection{A density result for outer measures}

An \emph{outer measure} $\Psi$ on $X$ is a map 
\[
\Psi \colon 2^{X} \to [0,\infty]
\]
satisfying:
\begin{itemize}
    \item $\Psi(\emptyset) = 0$,
    \item $\Psi\!\left( \bigcup_{i=1}^\infty E_i \right) 
    \le \sum_{i=1}^\infty \Psi(E_i)$ 
    whenever $E_i \subset X$ for all $i \in \mathbb{N}$.
\end{itemize}

We will make use of a result on densities for outer measures (see~\cite[Theorem 2.9.17]{fed1}), 
which we include here together with its proof for the reader’s convenience.

\begin{lemma}
\label{l:density}
    Let $\psi$ be an outer measure on $X$. If $E \subset X$ is a $\mathcal{H}^n$-measurable set such that $\psi(E)=0$, then for $\mathcal{H}^n$-a.e. $x \in E$ we have that
    \[
    \limsup_{r \to 0^+} \frac{\psi(B_r(x))}{r^n} \text{ equals either $0$ or $\infty$}.
    \]
\end{lemma}
\begin{proof}
Since $\mathcal{H}^n$ is $\sigma$-finite in $X$, we may assume with no loss of generality that $\mathcal{H}^n(E)<\infty$. Moreover, since by the inner regularity of Hausdorff measures, for every $\epsilon >0$, we find a compact set $K \subset E$ such that $\mathcal{H}^n(E \setminus K) \leq \epsilon$, we may further assume that $E$ is a closed set.

   For every $i=1,2,\dotsc$, consider the closed set 
   \[
   C_i:= \{x \in E \ | \  \psi(B_r(x)) \leq i r^n, \ \text{for } 0 <r <1/i  \}.
   \]
   Notice that, since
   \[
   \{x \in E \ | \  \limsup_{r \to 0^+} \psi(B_r(x))/r^n < \infty\}= \bigcup_{i=1}^\infty C_i,
   \]
   we can reduce ourselves to prove that for every $i=1,2,\dotsc$
   \begin{equation}
\limsup_{r \to 0^+} \frac{\psi(B_r(x))}{r^n} =0, \ \ \text{ for $\mathcal{H}^n$-a.e. $x \in C_i$}.
   \end{equation}
 To this purpose take a point $x \in C_i$ and $r \in (0,1/4i)$. Since $C_i$ is closed, we can cover $B_r(x) \setminus C_i$ by the family $\mathcal{F}$ of balls $\mathbb{B}_{t_{x'}}(x')$ with
\[
x' \in B_r(x) \setminus C_i \ \ \text{and} \ \ t_{x'} = \text{dist}(x,C_i)/2 >0.
\]
 For every such $(x',t_{x'})$, we have $t_{x'} \leq r/2$ and $|x'-x''| < 3t_{x'}$ for some $x'' \in C_i \cap B_r(x)$, hence
 \[
 \mathbb{B}_{t_{x'}}(x') \subset \mathbb{B}_{2r}(x), \qquad 8t_{x'} < 1/i, \qquad \psi(\mathbb{B}_{5t_{x'}}(x')) \leq \psi(\mathbb{B}_{8t_{x'}}(x'')) \leq i8^nt_{x'}^n.
 \]
By applying Vitali covering theorem, we find a countable disjoint subfamily $\mathcal{G}$ of $\mathcal{F}$ such that $B_r(x) \setminus C_i \subset \bigcup_{B \in \mathcal{G}} 5B$, where $5B$ denotes the concentric ball to $B$ with five times the radius of $B$, and, by denoting $c(B) \in X$ the center of the ball $B$, we have
\begin{align*}
    \psi(B_r(x)) =   \psi(B_r(x) \setminus C_i) &\leq \sum_{B \in \mathcal{G}} \psi(5B) \\
    &\leq \sum_{B \in \mathcal{G}} i8^n t_{c(B)}^n \\
    &\leq i8^n \mathcal{H}^n(\mathbb{B}_{2r}(x) \setminus C_i).
\end{align*}
Therefore we obtained that for every $x \in C_i$ and $r \in (0,1/4i)$ it holds
\[
\limsup_{r\to 0^+} \frac{\psi(B_r(x))}{r^n} \leq 16^n i\limsup_{i\to 0^+} \frac{\mathcal{H}^n(\mathbb{B}_{2r}(x) \setminus C_i)}{(2r)^n}.
\]
Since the right hand-side of the above inequality is $0$ for $\mathcal{H}^n$-a.e. $x \in C_i$ we obtain the desired assertion.
\end{proof}

\subsection{Transversality and cones}
\label{ss:transversal}

The notion of \emph{transversal family of maps} was originally introduced in \cite[Definition 7.2]{per} for families of maps $\Pi_\lambda \colon X \to \mathbb{R}^m$ parametrized by $\lambda \in \Lambda \subset \mathbb{R}^l$ and defined on a separable locally compact metric space $(X,d)$. The following definition of transversality can be found in \cite{hov} except for the fact that we do not assume regularity with respect to the spatial variable $x \in X$. 

\begin{definition}
\label{d:transversal}
Let $(X,d)$ be a separable locally compact metric space, let $m,l$ be integers satisfying $0 <m \leq l$, let $\Lambda \subset \mathbb{R}^l$ be open and bounded, and let 
\[
\Pi \colon X \times \Lambda \to \mathbb{R}^m
\]
be a continuous map. Define $\Pi_\lambda \colon X \to \mathbb{R}^m$ as $\Pi_\lambda(x):=\Pi(x,\lambda)$.

We say that the family $(\Pi_\lambda)_{\lambda \in \Lambda}$ is transversal if and only if the following conditions hold true.
\begin{enumerate}[(H.1)]
    \item For every $x \in X$ the map $\lambda \longmapsto \Pi_\lambda(x)$ belongs to $C^2(\Lambda;\mathbb{R}^m)$ and
    \begin{equation}
    \label{e:h1}
    \sup_{(\lambda,x) \in \Lambda \times X}\|D^j_\lambda \Pi_\lambda(x)\| < \infty, \ \ \text{for }j=1,2.
    \end{equation}
    \item For $\lambda \in \Lambda$ and $x,x' \in X$ with $x \neq x'$ define
    \begin{equation}
    \label{e:deftxx'}
        T_{xx'}(\lambda) := \frac{\Pi_\lambda(x) - \Pi_\lambda(x')}{d(x,x')};
    \end{equation}
    then there exists a constant $C' >0$ such that the property
    \begin{equation}
    \label{e:h2}
        |T_{xx'}(\lambda)| \leq C' \ \ \ \text{ implies } \ \ \
        |\text{J}_\lambda T_{xx'}(\lambda)| \geq C'.
    \end{equation}
   \item There exists a constant $C'' >0$ such that 
   \begin{equation}
   \label{e:h3}
     \|D_\lambda T_{xx'}(\lambda)\|, \|D^2_\lambda T_{xx'}(\lambda)\| \leq C'',
   \end{equation}
   for $\lambda \in \Lambda$ and $x,x' \in X$ with $x \neq x'$.
\end{enumerate}
\end{definition}

Here we need also to define cones around the preimages of points. 

\begin{definition}[Cone 1]
\label{d:cone1}
Let $\lambda \in \Lambda$, $x \in X$, $0<s<1$, and $r>0$, we define
\begin{align}
 \label{e:cone1}
    &X(x,\lambda,s) := \{x \in X \ | \ |\Pi_\lambda(x') -\Pi_\lambda(x)| < s\, d(x,x')  \},\\
    &X(x,\lambda,s,r) := X(x,\lambda,s) \cap \overline{B}_r(x).
\end{align}
\end{definition}
\begin{definition}[Cone 2]
\label{d:cone2}
Let $\lambda \in \Lambda$, $x \in X$, $0<s<1$, and let $V \subset \mathbb{R}^l$ be an $m$-dimensional plane, we define
\begin{align}
\label{e:cone2}
    &L_V(x,\lambda,s) := \{x \in X \ | \ \Pi_{\lambda'}(x')-\Pi_{\lambda'}(x) =0, \ |\lambda'-\lambda| < s, \ \pi_{V^\bot}(\lambda'-\lambda)=0 \},
\end{align}
where $V^\bot$ is the orthogonal to $V$.
\end{definition}

 In the studying of rectifiability, the key property of transversality relies on the equivalence between the two definitions of cones introduced above. Indeed, one can prove that, given $\lambda_0 \in \Lambda$, there exist a parameter $s_0 \in (0,1)$ and a constant $c \geq 1$ such that for every $x \in X$ and for every $\lambda$ locally around $\lambda_0$ the following equivalence holds
 \begin{equation}
 \label{e:equiconestand}
 \bigcup_{V \in \mathcal{V}_m} L_V(x,\lambda,s/c) \subset X(x,\lambda,s) \subset \bigcup_{V \in \mathcal{V}_m} L_V(x,\lambda,cs), \ \ \text{for every $s \in (0,s_0)$},
 \end{equation}
 where $\mathcal{V}_m$ denotes the family of all coordinate $m$-plane of $\mathbb{R}^l$. However, for our purposes, we need to guarantee that both sets in the left hand-side and right hand-side of \eqref{e:equiconestand} can be replaced with the preimages of suitable defined maps from $X$ with values in $\Lambda$. More precisely, given a point $x \in X$ and a parameter $\lambda \in \Lambda$, one can construct a map 
 \[
 \Phi_{V,\lambda} \colon X(x,\lambda,s_0) \to \Lambda,
 \]
 which will play a role of a (curvilinear) radial projection, and satisfies the following proposition.

 \begin{proposition}
 \label{p:equivcone}
 Let $(X,d)$ be a separable locally compact metric space, let $(\Pi_\lambda)_{\lambda \in \Lambda}$ be a transversal family of maps, let $x \in X$, let $\lambda_0 \in \Lambda$, and let $\delta >0$ be such that $\overline{\mathbb{B}}_{\delta}(\lambda_0) \subset \Lambda$. 
 
 Then there exist $s_0 \in (0,\delta/2)$ and a constant $c \geq 1$ such that, to every coordinate $m$-plane $V \subset \mathbb{R}^l$ and every $\lambda \in  \mathbb{B}_{s_0}(\lambda_0)$ we can associate a Borel set $U(V)$ in $\mathbb{R}^n \setminus \{x\}$ and a map
 \begin{align*}
 \Phi_{V,\lambda} \colon X(x,\lambda,s_0/c) \cap U(V)  \to \mathbb{B}_{s_0}(\lambda) \cap (V + \lambda) 
 \end{align*}
 such that 
 \begin{align}
 \label{e:curvradpro1}
     T_{xx'}(\Phi_{V,\lambda}(x'))&=0,  \ \ \text{for every $x' \in X(x,\lambda,s_0/c) \cap U(V)$},\\
     \label{e:contPsi}
    (x',&\lambda) \longmapsto \Phi_{V,\lambda}(x') \ \ \text{is continuous}\\
     \label{e:curvradpro1.97}
 &X(x,\lambda,s_0/c) \subset \bigcup_{V \in \mathcal{V}_{l,m}} U(V). 
 \end{align}
 
 In addition, by letting $\pi_{V^\bot} \colon \mathbb{R}^l \to V^\bot$ denote the projection of $\mathbb{R}^l$ onto the orthogonal complement of $V$ and by letting
 \[
 \lambda^\bot := \pi_{V^\bot}(\lambda -\lambda_0) \quad \text{and} \quad L:= \sup_{x \neq x'} \emph{Lip}(T_{xx'})+1,
 \]
 then, for every $\lambda \in \mathbb{B}_{s_0/2cL}(\lambda_0)$ and every $s \in (0,s_0/2c)$ the following inclusions hold true 
 \begin{align}
  \label{e:equivcone1234}
\Phi_{V,\lambda_0+\lambda^\bot}^{-1}\big(\mathbb{B}_{s/c}(\lambda)\big) \subset X(x,\lambda,s) \cap U(V) \subset  \Phi_{V,\lambda_0 + \lambda^\bot}^{-1}\big(\mathbb{B}_{cs}(\lambda)\big), 
 \end{align}
 for every $V \in \mathcal{V}_{l,m}$.
 \end{proposition}

\section{The Probabilistic Injective Projection Property}
\label{s:rp}

In this section, we prove that a measure possessing the singular projection property also satisfies the probabilistic injective projection property. To prove the main result of this section, we need the following two technical results. They are also crucial for the rectifiability criterion in Section~\ref{s:rvs}.

\begin{proposition}
\label{p:fcv1}
  Let $(X,d)$ be a separable locally compact metric space, let $(\Pi_\lambda)_{\lambda \in \Lambda}$ be a transversal family of maps, let $\mu$ be a finite Borel measure on $X$, and let $x \in X$. Then for $\mathcal{L}^l$-a.e. $\lambda \in \Lambda$ we have 
    \begin{equation}
    \label{e:boundcone3}
   \sup_{s \in (0,1)} s^{-m}\mu\big(X(x,\lambda,s)\big) < \infty.
    \end{equation}
   
   Moreover, given $\lambda_0\in \Lambda$, there exists a Borel measurable family $(\mathring{\mu}_{\lambda,x})_{\lambda \in \Lambda}$ of finite Borel measures in $X$ such that $\mathcal{L}^l$-a.e. $\lambda$ in an open ball centered at $\lambda_0$, any limit measures $\hat{\mu}_{\lambda,x}$ obtained as 
   \begin{equation}
        (s_i)^{-m}\mu \restr X(x,\lambda,s_i) \rightharpoonup \hat{\mu}_{\lambda,x}, \ \ \text{weakly* in $X$ as $s_i \to 0^+$},
    \end{equation}
    satisfies the following inequalities
    \begin{equation}
    \label{e:boundcar}
        {c^{-m}} \, \mathring{\mu}_{\lambda,x}(B) \leq \hat{\mu}_{\lambda,x}(B) \leq c^m \, \mathring{\mu}_{\lambda,x}(B), \ \ \text{ for every $B \subset X$ Borel}.
    \end{equation}
\end{proposition}

\begin{remark}
\label{r:lebequivalence}
 Although the family $(\mathring{\mu}_{\lambda,x})_{\lambda }$ will be defined for every $\lambda \in \Lambda$, we emphasize that only the $\mathcal{L}^l$-equivalence class of the map
\[
\lambda \longmapsto \mathring{\mu}_{\lambda,x}
\]
is relevant here and in Proposition~\ref{p:disretr1}. 
This should be contrasted with Lemma~\ref{l:keyalternative}, where the family 
$(\mathring{\mu}_{\lambda,x})_{\lambda}$ will instead be considered 
pointwise.
\end{remark}

\begin{proof}[Proof of Proposition \ref{p:fcv1}]
We will prove that the proposition holds for $\mathcal{L}^l$-a.e. $\lambda \in \mathbb{B}_{s_0/(16cL)}(\lambda_0)$, for any $\lambda_0 \in \Lambda$, where $s_0 \in (0,1)$, $L \geq 1$, and $c \geq 1$ are the constants given by Proposition \ref{p:equivcone}. Moreover, we will assume with no loss of generality that $\lambda_0=0 \in \Lambda$.

\vspace{2mm}

   Fix a coordinate $m$-plane $V \in \mathcal{V}_{l,m}$. Consider the corresponding Borel set $U(V) \subset \mathbb{R}^n \setminus \{x\}$ and, for every $\lambda \in \mathbb{B}_{s_0}(0)$, the corresponding map 
   \[
   \Phi_{V,\lambda} \colon X(x,\lambda,s_0/c) \cap U(V)  \to \mathbb{B}_{s_0}(\lambda) \cap (V + \lambda) 
   \]
   provided by Proposition \ref{p:equivcone}. Notice that, from the very definition of the constant $L$, by the triangular inequality we deduce the following inclusion 
\[
X(x,0 ,s_0/2c) \subset X(x,\lambda ,s_0/c), \ \ \text{ for every $\lambda \in \mathbb{B}_{s_0/2cL}(0)$}. 
\]
Therefore, for every $\lambda \in \mathbb{B}_{s_0/2cL}(0)$, the restriction 
\[
\Phi_{V,\lambda} \restr \colon X(x,0 ,s_0/2c) \cap U(V) \to V+\lambda 
\]
is a well defined map.

We will construct the family $(\mathring{\mu}_{\lambda})_{\lambda}$ by disintegrating $\mu$ with respect to the maps $\Phi_{V,\lambda}$.
   \vspace{2mm}

   \emph{-Step 1.} Fix $V \in \mathcal{V}_{l,m}$. In order to ease the notation we denote by $\Phi_{\lambda} \colon \mathbb{R}^n \to V + \lambda$ the map defined as
   \[
   \Phi_\lambda(x'):=
   \begin{cases}
       \Phi_{V,\lambda}(x') &\text{ if $x' \in X(x,0,s_0/2c) \cap U(V)$}\\
       \lambda  &\text{ otherwise}.   
\end{cases}
   \]
   
 We claim that the proposition holds for the measure $\tilde{\mu}$ defined as
   \[
\tilde{\mu}:= \mu \restr [X(x,0,s_0/8c) \cap U(V)].
\]
   
  Indeed, we can disintegrate $\tilde{\mu} \otimes \mathcal{H}^{l-m} \restr V^\bot$ with respect to the map (see property \eqref{e:contPsi})
  \[
  \Psi \colon X \times (\mathbb{B}_{s_0}(0) \cap V^\bot) \to \Lambda, \quad \Psi(x,\eta):= \Phi_\eta(x),
  \]
  as follows
  \begin{equation*}
   (\tilde{\mu} \otimes \mathcal{H}^{l-m}) (A) = \int_{\Lambda} \tilde{\mu}_{\lambda',x}(A)  \,d\big(\Psi (\tilde{\mu} \otimes \mathcal{H}^{l-m} \restr V^\bot)\big)(\lambda'), \ \ \text{ for $A \subset X \times \mathbb{B}_{s_0/16cL}(0)$ Borel},
\end{equation*}
  where $(\tilde{\mu}_{\lambda',x})_{\lambda \in \Lambda}$ is a Borel measurable family of probability measures, with $\tilde{\mu}_{\lambda',x}$ supported in $\Psi^{-1}(\lambda')$ for $\Psi(\tilde{\mu} \otimes \mathcal{H}^{l-m} \restr V^\bot)$-a.e. $\lambda' \in \Lambda$. By using that
  \[
  \Psi (\tilde{\mu} \otimes \mathcal{H}^{l-m} \restr V^\bot) = \Phi_\eta \tilde{\mu} \otimes \mathcal{H}^{l-m} \restr V^\bot,
  \]
  and the uniqueness of disintegration, one verifies that, for $\mathcal{H}^{l-m}$-a.e. $\eta \in V^\bot$, the family $(\tilde{\mu}_{\lambda',x})_{\lambda' \in (V + \eta)}$ is a disintegration for $\tilde{\mu}$. More precisely,
  \begin{equation}
   \tilde{\mu}(B) = \int_{V + \eta} \tilde{\mu}_{\lambda',x}(B)  \,d(\Phi_{\eta}\tilde{\mu})(\lambda'), \ \ \text{ for $B \subset X$ Borel}.
\end{equation}
 
By further decomposing the measure $\Phi_{\lambda}\tilde{\mu}$ as 
\[
\Phi_{\lambda}\tilde{\mu}= (\Phi_{\lambda}\tilde{\mu})^a \oplus (\Phi_{\lambda}\tilde{\mu})^s
\]
namely, its absolutely continuous and singular part with respect to $\mathcal{H}^m$, respectively, we can rewrite the above disintegration as 
\begin{align}
\label{e:dis99}
\tilde{\mu}(B) &= \int_{V + \lambda}  f_{\tilde{\mu},\lambda}(\lambda') \tilde{\mu}_{\lambda',x}(B) \, d\mathcal{H}^m(\lambda')\\
\nonumber
&+ \int_{V + \lambda} \tilde{\mu}_{\lambda',x}(B) \, d(\Phi_{\lambda}\tilde{\mu}^s)(\lambda'), \ \ \text{ for $B \subset X$ Borel},
\end{align}
where $f_{\tilde{\mu},\eta} \colon V+\eta \to [0,\infty)$ is a Borel function representing the density of $(\Phi_{\lambda}\tilde{\mu})^a$ with respect to $\mathcal{H}^m$. Furthermore, the functions 
$f_{\tilde{\mu},\eta}$ can be chosen so that the map 
$\lambda \longmapsto f_{\tilde{\mu},\lambda^\bot}(\lambda)$ 
is Borel measurable in the product space. Indeed, one may take $f \colon \Lambda \to [0,\infty)$ a Borel representative 
of the density of 
$\Psi(\tilde{\mu} \otimes \mathcal{H}^{l-m} \restr V^\bot)^a$ such that 
\[
\Psi(\tilde{\mu} \otimes \mathcal{H}^{l-m} \restr V^\bot)^s(\{f >0\})=0.
\]
Setting 
$f_{\tilde{\mu},\lambda^\bot}(\lambda) := f(\lambda)$, 
the desired property follows immediately.

Eventually, we define \emph{for every} $\lambda \in \Lambda$ the measure $\mathring{\mu}_{\lambda,x}$ in $X$ as
\begin{equation}
\mathring{\mu}_{\lambda,x} := 
\begin{cases}
\label{e:defringmu}
  (f_{\tilde{\mu},\lambda^\bot}(\lambda)\,\tilde{\mu}^{\lambda}_x  &\text{ if $ f_{\tilde{\mu},\lambda^\bot}(\lambda) >0$}\\
 \tilde{\mu}^{\lambda}_x &\text{ if $f_{\tilde{\mu},\lambda^\bot}(\lambda) =0$}.
\end{cases}
\end{equation}
As a consequence, the family $(\mathring{\mu}_{\lambda,x})_{\lambda \in \Lambda}$ is Borel measurable.

The following representation for $\tilde{\mu}$ holds for $\mathcal{H}^{l-m}$-a.e. $\eta \in V^\bot$
\begin{equation}
    \label{e:dis11}
\tilde{\mu}(B)= \int_{V + \eta} \mathring{\mu}_{\lambda',x}(B) \, d(\mathcal{H}^m+(\Phi_{\eta}\tilde{\mu})^s)(\lambda'), \ \ \text{ for $B \subset X$ Borel}.
\end{equation}

\vspace{2mm}

\emph{-Step 2}. Now fix $\eta \in \mathbb{B}_{s_0/2cL}(0) \cap V^\bot$ such that the representation \eqref{e:dis11} holds. From \eqref{e:dis11} and the standard properties of disintegration, we know that for $\mathcal{H}^m$-a.e. $\lambda' \in V+\eta$
    \begin{align}
     \label{e:existwl}
         s^{-m}\tilde{\mu} \restr \Phi_{\eta}^{-1}\big(\mathbb{B}_s(\lambda')\big)  \rightharpoonup  \mathring{\mu}_{\lambda',x}, \ \ \text{weakly* in $X$ as $s \to 0^+$.}
    \end{align}
    
    Moreover, from \eqref{e:equivcone1234} we deduce that for $\mathcal{H}^m$-a.e. $\lambda' \in \mathbb{B}_{s_0/16cL}(0) \cap (V+\eta)$, for every but sufficiently small values of $s \in (0,1)$, it holds 
\begin{align}
\label{e:ineq12}
        \tilde{\mu} \restr  \Phi_{\eta}^{-1}\big(\mathbb{B}_{s/c}(\lambda')\big)  \leq \tilde{\mu} \restr X(x,\lambda',s) \leq  \tilde{\mu} \restr \Phi_{\eta}^{-1}\big(\mathbb{B}_{cs}(\lambda')\big), \ \ \text{ as measures in $X$}.
    \end{align}
  where, for the right hand-side of the above inequality, we also used that 
  \begin{equation}
  \label{e:cont99}
   X(x,\lambda',s) \subset X(x,\lambda_0,s_0/8c), \ \ \text{ for every sufficiently small $s \in (0,1)$},
  \end{equation}
  whenever $\lambda' \in \mathbb{B}_{s_0/16cL}(\lambda_0)$.
  
  Conditions \eqref{e:existwl} and \eqref{e:ineq12} imply that, for $\mathcal{H}^m$-a.e. $\lambda' \in \mathbb{B}_{s_0/16cL}(\lambda_0) \cap (V +\eta)$ any weak* limit point $\hat{\mu}^{\lambda'}_x$ in $X$ for the family $\big(s^{-m}\tilde{\mu} \restr X(x,\lambda',s)\big)_{s \in (0,1)}$ has to satisfy
    \begin{align}
        c^{-m}\int \varphi \, d\mathring{\mu}_{\lambda',x}&=\lim_{i \to \infty} s_i^{-m} \int \varphi \, d\big[\tilde{\mu} \restr \Phi_{\eta}^{-1}\big(\mathbb{B}_{s_i/c}(\lambda')\big)\big] \\
        &\leq  \lim_{i \to \infty} s_i^{-m}\int \varphi \, d\big[\tilde{\mu} \restr X(x,\lambda',s_i)\big] \\
       &=\lim_{i \to \infty} s_i^{-m} \int \varphi \, d\big[\tilde{\mu} \restr \Phi_{\eta}^{-1}\big(\mathbb{B}_{cs_i}(\lambda')\big)\big] \\
       &= c^m \int \varphi \, d\mathring{\mu}_{\lambda',x}, \ \ \text{ for every $\varphi \in C^0_c(X,[0,\infty))$}.
    \end{align}

As a consequence, the above chain of inequalities allows us to deduce that, for $\mathcal{H}^m$-a.e. $\lambda' \in \mathbb{B}_{s_0/16cL}(0) \cap (V + \eta)$, the conditions \eqref{e:boundcone3} and \eqref{e:boundcar} are satisfied with respect to $\tilde{\mu}$.
 
In order to extend this property to $\mathcal{L}^l$-a.e. $\lambda \in \mathbb{B}_{s_0/16cL}(0)$, we first note that the preceding argument applies for $\mathcal{H}^{l-m}$-a.e. choice of $\eta \in V^\bot$. Therefore, by Fubini's theorem, we are left to verify the $\mathcal{L}^l$-measurability of the set of $\lambda \in \mathbb{B}_{s_0/16cL}(0)$ satisfying \eqref{e:boundcone3} and \eqref{e:boundcar}. Indeed, consider the set of points $\lambda \in \mathbb{B}_{s_0/16cL}(0)$ such that
\begin{equation}\label{e:measring}
\lim_{s \to 0^+} s^{-m}\,\tilde{\mu} \restr \Phi_{ \lambda^\bot}^{-1}\big(\mathbb{B}_s(\lambda)\big)
\quad \text{exists weakly as a measure.}
\end{equation}
From the continuity of $\Phi_V$ in \eqref{e:contPsi}, it is not difficult to see that such set is $\mathcal{L}^l$-measurable. Since \eqref{e:existwl} holds for every 
$\eta \in \mathbb{B}_{s_0/16cL}(0) \cap V^\bot$, Fubini’s theorem implies that the family of $\lambda$ satisfying \eqref{e:measring} has full $\mathcal{L}^l$-measure. By the previous argument, each such $\lambda$ also satisfies \eqref{e:boundcone3} and \eqref{e:boundcar}. Hence, the set of $\lambda \in \mathbb{B}_{s_0/16cL}(0)$ for which \eqref{e:boundcone3} and \eqref{e:boundcar} hold but \eqref{e:measring} fails is $\mathcal{L}^l$-negligible. Consequently, the desired measurability follows.

\vspace{2mm}

\emph{-Step 3.} By virtue of the second step, for every $V \in \mathcal{V}_{l,m}$ there corresponds a family $(\mathring{\mu}_{V,\lambda})_{\lambda \in \Lambda}$ satisfying condition \eqref{e:boundcar} for every weak* limit point in $X$ for the family $\big( s^{-m}\tilde{\mu}_V \restr X(x,\lambda,s)\big)_{s \in (0,1)}$, where we set now
\[
\tilde{\mu}_V := \mu \restr [X(x,0,s_0/8c) \cap U(V)].
\]
Let us call $\Lambda_0 \subset \mathbb{B}_{s_0/16cL}(0)$ a Borel set of full $\mathcal{L}^l$-measure for which condition \eqref{e:boundcar} holds simultaneously for $\tilde{\mu}_V$, for every $V \in \mathcal{V}_{l,m}$. We further define the family of measures $(\mathring{\mu}_{\lambda,x})_{\lambda \in \Lambda}$ as
\begin{equation}
\label{e:defmuring}
\mathring{\mu}_{\lambda,x}:= \sum_{V \in \mathcal{V}_{l,m}} \mathring{\mu}^\lambda_{V,x}.
\end{equation}

Suppose now that $s_i \to 0^+$ satisfies $\hat{\mu}^\lambda_{x,s_i} \rightharpoonup \hat{\mu}^\lambda_{x}$ weakly* in $X$ as $i \to \infty$. By the previous step, $\mathcal{L}^l$-a.e. $\lambda \in \mathbb{B}_{s_0/16cL}(\lambda_0)$ satisfies
\[
\sup_{s \in (0,1)} s^{-m}\tilde{\mu}_V\big(X(x,\lambda,s)\big) < \infty, \ \ \text{ for every $V \in \mathcal{V}_{l,m}$}.
\]
 As a consequence, we can apply the Banach-Alaoglu theorem for the weak* convergence of measures on separable locally compact metric spaces to possibly pass to a not-relabeled subsequence $s'_i \to 0^+$ such that for every $\lambda \in \Lambda_0$ and every $V \in \mathcal{V}_{l,m}$ it holds
 \[
 s^{-m}\tilde{\mu}_V\big(X(x,\lambda,s)\big) \rightharpoonup \hat{\mu}^\lambda_{V,x}, \ \ \text{ weakly* in $X$ as $i \to \infty$}.
\]
Therefore, for every $\lambda \in \Lambda_0$, we can give the following estimate (recall also \eqref{e:cont99})
\begin{align*}
    c^{-m}\int \varphi \ d\mathring{\mu}_{\lambda,x} &\leq \lim_{i \to \infty} s_i^{-m} \sum_{V \in \mathcal{V}_{l,m}} \int \varphi \, d[\tilde{\mu}_V \restr X(x,\lambda,s_i)] \\
    &= \lim_{i \to \infty} s_i^{-m} \int \varphi \, d[\mu \restr X(x,\lambda,s_i)]  \\
    &= \int \varphi \, d\hat{\mu}_{\lambda,x} \\
     &= \lim_{i \to \infty} s_i^{-m} \int \varphi \, d[\mu \restr X(x,\lambda,s_i)]  \\
     &=\lim_{i \to \infty} s_i^{-m} \sum_{V \in \mathcal{V}_{l,m}} \int \varphi \, d[\tilde{\mu}_V \restr X(x,\lambda,s_i)] \\
     &\leq c^{m}\int \varphi \ d\mathring{\mu}_{\lambda,x}, \ \ \text{ for every $\varphi \in C^0_c(X;\mathbb{R})$}.
\end{align*}

From this property \eqref{e:boundcar} follows and the proof is concluded.
\end{proof}

\begin{proposition}
\label{p:disretr1}
    Let $(X,d)$ be a separable locally compact metric space, let $(\Pi_\lambda)_{\lambda \in \Lambda}$ be a transversal family of maps, let $\mu$ be a finite Borel measure on $X$, and let $x \in X$. Then, for $\mathcal{L}^l$-a.e. $\lambda \in \Lambda$ we have 
    \begin{equation}
    \label{e:locweak}
    \sup_{r >0} r^{-m} \mu\big( B \cap \Pi_\lambda^{-1}(B^m_r(\Pi_\lambda(x)))\big) < \infty, \ \ \text{ for every Borel set $B \Subset \mathbb{R}^n \setminus \{x\}$}.
    \end{equation}
       Moreover, given $\lambda_0 \in \Lambda$, for $\mathcal{L}^l$-a.e. $\lambda$ in an open ball centered at $\lambda_0$, any limit measure $\mu_{\lambda,x}$ obtained as
       \begin{equation}
       \label{e:locweak14}
           r_i^{-m} \mu\restr  \Pi_\lambda^{-1}(B^m_{r_i}(\Pi_\lambda(x))) \rightharpoonup \mu_{\lambda,x}, \ \ \text{locally weakly* in $\mathbb{R}^n \setminus \{x\}$ as $r_i \to 0^+$}.
       \end{equation}
       satisfies the following inequalities
    \begin{align}
         \label{e:equ21}
        c^{-m} \mathring{\mu}_{\lambda,x}(B) \leq \int_{B} d(x,x')^m \, d\mu_{\lambda,x}(x') \leq c^m \, \mathring{\mu}_{\lambda,x}(B), \ \ \text{for every Borel set $B \subset X$}, 
    \end{align}
   where $(\mathring{\mu}_{\lambda,x})_{\lambda \in \Lambda}$ is the family of measures given by Proposition \ref{p:fcv1}.
\end{proposition}

\begin{proof}[Proof of Proposition \ref{p:disretr1}]
It is sufficient to prove that the proposition holds for $\mathcal{L}^l$-a.e. $\lambda \in \mathbb{B}_{s_0/(16cL)}(\lambda_0)$, for any $\lambda_0 \in \Lambda$, where $s_0 \in (0,1)$, $L \geq 1$, and $c \geq 1$ are the constants given by Proposition \ref{p:equivcone}. 

For a given $\rho >0$, since 
\[
 \Pi_\lambda^{-1}(B^m_{\rho s}(\Pi_\lambda(x))) \setminus B_\rho(x)\subset  X(x,\lambda,s) \setminus B_\rho(x), \ \ \text{ for every $s \in (0,1)$},
\]
we have that
\begin{align}
\label{e:ine51}
   \frac{1}{s^m}\int_{\Pi_\lambda^{-1}(B^m_{\rho s}(\Pi_\lambda(x)))  } \varphi \, d\mu \leq \frac{1}{s^m}\int_{X(x,\lambda,s)} \varphi \, d\mu, \ \ \varphi \in C^0_c(X \setminus \overline{B}_\rho(x);[0,\infty)).
\end{align}
    
    By virtue of inequality \eqref{e:ine51} we infer for every $\rho >0$
    \begin{align}
    \label{e:ine15}
        \sup_{r \in (0,\rho)} r^{-m} \mu\big(\Pi_\lambda^{-1}(B^m_r(\Pi_\lambda(x))) \setminus \overline{B}_\rho(x))\big) \leq \sup_{s \in (0,1)} \frac{1}{\rho^m}  \frac{\mu(X(x,\lambda,s))}{s^m}.  
    \end{align}
    Then, property \eqref{e:boundcone3} allows us to infer the validity of \eqref{e:locweak}, thanks to the arbitrariness of $\rho > 0$ in \eqref{e:ine15}.

     By using again \eqref{e:ine51} and Proposition \ref{p:fcv1}, we infer for $\mathcal{L}^l$-a.e. $\lambda \in \mathbb{B}_{s_0/16cL}(\lambda_0)$ that every limit point $\mu_{\lambda,x}$ obtained as in \eqref{e:locweak14} satisfies
    \begin{equation}
         \label{e:ine161}
        \rho^m\mu_{\lambda,x} \leq c^m \mathring{\mu}_{\lambda,x}, \ \ \text{ as measures in $X \setminus \overline{B}_\rho(x)$, for every $\rho >0$}.
    \end{equation}
    Condition \eqref{e:ine161} implies in particular that, if we take a parameter $b >1$, then
    \begin{equation}
    \label{e:ine11}
       b^{im} \mu_{\lambda,x} \restr (B_{b^{i+1}}(x) \setminus \overline{B}_{b^i}(x)) \leq c^m \mathring{\mu}_{\lambda,x} \restr (B_{b^{i+1}}(x) \setminus \overline{B}_{b^i}(x)), \ \ \text{ for every $i \in \mathbb{Z}$}
    \end{equation}
    But this means that, by defining the measure $\phi_b$ as 
    \[
    \phi_b:= \sum_{i \in \mathbb{Z}}^\infty  b^{im} \mu_{\lambda,x} \restr (B_{b^{i+1}}(x) \setminus \overline{B}_{b^i}(x)), 
    \]
    then, by exploiting the fact that $\phi_b \to d(\cdot,x)^m \, \mu_{\lambda,x}$ locally strongly in the sense of measures in $X \setminus \{x\}$ as $b \to 1^+$, together with \eqref{e:ine11}, we finally infer that 
    \begin{equation}
    \label{e:ine31}
        \int_{X \setminus \{x\}}  \varphi(x') \,d(x',x)^m d\mu_{\lambda,x}(x') \leq c^m \int_{X \setminus \{x\}}  \varphi(x') \, d\mathring{\mu}_{\lambda,x}(x'), \ \ \varphi \in C^0_c(X \setminus \{x\};[0,\infty)).
    \end{equation}

    For a given $\rho >0$, since
    \[
    X(x,\lambda,s,\rho) \subset \Pi^{-1}_\lambda(B^m_{\rho s}(\Pi_\lambda(x))) \cap \overline{B}_\rho(x), \ \ \text{ for every $s \in (0,1)$},
    \]
     we have that
    \begin{align}
   \frac{1}{s^m}\int_{X(x,\lambda,s,\rho)} \varphi \, d\mu   \leq \frac{1}{s^m}\int_{\Pi_\lambda^{-1}(B^m_{\rho s}(\Pi_\lambda(x)))  } \varphi \, d\mu, \ \ \varphi \in C^0_c(B_\rho(x);[0,\infty)).
\end{align}
 Therefore, by virtue of Proposition \ref{p:fcv1} we can argue exactly as above, to infer that, if $b >1$ and $\mu^\lambda$ is a limit point, then for $\mathcal{L}^l$-a.e. $\lambda \in \Lambda$ it holds 
    \begin{equation}
    \label{e:ine21}
         c^{-m}\mathring{\mu}^\lambda \restr (B_{b^i}(x) \setminus \overline{B}_{b^{i-1}}(x)) \leq b^{im} \mu^\lambda \restr (B_{b^i}(x) \setminus \overline{B}_{b^{i-1}}(x)), \ \ \text{as measures, for every $i\in \mathbb{Z}$}
    \end{equation}
   Hence, by arguing once again as above, we can infer from \eqref{e:ine21} that for $\mathcal{L}^l$-a.e. $\lambda \in \Lambda$
    \begin{equation}
    \label{e:ine41}
 \int_{X \setminus \{x\}}  \varphi(x') \, d(x',x)^m d\mu_{\lambda,x}(x') \geq  c^{-m}\int_{X \setminus \{x\}}  \varphi(x') \,  d\mathring{\mu}_{\lambda,x}(x'), \ \ \varphi \in C^0_c(X \setminus \{x\};[0,\infty)),
    \end{equation}
    whenever $\mu_{\lambda,x}$ is a limit point obtained as in \eqref{e:locweak14}. By putting together \eqref{e:ine31} and \eqref{e:ine41} we conclude the proof of the proposition. 
\end{proof}

We are finally in position to prove the main result concerning the probabilistic injective projection property.

\begin{theorem}
\label{p:finite1}
     Let $(X,d)$ be a separable locally compact metric space, let $(\Pi_\lambda)_{\lambda \in \Lambda}$ be a transversal family of maps, and let $\mu$ be a finite Borel measure on $X$. Then, for every Borel set $\Sigma \subset \Lambda$, one of the following two conditions holds
\begin{enumerate}
   \item $\Pi_\lambda \colon X \to \mathbb{R}^m$ is injective on a set of full $\mu$-measure for $\mathcal{L}^l$-a.e. $\lambda \in \Sigma$ 
    \item  $\mathcal{L}^l(\{\lambda \in \Sigma \ |\ (\Pi_\lambda \mu)^a \neq 0\}) >0$ .
\end{enumerate}

\end{theorem}
\begin{proof}
It is sufficient to prove that the proposition holds for $\mathcal{L}^l$-a.e. $\lambda \in \mathbb{B}_{s_0/(16cL)}(\lambda_0) \cap \Sigma$, for any $\lambda_0 \in \Lambda$, where $s_0 \in (0,1)$, $L > 0$, and $c \geq 1$ are the constants given by Proposition \ref{p:equivcone}. Fix $\lambda_0 \in \Lambda$ and, in order to ease the notation, let us denote the ball $\mathbb{B}_{s_0/(16cL)}(\lambda_0)$ simply by $\Lambda$. 

 We will prove that, if a measure $\mu$ does not satisfies (2), namely, if 
\begin{equation}
\label{e:singprop7}
\mathcal{L}^l(\{\lambda \in \Sigma \ |\ (\Pi_\lambda \mu)^a \neq 0\}) =0,
\end{equation}
then $\mu$ satisfies condition (1). 

To this purpose, we start with the following observation: $\Pi_\lambda$ is injective on a set of full $\mu$-measure if and only if  
\begin{equation}
\label{e:supp=1}
\Pi_\lambda\mu(\{y \in \mathbb{R}^m \ | \ \mathcal{H}^0(\text{supp} \, \mu_{\lambda,y}) \geq 2\})=0,
\end{equation}
where $(\mu_{\lambda,y})_{y \in \mathbb{R}^m}$ are the probability measures given by disintegrating $\mu$ with respect to $\Pi_\lambda \colon X \to \mathbb{R}^m$:
\begin{equation}
\label{e:displa}
\mu(B)= \int_{\mathbb{R}^m} \mu_{\lambda,y}(B) \, d(\Pi_\lambda\mu)(y), \ \ \text{ for every $B \subset X$ Borel}.
\end{equation}
Indeed, if $\Pi_\lambda$ is injective on a Borel set $E \subset X$ of full $\mu$-measure, then $\mathcal{H}^0(\Pi_\lambda^{-1}(y) \cap E)=1$ for every $y \in \Pi_\lambda(E)$. For every $y \in \Pi_\lambda(E)$ let us denote by $x_y$ the point belonging to $\Pi_\lambda^{-1}(y) \cap E$. By exploiting the fact that $\mu = \mu \restr E$, from formula \eqref{e:displa} we deduce that
\[
0=\mu_{\lambda,y}(\Pi_\lambda^{-1}(y) \setminus E) = \mu_{\lambda,y}(\Pi_\lambda^{-1}(y) \setminus \{x_y\}), \ \ \text{for $\Pi_\lambda \mu$-a.e. $y \in \mathbb{R}^m$},
\]
which immediately implies \eqref{e:supp=1}. Assume instead that \eqref{e:supp=1} holds. Then, by considering the Borel set $E$ defined as
\[
E:= \{x \in X \ | \  x \in \mu_{\lambda,\Pi_\lambda(x)}(\{x\}) >0 \},
\]
formula \eqref{e:displa} tells us that $E$ has full $\mu$-measure. Moreover, by using again \eqref{e:supp=1}, we also have that $\mathcal{H}^0(\Pi_\lambda^{-1}(y) \cap E)=1$ for every $\Pi_\lambda\mu$-a.e. $y \in \mathbb{R}^m$. Therefore, $\Pi_\lambda$ is injective on a Borel set of full $\mu$-measure. The claim is thus proved.

Now we argue by contradiction: assume that $\mu$ satisfies the singular projection property \eqref{e:singprop7} and that 
\[
\mathcal{L}^l\big(\Sigma':= \{\lambda \in  \Sigma \ |\  \Pi_\lambda\mu(\{y \in \mathbb{R}^m \ | \ \mathcal{H}^0(\text{supp} \, \mu_{\lambda,y}) \geq 2\})>0\}  \}\big)>0.
\]

We apply Fubini's theorem on the set of points $(x,\lambda ) \in X \times \Lambda$ satisfying the statement of Proposition \ref{p:disretr1} (see Appendix \ref{a:measprod}) to infer that, for $\mathcal{L}^{l}$-a.e. $\lambda \in \Lambda$, for $\mu$-a.e. $x \in X$ condition \eqref{e:locweak} holds true, namely,
\begin{equation}
     \label{e:locweak135}
    \sup_{r >0} r^{-m} \mu\big( B \cap \Pi_\lambda^{-1}(B^m_r(\Pi_\lambda(x)))\big) < \infty, \ \ \text{ for every Borel set $B \Subset \mathbb{R}^n \setminus \{x\}$}.
\end{equation}

We observe that, regardless of the measurability of the set $\Sigma'$, condition $\mathcal{L}^l(\Sigma') >0$ together with the assumption \eqref{e:singprop7} and the property \eqref{e:locweak135}, implies that 
\[
\Sigma' \cap \{\lambda \in \Sigma \ | \ (\Pi_\lambda \mu)^a = 0 \ \text{ and \eqref{e:locweak135} holds} \} \neq \emptyset. 
\]
If not, since the set 
\[
\Sigma'' := \{\lambda \in \Sigma \mid (\Pi_\lambda \mu)^a = 0 \ \text{and \eqref{e:locweak135} holds}\}
\]
is $\mathcal{L}^l$-measurable, we have $\mathcal{L}^l(\Sigma \setminus \Sigma'') = 0$. 
Hence, we would obtain 
\[
\Sigma' \subset \Sigma \setminus \Sigma''
\quad \text{and} \quad
\mathcal{L}^l(\Sigma') = 0,
\]
which yields a contradiction. As a consequence, there exists 
$\lambda \in \Sigma$ such that
\begin{enumerate}[(i)]
\item $(\Pi_\lambda\mu)^a =0$
\item $ \Pi_\lambda\mu(\{y \in \mathbb{R}^m \ | \ \mathcal{H}^0(\text{supp}\, \mu_{\lambda,y}) \geq 2\})>0$ 
\item for $\mu$-a.e. $x \in X$ \eqref{e:locweak135} holds.
\end{enumerate}

 We claim that (i) is in contradiction with (ii) and (iii). Indeed, condition (i) implies by virtue of the Radon-Nykodym differentiation Theorem that 
 \begin{equation}
\label{e:cond1.2}
 \lim_{r \to 0^+} r^{-m}\Pi_\lambda\mu(B^{m}_r(y))=\infty, \ \ \text{ for $\Pi_\lambda\mu$-a.e. $y \in \mathbb{R}^m$}.
 \end{equation}
 Moreover, property (iii) tells us that, the $\mu$-measurable set (see Appendix \ref{a:measprod}) $E_0 \subset  \mathbb{R}^n$ made of points $x$ fulfilling condition \eqref{e:locweak135} for $\mathcal{L}^l$-a.e. $\lambda$ has full $\mu$-measure. By using the inner regularity of finite Borel regular measures, we find a $\sigma$-compact set $K_0 \subset E_0$ of full $\mu$-measure. In particular, by means of formula \eqref{e:displa}, we deduce that
 \[
 K_0 \cap \Pi_\lambda^{-1}(y) \text{ has full $\mu_{\lambda,y}$-measure, for $\Pi_\lambda\mu$-a.e. $y \in \mathbb{R}^m$}.
 \]
Combining this information with \eqref{e:cond1.2} and (ii), we infer the existence of $y_0 \in \mathbb{R}^m$ and of two points $x_1 \neq x_2$ such that
\begin{equation}
\label{e:cond1.2.67}
x_1,x_2 \in \Pi_\lambda^{-1}(y_0) \cap K_0 \quad \text{and} \quad \lim_{r \to 0^+} r^{-m}\Pi_\lambda\mu(B^{m}_r(y_0))=\infty.
\end{equation}
In particular, condition \eqref{e:locweak135} is fulfilled at both points $x_1$ and $x_2$.
Therefore, by setting $0<\delta:= |x_2 -x_1|/4$, since
\[
\Pi_\lambda^{-1}(B^{m}_{r}(y_0)) \subset [ \Pi_\lambda^{-1}(B^{m}_{r}(\Pi_\lambda(x_1)) \setminus B_\delta(x_1)] \cup  [\Pi_\lambda^{-1}(B^{m}_{r}(\Pi_\lambda(x_2)) \setminus B_\delta(x_2)],
\]
we can estimate
\begin{align}
\nonumber
\limsup_{r \to 0^+} \, r^{-m}\Pi_\lambda\mu(B^{m}_{r}(y_0)) &= \limsup_{r \to 0^+}\, r^{-m}\mu\big( \Pi_\lambda^{-1}(B^{m}_{r}(y_0))\big) \\
&\leq \limsup_{r \to 0^+}\, r^{-m}\mu\big( \Pi_\lambda^{-1}(B^{m}_{r}(\Pi_\lambda(x_1))) \setminus B_\delta(x_1)\big) \\
&+\limsup_{r \to 0^+}\, r^{-m}\mu\big( \Pi_\lambda^{-1}(B^{m}_{r}(\Pi_\lambda(x_2))) \setminus B_\delta(x_2)\big) \\
&< \infty.
\end{align}

This contradicts the right hand-side of \eqref{e:cond1.2.67} and the claim is proved.    
\end{proof}

\section{Besicovitch-Federer projection theorem for measures}
\label{s:rvs}

We start with a proposition which tells us that, under certain circumstances, the measures $(\mathring{\mu}_{\lambda,x})_{x \in X}$ in Proposition \ref{p:disretr1} can be related to the measures $(\mu_{\lambda,y})_{y \in \mathbb{R}^m}$ obtained by disintegrating $\mu$ with respect to the map $\Pi_\lambda$.

\begin{proposition}
\label{p:finite}
    Let $(X,d)$ be a separable locally compact metric space, let $(\Pi_\lambda)_{\lambda \in \Lambda}$ be a transversal family of maps, and let $\mu$ be a finite Borel measure on $X$. Assume also that $\Pi_\lambda \mu^a \neq 0$ for every $\lambda \in \Sigma$ for some Borel set $\Sigma \subset \Lambda$ with $\mathcal{L}^l(\Sigma)>0$. Then, there exists a $(\mu  \otimes \mathcal{L}^l)$-measurable set $A \subset X \times \Lambda$ such that $(\mu \otimes \mathcal{L}^l)(A) >0$ and such that for $\mu$-a.e. $x \in X$ it holds  
    \begin{equation}
    \label{e:finite19}
    \mu^\lambda_{x,r} :=  \mu \restr \Pi_\lambda^{-1}\big(B_r^m(\Pi_\lambda(x))\big) \rightharpoonup \mu_{\lambda,x}, \ \ \text{ weakly in $X$ as $r \to 0^+$},
    \end{equation}
    for $\mathcal{L}^l$-a.e. $\lambda \in A_x$. 
\end{proposition}

\begin{remark}
\label{r:uniquelp}
    For those pairs $(x,\lambda)$ for which the above proposition holds, the convergence in \eqref{e:locweak14} can be strengthened from locally weak convergence in $X \setminus \{x\}$ to weak convergence in the whole of $X$, with a unique limit point $\mu_{\lambda,x}$. Consequently, in this case we also deduce that $\mu_{\lambda,x}$ is a finite measure on $X$. 
\end{remark}

\begin{proof}[Proof of Proposition \ref{p:finite}]
Define for every $\lambda \in \Lambda$ the set $Y_\lambda \subset \mathbb{R}^m$ as
\[
Y_\lambda := \{y \in \mathbb{R}^m \ | \ y \text{ is a non-null Lebesgue point for $\Pi_\lambda \mu^a$} \},
\]
where, with non-null Lebesgue point we intend a point $y_0 \in \mathbb{R}^m$ such that there exists $a >0$
\[
\lim_{r \to 0^+} \frac{1}{r^m} \int_{B^m_r(y_0)} |f_\lambda(y)- a| \, d\mathcal{H}^m(y)=0, \qquad (\Pi_\lambda \mu)^a = f_\lambda \, \mathcal{H}^m.
\]
The set $A \subset X \times \Lambda$ is then defined as
\[
A := \{(x,\lambda) \in X \times \Lambda \ | \ \Pi_\lambda(x) \in Y_\lambda  \}.
\]
Thanks to the fact that, by assumption, $\mathcal{L}^l(\Sigma)>0$, we see that 
\[
(\mu \otimes \mathcal{L}^l)(A) \geq  \int_{\Sigma}\mu(A_\lambda) \, d\mathcal{L}^l(\lambda) = \int_{\Sigma}\Pi_\lambda\mu^a(Y_\lambda) \, d\mathcal{L}^l(\lambda) >0
\]

    Given $\lambda \in \Lambda$, for $\mathcal{H}^m$-a.e. $y \in Y_\lambda$ it holds 
    \begin{equation}
    \label{e:cond1}
       \frac{\mu \restr \Pi^{-1}_\lambda(B^m_r(y))}{r^m } \rightharpoonup \mu_{\lambda,y}, \ \ \text{ weakly in $X$ as $r \to 0^+$},
    \end{equation}
    where $(\mu_{\lambda,y})_{y \in \mathbb{R}^m}$ are the probability measures given by disintegrating $\mu$ with respect to $\Pi_\lambda$ as in \eqref{e:displa}.
    Since by definition $\Pi_\lambda (\mu \restr A_\lambda) \ll \mathcal{H}^m$, from \eqref{e:cond1} we deduce that for $\mu$-a.e. $x \in A_\lambda$ it holds
    \begin{equation}
    \label{e:cond2}
       \frac{\mu \restr \Pi^{-1}_\lambda(B^m_r(\Pi_\lambda(x)))}{r^m } \rightharpoonup \mu_{\lambda,\Pi_\lambda(x)}, \ \ \text{ weakly in $X$ as $r \to 0^+$}.
    \end{equation}
    Since the set $\{(y,\lambda) \in \mathbb{R}^m \times \Lambda \ |\  y \in Y_\lambda \}$ is Borel (see Appendix \ref{a:measYlambda}), it follows that the set $A$ is $(\mu \otimes \mathcal{L}^l)$-measurable. We can apply Fubini's theorem to infer from \eqref{e:cond2} that for $\mu$-a.e. $x \in X$ it holds for $\mathcal{L}^l$-a.e. $\lambda \in A_x$ that
    \begin{equation}
    \label{e:cond3}
       \mu^\lambda_{x,r}=\frac{\mu \restr \Pi^{-1}_\lambda(B^m_r(\Pi_\lambda(x)))}{r^m } \rightharpoonup \mu_{\lambda,\Pi_\lambda(x)}, \ \ \text{ weakly in $X$ as $r \to 0^+$}.
    \end{equation}
    By setting $\mu_{\lambda,x}:= \mu_{\lambda,\Pi_\lambda(x)}$, this gives \eqref{e:finite19} and concludes the proof. 
\end{proof}

\begin{remark}
\label{r:coincidence}
    We notice that the proof of the above proposition tells us in particular that the unique limit points $\mu_{\lambda,x}$ in \eqref{e:finite19} coincides with $\mu_{\lambda,\Pi_\lambda(x)}$ for $\mu \otimes \mathcal{L}^l$-a.e. $(x,\lambda) \in A$.  
\end{remark}

The next is the key two alternatives lemma for measures which project with a non-trivial absolutely continuous part.

\begin{lemma}
\label{l:keyalternative}
    Let $(X,d)$ be a separable locally compact metric space, let $(\Pi_\lambda)_{\lambda \in \Lambda}$ be a transversal family of maps, and let $\mu$ be a finite Borel measure on $X$. By letting $A \subset X \times \Lambda$ be the measurable set given by Proposition \ref{p:finite}, then, for $\mu$-a.e. $x \in X$ the following condition 
    \begin{equation}
        \lim_{\delta \to 0^+} \limsup_{s \to 0+} \sup_{0 < \rho <\delta} \frac{\mu(X(x,\lambda,s,\rho) )}{(s\rho)^m} =0,
    \end{equation}
    is satisfied for $\mathcal{L}^l$-a.e. $\lambda \in A_x$.
\end{lemma}

\begin{proof}
It suffices to prove that the proposition holds for $\mathcal{L}^l$-almost every 
$\lambda \in \mathbb{B}_{s_0/(32cL)}(\lambda_0)$, for any $\lambda_0 \in \Lambda$, 
where $s_0 \in (0,1)$, $L > 0$, and $c \geq 1$ are the constants given by 
Proposition~\ref{p:equivcone}. 
To simplify the notation, we assume without loss of generality that $\lambda_0 = 0 \in \Lambda$.

Since the restriction $\lambda \in \mathbb{B}_{s_0/(32cL)}(0)$ corresponds, for each 
$x \in X$, to working spatially within the conical region $X(x,0,s_0/8c)$, 
it is convenient to define the measure $\tilde{\mu}_x$ on $X$ by
\[
    \tilde{\mu}_x := \mu \restr X(x,0,s_0/8c).
\]
Furthermore, with a slightly abuse of notation, we set
\[
A \equiv A \cap [X \times \mathbb{B}_{s_0/32cL}(0)].
\]

We will show that, for $\mu$-almost every $x \in X$,
\begin{equation}
\label{e:claim23}
\text{$\tilde{\mu}_x$ satisfies condition~(1) or~(2) for 
$\mathcal{L}^l$-a.e. $\lambda \in  A_x$.}
\end{equation}
Thanks to the inclusion~\eqref{e:ineq12}, the same conclusion will then hold for $\mu$.

\vspace{4mm}

\emph{Step 1.} We start by showing that for $\mu$-a.e. $x \in X$ the following condition is satisfied for $\mathcal{L}^l$-a.e. $\lambda \in A_x$ 
\begin{equation}
    \label{e:finconden}
    \limsup_{\delta \to 0^+} \limsup_{s \to 0+} \sup_{0 < \rho <\delta} \frac{\mu(X(x,\lambda,s,\rho) )}{(s\rho)^m} < \infty.
\end{equation}
Indeed, notice that, by the definition of the set $A$, for every $\lambda \in \Lambda$ we have
    \[
    \limsup_{r \to 0^+} \frac{\mu(\Pi^{-1}_\lambda(B^m_r(y)))}{r^m} < \infty, \ \ \text{ for every $y \in \Pi_\lambda(A_\lambda)$}.
    \]
    This means that for every $y \in \Pi_\lambda(A_\lambda)$ there exists $\delta_y >0$ such that
    \begin{equation}
    \label{e:estfinls}
    \sup_{0 < r < \delta_y} \frac{\mu(\Pi^{-1}_\lambda(B^m_r(y)))}{r^m} < \infty.
    \end{equation}
    Therefore, by using the inclusion $X(x,\lambda,s,\rho)  \subset \Pi^{-1}_\lambda\big(B^m_{s\rho}(\Pi_\lambda(x))\big)$ for every $s \in (0,1)$ and $\rho>0$, we estimate 
    \begin{align*}
       \sup_{0 < \rho <\delta_{\Pi_\lambda(x)}} \frac{\mu(X(x,\lambda,s,\rho) )}{(s\rho)^m} &\leq \sup_{0 < \rho <\delta_{\Pi_\lambda(x)}} \frac{\mu(\Pi^{-1}_\lambda(B^m_{s\rho}(\Pi_\lambda(x))))}{(s\rho)^m} \\
       &\leq \sup_{0 < r <\delta_y} \frac{\mu(\Pi^{-1}_\lambda(B^m_{r}(y)))}{r^m}, \ \ \text{ for $y= \Pi_\lambda(x)$}.
    \end{align*}
    Therefore, by using \eqref{e:estfinls}, we infer for every $\lambda \in \Lambda$
    \[
    \limsup_{s \to 0^+}\sup_{0 < \rho <\delta_{\Pi_\lambda(x)}} \frac{\mu(X(x,\lambda,s,\rho) )}{(s\rho)^m} < \infty, \ \ \text{ for every $x \in A_\lambda$}.
    \]
   Finally, thanks to $(\mu \otimes \mathcal{L}^l)$-measurability of the set of points $(x,\lambda) \in X \times \Lambda$ fulfilling condition \eqref{e:finconden} (similar to Appendix \ref{a:measprod}), we make use of Fubini's theorem to infer the validity of the claim. 

\vspace{4mm}

\emph{Step 2.} By virtue of \eqref{e:finconden}, in order to prove condition \eqref{e:claim23}, it suffices to show the following claim: for $\mu$-a.e. $x \in X$ one of the following conditions is satisfied for $\mathcal{L}^l$-a.e. $\lambda \in A_x$
\begin{align}
\label{e:(1)}
        &\lim_{\delta \to 0^+} \limsup_{s \to 0+} \sup_{0 < \rho <\delta} \frac{\mu(X(x,\lambda,s,\rho) )}{(s\rho)^m} =0 \\
        \label{e:(2)}
        &\lim_{\delta \to 0^+}\limsup_{s \to 0+} \sup_{0 < \rho <\delta} \frac{\mu(X(x,\lambda,s,\rho) )}{(s\rho)^m} =\infty.
    \end{align}

To prove the claim, for every $x \in X$ we decompose the measure $\tilde{\mu}_x$ as follows. First notice that
\[
\tilde{\mu}_{V,x} \leq \tilde{\mu}_x \leq \sum_{V \in \mathcal{V}_{l,m}} \tilde{\mu}_{V,x}, \ \ \text{ for every $V \in \mathcal{V}_{l,m}$},
\]
where we set $\tilde{\mu}_{V,x} := \tilde{\mu} \restr U(V)$ and $U(V)$ is the Borel set in $X \setminus \{x\}$ introduced in Proposition \ref{p:equivcone} and satisfying \eqref{e:curvradpro1.97}. As a consequence, the claim is proved if we show that for $\mu$-a.e. $x \in X$ and for every $V \in \mathcal{V}_{l,m}$,
\begin{equation}
    \label{e:claim24}
\text{$\tilde{\mu}_{V,x}$ satisfies condition~\eqref{e:(1)} or~\eqref{e:(2)} for 
$\mathcal{L}^l$-a.e. $\lambda \in A_x$.}
\end{equation}

Let us prove \eqref{e:claim24}. To this purpose, let us fix $V \in \mathcal{V}_{l,m}$. For every $\eta \in V^\bot$, by exploiting the disintegration in \eqref{e:dis11}, we decompose the measure $\tilde{\mu}_{V,x}$ in \emph{an $\eta$-dependent way} as follows
\begin{equation}
    \label{e:a+s}
    \tilde{\mu}_{V,x} =\tilde{\mu}^a_{V,x} \oplus \tilde{\mu}^s_{V,x}, 
    \end{equation}
    where $\tilde{\mu}^a_{V,x}$ and $\tilde{\mu}^s_{V,x}$ are mutually singular measures defined as 
    \begin{equation}
    \label{e:a+sdef}
     \tilde{\mu}^a_{V,x}:=\mathring{\mu}_{V,x}^{\lambda} \otimes \mathcal{H}^m \restr (V +\eta)\quad  \text{ and } \quad \tilde{\mu}^s_{V,x} :=  \mathring{\mu}_{V,x}^{\lambda} \otimes(\Phi_{V,\eta,x}\tilde{\mu}_{V,x})^s.
      \end{equation}
      Recall that the map $\Phi_{V,\eta,x}$ is the one provided by Proposition \ref{p:equivcone} and that $(\mathring{\mu}^\lambda_{V,x})_{\lambda \in \Lambda}$ is the Borel measurable family of measures in $X$ defined for every $\lambda \in \Lambda$ as in \eqref{e:defringmu}. We observe that, because of the obvious inequalities 
 \begin{align*}
    \tilde{\mu}^a_{V,x} \leq \tilde{\mu}^a_{V,x}\oplus \tilde{\mu}^s_{V,x} \leq \tilde{\mu}_{V,x}  \ \ &\text{ as measures in $X \setminus \{x\}$} \\
    \tilde{\mu}^s_{V,x} \leq \tilde{\mu}^a_{V,x}\oplus \tilde{\mu}^s_{V,x} \leq \tilde{\mu}_{V,x} , \ \ &\text{ as measures in $X \setminus \{x\}$}, 
    \end{align*}
 to prove \eqref{e:claim24}, it suffices to verify that for $\mathcal{H}^{l-m}$-a.e. $\eta \in V^\bot$, for $\mu$-a.e. $x \in X$,
 \begin{equation}
     \label{e:claim25}
     \text{$\tilde{\mu}^a_{V,x}$ and $\tilde{\mu}^s_{V,x}$ satisfies~\eqref{e:(1)} or~\eqref{e:(2)} for 
$\mathcal{H}^m$-a.e. $\lambda \in A_x \cap (V + \eta)$.}
 \end{equation}
Indeed, since the couples of points $(x,\lambda) \in X \times \Lambda$ fulfilling the condition \eqref{e:(1)} or \eqref{e:(2)} form a $(\mu \otimes \mathcal{L}^l)$-measurable set (similar to Appendix \ref{a:measprod}), property \eqref{e:claim25} in combination with Fubini's theorem allows us to deduce \eqref{e:claim24} directly from \eqref{e:claim25}.

\vspace{2mm}

The rest of the proof is dedicated to proving \eqref{e:claim25}. To simplify notation, we omit the $V$-dependence for the remainder of the proof and write 
\[
\tilde{\mu}^a_{x} \equiv \tilde{\mu}^a_{V,x}, \qquad  \tilde{\mu}^s_{x} \equiv \tilde{\mu}^s_{V,x}, \qquad \mathring{\mu}_{\lambda,x} \equiv \mathring{\mu}^\lambda_{V,x}, \qquad   \Phi_{\eta} \equiv \Phi_{V,\eta}.
\]

\vspace{4mm}

\emph{Step 3: The measure $\tilde{\mu}^a_{x}$}. The proof of \eqref{e:claim25} for $\tilde{\mu}^a_{x}$ can be further reduced to the following statement. Let $\epsilon>0$. Then, for $\mathcal{H}^{l-m}$-a.e. $\eta \in V^\bot$ we find a Borel set $\Sigma_\eta  \subset A_x \cap (V + \eta)$ such that for $\mu$-a.e. $x \in X$ 
\begin{enumerate}
\item $\mathcal{H}^m(A_x \cap (V + \eta) \setminus \Sigma_\eta ) \leq \epsilon$ 
\item $\tilde{\mu}^a_x$ satisfies condition \eqref{e:(1)} or \eqref{e:(2)} for $\mathcal{H}^m$-a.e. on $ \Sigma_\eta $. 
  \end{enumerate}  
    In order to show this, we use the weak* convergence in \eqref{e:finite19} (see also Remark \ref{r:uniquelp}) to deduce that, for $\mu$-a.e.\ $x \in X$, the convergence in \eqref{e:locweak14} can be strengthened from locally weak* convergence on $X \setminus \{x\}$ to weak* convergence on $X$ along the entire family as $s \to 0^+$, with the unique limit point $\mu_{\lambda,x}$ for $\mathcal{L}^l$-a.e.\ $\lambda \in A_x$. In this case, $\mu_{\lambda,x}$ is a finite measure on $X$.

As a consequence, for $\mathcal{H}^{l-m}$-a.e. $\eta \in V^\bot$, there exists a sufficiently large $M >0$ (depending on $\epsilon$) such that 
    \begin{equation}
    \label{e:epsilon}
    \mathcal{H}^m(\{\lambda \in A_x \cap (V + \eta)  \ |\ \mu_{\lambda,x}(X)  > M\}) \leq \epsilon .
    \end{equation}
    Hence, by defining $\Sigma_\eta := \{\lambda \in A_x \cap (V +\eta) \ |\ \mu_{\lambda,x}(X) \leq M\}$ we have that
    \begin{equation}
    \label{e:usefin}
        \mathcal{H}^m(A_x \cap (V + \eta) \setminus \Sigma_\eta ) \leq \epsilon.
    \end{equation}
    We further decompose $\tilde{\mu}^a_x$ as
    \[
    \tilde{\mu}^a_x= \tilde{\mu}^a_{x,\Sigma_\eta } \oplus \tilde{\mu}^a_{x,\Sigma_\eta^c},
    \]
    where $\tilde{\mu}^a_{x,\Sigma_\eta }$ and $\tilde{\mu}^a_{x,\Sigma_\eta^c}$ are mutually singular measures defined as
    \begin{align}
        \tilde{\mu}^a_{x,\Sigma_\eta} &:= \mathring{\mu}_{\lambda,x} \otimes \mathcal{H}^m \restr \Sigma_\eta \\
        \tilde{\mu}^a_{x,\Sigma_\eta^c} &:=\mathring{\mu}_{\lambda,x} \otimes \mathcal{H}^m \restr \big(A_x \cap (V+\eta) \setminus \Sigma_\eta\big).
    \end{align}
    
    \vspace{2mm}
    
    As above, in order to prove that for $\mathcal{H}^{l-m}$-a.e. $\eta \in V^\bot$, for $\mu$-a.e. $x \in X$ condition (2) is satisfied, it is enough to prove that $\tilde{\mu}^a_{x,\Sigma_\eta}$ and $\tilde{\mu}^a_{x,\Sigma_\eta^c}$ enjoy the same property.

      \vspace{4mm}

    \emph{-Substep 3.1: The measure $\tilde{\mu}^a_{x,\Sigma_\eta}$.} Recall the definition in \eqref{e:defmuring} and the property \eqref{e:equ21}. Thanks to the $(\mu \otimes \mathcal{L}^l)$-measurability of the set of points $(x,\lambda) \in X \times \Lambda$ satisfying $\eqref{e:equ21}$ (see Appendix \ref{a:measprod}), Fubini's theorem yields that for 
$\mathcal{H}^{l-m}$-a.e.\ $\eta \in V^\perp$, for $\mu$-a.e.\ $x \in \mathbb{R}^n$, 
and for $\mathcal{H}^m$-a.e.\ $\lambda \in \Sigma_\eta$, 
\begin{equation*}
    \mathring{\mu}_{\lambda,x}(B) 
    \leq \int_B \mathrm{d}(x,x')^m \, d\mu^{\lambda}_x(x'),
    \ \ 
    \text{ for every Borel set $B \subset X$}. 
\end{equation*}
Consequently, for $\mathcal{H}^{l-m}$-a.e.\ $\eta \in V^\bot$, 
for $\mu$-a.e.\ $x \in \mathbb{R}^n$, 
and for $\mathcal{H}^m$-a.e.\ $\lambda \in \Sigma_\eta$, we can give the following estimate
for every $s \in (0, s_0/32c^2L)$
\begin{align}
\nonumber
    \sup_{0 < \rho < \delta} 
    \frac{\tilde{\mu}^a_{x,\Sigma_\eta}(X(x,\lambda,s,\rho))}{(s\rho)^m}
    &\leq 
    \sup_{0 < \rho < \delta} 
    \frac{1}{(s\rho)^m}
    \int_{\mathbb{B}_{cs}(\lambda) \cap \Sigma_\eta}
    \mathring{\mu}_{\lambda',x}(\overline{B}_\rho(x)) \, d\mathcal{H}^m(\lambda')  \\
\nonumber
    &\leq 
    \sup_{0 < \rho < \delta} 
    \frac{c}{s^m}
    \int_{\mathbb{B}_{cs}(\lambda) \cap \Sigma_\eta}
    \bigg( 
        \int_{\overline{B}_\rho(x)} 
        \frac{\mathrm{d}(x,x')^m}{\rho^m} \, d\mu^{\lambda'}_x(x') 
    \bigg)
    d\mathcal{H}^m(\lambda')\\
\label{e:pass3}
    &\leq 
    \frac{c}{s^m}
    \int_{\mathbb{B}_{cs}(\lambda) \cap \Sigma_\eta}
    \mu^{\lambda'}_x(\overline{B}_\delta(x) \setminus \{x\})
    \, d\mathcal{H}^m(\lambda'),
\end{align}
   where in the first inequality above we applied the fundamental inclusion \eqref{e:equivcone1234} for transversal families of maps, and where we may use \eqref{e:equ21} since 
$\mathbb{B}_{cs}(\lambda) \subset \mathbb{B}_{s_0/16cL}(0)$ 
for every $\lambda \in \mathbb{B}_{s_0/32cL}(0)$ and 
$s \in (0, s_0/32c^2L)$.

   Now recall that, from the definition of $\Sigma_\eta$, the integrand in \eqref{e:pass3} is essentially bounded by the constant $M$. Therefore, by virtue of Lebesgue's differentiation theorem, we can write for $\mathcal{H}^m$-a.e. $\lambda \in   \Sigma_\eta$
    \[
    \limsup_{s \to 0^+}\sup_{0 < \rho <\delta} \frac{\tilde{\mu}^a_{\Sigma_\eta} (X(x,\lambda,s,\rho) )}{(s\rho)^m} \leq \mu_{\lambda,x}(\overline{B}_{\delta}(x) \setminus \{x\}), \ \ \text{ for every $\delta>0$.}
    \]
    But this means that for $\mathcal{H}^m$-a.e. $\lambda \in  \mathbb{B}_{s_0/32cL}(0) \cap \Sigma_\eta$ we have
    \begin{align}
        \lim_{\delta \to 0^+} \limsup_{s \to 0^+}\sup_{0 < \rho <\delta} \frac{\tilde{\mu}^a_{\Sigma^c} (X(x,\lambda,s,\rho) )}{(s\rho)^m}\leq \lim_{\delta \to 0^+} \mu_x^\lambda(\overline{B}_\delta(x) \setminus \{x\}) = \mu_x^\lambda(\emptyset)=0,
    \end{align}
    where, in the last limit, we used the monotonicity property of finite measures. 
    
    This gives the desired property for $\tilde{\mu}^a_{\Sigma_\eta}$.

    \vspace{2mm}

\emph{-Substep 3.2: The measure $\tilde{\mu}^a_{\Sigma_\eta^c}$.} Fix $\eta \in V^\bot$. For every $\delta >0$, we consider the outer measure $\psi_\delta \colon 2^{V+\eta} \to [0,\infty]$ defined as
\[
\psi_\delta(E) := \sup_{0 < \rho < \delta} \frac{1}{\rho^m} \int_{E \setminus \Sigma_\eta }^* \mathring{\mu}_{\lambda',x}(\overline{B}_\rho(x)) \, d\mathcal{H}^m(\lambda'), \ \ E \subset V + \eta.
\]
Since by definition $\psi_\delta( \Sigma_\eta)=0$, we can apply Lemma \ref{l:density} to infer that for every $\delta >0$
\begin{equation}
    \limsup_{s \to 0^+} \frac{\psi_\delta(\mathbb{B}_s(\lambda))}{s^m}=0 \ \ \text{or} \ \ \limsup_{s \to 0^+} \frac{\psi_\delta(\mathbb{B}_s(\lambda))}{s^m}=\infty, \ \ \text{for $\mathcal{H}^m$-a.e. $\lambda \in \Sigma_\eta$}.
\end{equation}
Therefore, we can finally infer that for $\mathcal{H}^m$-a.e. $\lambda \in \Sigma_\eta$
\begin{equation}
   \lim_{\delta \to 0^+} \limsup_{s \to 0^+} \frac{\psi_\delta(\mathbb{B}_s(\lambda))}{s^m}=0 \ \ \text{ or } \ \ \lim_{\delta \to 0^+} \limsup_{s \to 0^+} \frac{\psi_\delta(\mathbb{B}_s(\lambda))}{s^m}= \infty.
\end{equation}

From the very definition of $\psi_\delta$, we may again use the inclusion \eqref{e:equivcone1234} to infer from the above condition that, for every $\eta \in V^\bot$ and for $\mu$-a.e.\ $x \in X$, the condition (2) is satisfied by $\mu^a_{\Sigma_\eta^c}$.

By putting together these two Substeps, we finally infer the validity of \eqref{e:claim25} for $\tilde{\mu}^a_{V,x}$.

\vspace{4mm}

\emph{Step 4: The measure $\tilde{\mu}^s_x$.} We can proceed analogously as in Substep 1.2. 

Fix $\eta \in V^\bot$. For every $\delta >0$, we consider the outer measure $\psi_\delta \colon 2^{V +\eta} \to [0,\infty]$ defined as 
\[
\psi_\delta(E) := \sup_{0 < \rho < \delta} \frac{1}{\rho^m} \int_{E}^* \mathring{\mu}_{\lambda',x}(\overline{B}_\rho(x)) \, d \Phi_{\eta}\mu^s(\lambda'), \ \ E \subset V +\eta.
\]
Since $\Phi_{\eta}\mu^s \, \bot \, \mathcal{H}^m \restr (V+\eta)$ we find a Borel set $\Sigma_\eta \subset V + \eta$ with $\mathcal{H}^m((V+\eta) \setminus \Sigma_\eta)=0$ and $\Phi_{\eta}\mu^s((V + \eta) \cap \Sigma_\eta)=0$. Therefore, for every $\delta>0$, we deduce also that $\psi_{\delta}(\Sigma_\eta)=0$. For this reason we can apply once again Lemma \ref{l:density} to finally infer that for every $\delta>0$
\begin{equation}
\label{e:lemfed}
    \limsup_{s \to 0^+} \frac{\psi_\delta(\mathbb{B}_s(\lambda))}{s^m}=0 \ \ \text{or} \ \ \limsup_{s \to 0^+} \frac{\psi_\delta(\mathbb{B}_s(\lambda))}{s^m}=\infty, \ \ \text{for $\mathcal{H}^m$-a.e. $\lambda \in V + \eta$}.
\end{equation}
    
    From the very definition of $\psi_\delta$, we may use again \eqref{e:equivcone1234} to infer from the above condition that, for every $\eta \in V^\bot$ and for $\mu$-a.e.\ $x \in X$, the condition \eqref{e:(1)} or \eqref{e:(2)} is satisfied by $\tilde{\mu}^s_x$ for $\mathcal{H}^m$-a.e. $\lambda \in A_x \cap (V + \eta)$.

Property \eqref{e:claim25} is finally proved also for $\tilde{\mu}^s_x$.

\end{proof}

As a corollary of Lemma \ref{l:keyalternative} we derive the following.

\begin{lemma}
\label{l:keyalternative1}
    Let $(X,d)$ be a separable locally compact metric space, let $(\Pi_\lambda)_{\lambda \in \Lambda}$ be a transversal family of maps, and let $\mu$ be a finite Borel measure on $X$. By letting $A \subset X \times \Lambda$ be the measurable set given by Proposition \ref{p:finite}, then, for $\mathcal{L}^l$-a.e. $\lambda \in \Lambda$ the following condition 
    \begin{equation}
    \label{e:3limit1}
        \lim_{\delta \to 0^+} \limsup_{s \to 0+} \sup_{0 < \rho <\delta} \frac{\mu(X(x,\lambda,s,\rho) )}{(s\rho)^m} =0,
    \end{equation}
    is satisfied for $\mu$-a.e. $x \in A_\lambda$.
\end{lemma}
\begin{proof}
  Since the couple of points $(x,\lambda) \in X \times \Lambda$ satisfying condition \eqref{e:3limit1} form a $(\mu \otimes \mathcal{L}^l)$-measurable set (see Appendix \ref{a:measkeyalt}), it suffices to apply Fubini's theorem to Lemma \ref{l:keyalternative}.
\end{proof}

Now we prove e technical lemma for \emph{graphical} purely unrectifiable sets.

\begin{lemma}
\label{l:graphical}
   Let $(X,d)$ be a separable locally compact metric space, let $(\Pi_\lambda)_{\lambda \in \Lambda}$ be a family of maps, let $\mu$ be a finite Borel measure on $X$, and let $\lambda \in \Lambda$. Assume that for $\mu$-a.e. $x \in X$ it holds true 
     \begin{equation}
     \label{e:hyp1}
         \lim_{\delta \to 0^+} \limsup_{s \to 0+} \sup_{0 < \rho <\delta} \frac{\mu(X(x,\lambda,s,\rho) )}{(s\rho)^m} =0.
     \end{equation}
     Then, given a $\mathcal{H}^m$-measurable set $B \subset \mathbb{R}^m$ and a $\mathcal{H}^m$-measurable function $g \colon B \to \mathbb{R}^{n}$ such that 
     \begin{enumerate}
         \item $\Pi_\lambda(g(y)) =y$ for every $y \in B$
         \item $\emph{Im}(g):= \{x \in X \ | \  x=g(\Pi_\lambda(x))\}$ satisfies
         \[
         \mu(R \cap \emph{Im}(g))=0, \ \ \text{ for every countably $m$-rectifiable set $R \subset X$,}
         \]
     \end{enumerate}
     it holds $\mu(\emph{Im}(g))=0$.

\end{lemma}

\begin{proof}
     Let us first assume that $g$ is continuous and that $\text{Im}(g)$ is a bounded set.
     
    In view of the pointwise convergence with respect to the $x$-variable in \eqref{e:hyp1}, we can apply Egorov's theorem to deduce for every $\epsilon >0$ there exists a compact set $K \subset \text{Im}(g)$ such that \eqref{e:hyp1} holds uniformly for $x \in K$ and such that $\mu(\text{Im}(g) \setminus K) \leq \epsilon$. This means that there exists $\delta_0 >0$ such that 
    \begin{equation}
    \label{e:egorov1}
    \limsup_{s \to 0+} \sup_{0 < \rho <\delta} \frac{\mu(X(x,\lambda,s,\rho) )}{(s\rho)^m} \leq \epsilon, \ \ \text{ for every $x \in K$ and every $\delta \in (0,2\delta_0)$}.
    \end{equation}
Moreover, by using the definition of limsup and by applying once again Egorov's theorem in \eqref{e:egorov1}, we find a compact subset $K_{\delta_0} \subset K$ such that $\mu(K \setminus K_{\delta_0}) \leq \epsilon$ and a parameter $s_{\delta_0} \in (0,1)$ such that 
    \begin{equation}
    \label{e:estimate}
    \sup_{s \in (0,s_{\delta_0})} \sup_{0<\rho<\delta_0} \frac{\mu(X(x,\lambda,s,\rho) )}{(s\rho)^m} \leq 2\epsilon, \ \ \text{ for every $x \in K_{\delta_0}$}.
    \end{equation}

    Let us fix a modulus of continuity $\omega \colon (0,\infty) \to (0,\infty)$ for the map $g \restr K$. 
For the rest of the proof we fix $\delta \in (0,\delta_0/3)$.  

Consider the covering $\mathcal{C}_\delta$ of $K_{\delta_0}$ defined as follows. For every $y \in \Pi_\lambda(K_{\delta_0})$ let 
    \[
    C_y:= \{ x \in \overline{B}_{\delta/2}\big(g(y)\big) \ | \ \Pi_\lambda(x) \in \overline{B}^m_{\omega(\delta/4)}(y) \subset \mathbb{R}^m \},
    \]
    and define $\mathcal{C}_\delta:= \{C_y \ | \ y \in \Pi_\lambda(K_{\delta_0})\}$. Notice that $\mathcal{C}_\delta$ is actually a covering of $K_{\delta_0}$, because for every $y \in \Pi_\lambda(K_{\delta_0})$ the set $C_y$ contains the point $\{g(y)\}$, which, by property (1), is the only point of $\text{Im}(g)$ belonging to the level set $\Pi_\lambda^{-1}(y)$.
    We further notice that 
    \begin{equation}
    \label{e:impli1}
    x_1,x_2 \in C_y \cap K_{\delta_0} \text{ for some $y \in \Pi_\lambda(K_{\delta_0})$} \Rightarrow |x_1-x_2| \leq \delta/2.
    \end{equation}
    Indeed, in this case we find $y_1,y_2 \in \overline{B}^m_{\omega(\delta/4)}(y)$ such that $x_1=g(y_1)$ and $x_2=g(y_2)$ and hence it holds 
    \begin{align*}
        |x_1 -x_2| = |g(y_1)-g(y_2)| &\leq  |g(y_1)-g(y)| + |g(y)-g(y_2)|\\
        &\leq \delta/4 + \delta/4 =\delta/2.
    \end{align*}

    Let us denote by $\mathcal{B}_\delta$ the covering of $\Pi_\lambda(K_{\delta_0})$ defined as $\mathcal{B}_\delta:=\{\overline{B}^m_{\omega(\delta/4)}(y) \ | \ y \in \Pi_\lambda(K_{\delta_0})\}$.

By applying Besicovitch's covering theorem, we find $N_{m}$ families made of pairwise disjoint balls in $\mathcal{B}_\delta$, say $\{\mathcal{B}_1, \dotsc,\mathcal{B}_{N_{m}}\}$, such that 
\[
\Pi_\lambda(K_{\delta_0}) \subset \bigcup_{i=1}^{N_m} \bigcup_{B \in \mathcal{B}_i} B,
\]
with $N_m$ a dimensional constant. We will further denote
\[
\mathcal{C}_i := \{C_y \ | \ y \in \Pi_\lambda(K_{\delta_0}), \  \overline{B}^m_{\omega(\delta/4)}(y) \in \mathcal{B}_i\}, \ \ \text{ for every $i=1,\dotsc, N_m$}.
\]
We observe that $\{\mathcal{C}_1,\dotsc,\mathcal{C}_{N_m}\}$ covers $K_{\delta_0}$. Indeed, if $x \in K_{\delta_0}$ then there exists $\overline{B}^m_{\omega(\delta/4)}(y) \in \mathcal{B}_i$ for some $i=1,\dotsc, N_m$, such that $y' := \Pi_\lambda(x) \in \overline{B}^m_{\omega(\delta/4)}(y)$. Therefore we can use the modulus of continuity of $g$ to estimate
\[
|g(y) - g(y')| \leq \delta/4 < \delta/2,
\]
which immediately implies $x = g(y') \in \overline{B}_{\delta/2}(g(y))$ and hence $x \in C_y$, as desired.

Now fix an index $i \in {1, \dotsc, N_m}$ and a ball $B \in \mathcal{B}_i$.
Recall that $B$ has the form $B = \overline{B}^m_{\omega(\delta/4)}(y)$ for some $y \in \Pi_\lambda(K_{\delta_0})$.
We now work on the corresponding set $C_y$.

We claim that, for every $s \in (0,1)$, for $\mu$-a.e. $x \in K_{\delta_0} \cap C_y$ the set $X(x,\lambda,s,\delta) \cap K_{\delta_0} \cap C_y$ is non-empty. Indeed, assume by contradiction that there exists a set 
$K' \subset K_{\delta_0} \cap C_y$ with $\mu(K') > 0$ 
such that every $x \in K'$ contradicts the claim. Then, since by property (1) we know that the map $\Pi_\lambda \restr K'$ is injective, by calling $f \colon \Pi_\lambda(K') \to K'$ the inverse map, we can make use of the condition $X(x,\lambda,s,\delta) \cap K_{\delta_0} \cap C_y = \emptyset$ for every $x \in K'$ together with the implication \eqref{e:impli1} to easily deduce that
\[
|y-y'| = |\Pi_\lambda(f(y)) - \Pi_\lambda(f(y'))| \geq s |f(y)-f(y')|, \ \ \text{ for every $y,y' \in \Pi_\lambda(K')$}.
\] 
Since $\text{Im}(f)= K'$, we deduce from the above inequality that $K'$ is contained in the image of a $1/s$-Lipschitz map. However, since we are assuming that $\text{Im}(g)$ satisfies condition (2), and since $K' \subset \text{Im}(g)$, we deduce $\mu(K')=0$. This immediately gives a contradiction and our claim is thus confirmed. 

In view of the above claim, by defining the map $h \colon K_{\delta_0} \cap C_y \to [0,\infty)$ as (the map $h$ actually also depends on $s$, even though this is not reflected in our notation.
)
\[
h(x):= \sup \{d(x,x') \ | \ x' \in  X(x,\lambda,s,\delta) \cap K_{\delta_0} \cap C_y\},
\]
we deduce that for every $s \in (0,1)$
\[
h(x) >0, \ \ \text{ for $\mu$-a.e. $x \in K_{\delta_0} \cap C_y$}.
\]
Therefore, for $\mu$-a.e. $x \in K_{\delta_0} \cap C_y$ there exists $y_x \in X(x,\lambda,s,\delta) \cap K_{\delta_0} \cap C_y$ such that
\[
4|x-y_x| \geq 3h(x). 
\]
Observe that $h(x) \leq \text{diam}(C_y) \leq \delta $ for every $x \in K_{\delta_0} \cap C_y$. In addition, for every $s \in (0,s_{\delta_0}/4)$, it holds
\[
\Pi_\lambda(K_{\delta_0} \cap C_y) \subset \bigcup_{x \in K_{\delta_0} \cap C_y} \overline{B}^m_{sh(x)/5}(\Pi_\lambda(x)) \subset  \overline{B}^m_{\omega(\delta/4)+s\delta/5}(y).
\]
Vitali's covering theorem provides a countable subset $D$ of $K_{\delta_0} \cap C_y$ such that
\[
\{\overline{B}^m_{sh(x)/5}(\Pi_\lambda(x)) \ | \ x \in D\} \text{ is disjointed}, \qquad \Pi_\lambda(K_{\delta_0} \cap C_y) \subset \bigcup_{x \in D} \overline{B}^m_{sh(x)}(\Pi_\lambda(x)).
\]
Since by construction we have
\[
\begin{split}
\sum_{x \in D} \alpha_m (sh(x))^m &\leq 5^m \alpha_m[\omega(\delta/4) +s \delta/5]^m,
\end{split}
\]
if we prove that for every $x \in D$ it holds
\begin{equation}
\label{e:lemmafederer}
\mu(K_{\delta_0} \cap C_y \cap \Pi_\lambda^{-1}(\overline{B}^m_{sh(x)}(\Pi_\lambda(x))) \leq 4\epsilon(12)^m \alpha_m (sh(x))^m,
\end{equation}
then, it will immediately follows that
\begin{equation}
\label{e:estimate1}
\mu(K_{\delta_0} \cap C_y) \leq  4\epsilon (60)^m \alpha_m[\omega(\delta/4) + s\delta/5]^m.
\end{equation}

By arguing exactly as \cite[Lemma 3.3.6]{fed1} we assert that for every $x \in K_{\delta_0} \cap C_y$ it holds
\begin{equation}
\label{e:estimate2}
    K_{\delta_0} \cap C_y \cap \Pi_\lambda^{-1}(\overline{B}^m_{sh(x)}(\Pi_\lambda(x))) \subset X(x,\lambda,4s,3h(x))  \cup X(y_x,\lambda,4s,3h(x)).
\end{equation}
Therefore, by using that $3 h(x) \leq 3\delta < \delta_0$ (recall that we choose $\delta \in (0,\delta_0/3)$) and that $4s \leq s_{\delta_0}$, we can use \eqref{e:estimate} twice to infer from \eqref{e:estimate2} the validity of \eqref{e:lemmafederer} and, as a consequence, the validity of \eqref{e:estimate1}. By recalling that $s$ can be arbitrarily chosen in $(0,s_{\delta_0}/4)$, we finally infer that
\begin{equation}
    \label{e:keyestimate}
    \mu(K_{\delta_0} \cap C_y) \leq 4\epsilon \, 60^m \alpha_m\, \omega(\delta/4)^m = 4 \epsilon \, 60^m \mathcal{H}^m(\overline{B}^m_{\omega(\delta/4)}(y)).
\end{equation}

Since both $i \in \{1, \dotsc, N_m\}$ and $C_y \in \mathcal{C}_i$ were arbitrary, the above estimate holds for all such $C_y$ and $i$, and therefore
\begin{align}
    \mu(\text{Im}(g)) \leq \mu(K_{\delta_0}) + 2\epsilon &\leq \sum_{i=1}^{N_m} \sum_{C_y \in \mathcal{C}_i} \mu(K_{\delta_0} \cap C_y)  +2\epsilon \\
    &\leq 4\epsilon \, 60^m\sum_{i=1}^{N_m} \sum_{\overline{B}^m_{\omega(\delta/4)}(y) \in \mathcal{B}_i} \mathcal{H}^m(\overline{B}^m_{\omega(\delta/4)}(y))  +2\epsilon \\
    &\leq 4\epsilon \, 60^m\sum_{i=1}^{N_m}  \alpha_m(\text{diam}(\Pi_\lambda(K_{\delta_0}))+2\omega(\delta/4))^m +2\epsilon \\
    &\leq \epsilon [60^m 4N_m\alpha_m(\text{diam}(\Pi_\lambda(K_{\delta_0}))+2\omega(\delta/4))^m+2],
\end{align}
where we used that $\bigcup_{\overline{B}^m_{\omega(\delta/4)}(y) \in \mathcal{B}_i} \overline{B}^m_{\omega(\delta/4)}(y) $ is contained in a ball with radius equal to $\text{diam}(\Pi_\lambda(K_{\delta_0}))+2\omega(\delta/4)$ for every $i=1,\dotsc, N_m$. By using also that $\delta$ was arbitrarily chosen in $(0,\delta_0/3)$, and that the modulus of continuity $\omega$ depends only on $g$ and $K$, we finally infer    
\begin{align*}
\mu(\text{Im}(g)) &\leq \epsilon [60^m 4N_m\alpha_m\text{diam}(\Pi_\lambda(K_{\delta_0}))^m+2] \\
& \leq \epsilon [60^m 4N_m\alpha_m\text{diam}(\Pi_\lambda(K))^m+2]\\
& \leq \epsilon [60^m 4N_m\alpha_m\text{diam}(\text{Im}(g))^m+2]
\end{align*}
Eventually, the arbitrariness of $\epsilon >0$ gives the desired assertion since we assumed at the beginning that $\text{diam}(\text{Im}(g)) < \infty$.

Eventually, the case of a $\mathcal{H}^m$-measurable function $g$ follows by an easy application of Lusin's theorem together with a covering argument.
\end{proof}

We are finally in position to prove the main result. 

\begin{theorem}
\label{t:final}
    Let $(X,d)$ be a separable locally compact metric space, let $(\Pi_\lambda)_{\lambda \in \Lambda}$ be a transversal family of maps, and let $\mu$ be a finite Borel measure on $X$. Assume that there exists $\Sigma \subset \Lambda$ with $\mathcal{L}^l(\Sigma)>0$ such that the disintegration 
    \begin{equation}
    \mu= \mu_{\lambda,y} \otimes \Pi_\lambda \mu,
    \end{equation}
    satisfies
    \begin{equation}
    \label{e:fin(2)}
        \mathcal{H}^m(\{y \in \mathbb{R}^m \ |\ \mu_{\lambda,y}\text{ contains non-trivial atoms}\})>0, \text{ for every $\lambda \in \Sigma$}. 
    \end{equation}
    Then, there exist a submeasure $\mu' \leq \mu$ and a countably $m$-rectifiable set $R \subset X$ such that $\mu'(R)>0$ and
    \begin{equation}
    \label{e:abscont11}
    \mu' \restr R \ll \|\Pi_\lambda\|^m_{C^{0,\alpha}} \mathcal{H}^{\alpha m},
    \end{equation}
    for any $\alpha \in (0,1]$.
\end{theorem}

\begin{proof}
    Let $A \subset X \times \Lambda$ be the set given by Proposition \ref{p:finite}. Thanks to the $(\mu \otimes \mathcal{L}^l)$-measurability of the set of points $(x,\lambda) \in X \times \Lambda$ satisfying condition \eqref{e:3limit1} (see Appendix \ref{a:measkeyalt}), we can apply Lemma \ref{l:keyalternative1} to find $\lambda \in \Sigma$ such that
\begin{enumerate}[(i)]
\item \eqref{e:fin(2)} is satisfied at $\lambda$
\item \eqref{e:hyp1} is satisfied for $\mu$-a.e. $z \in A_\lambda$.
\end{enumerate}
Notice that the condition \eqref{e:fin(2)} at $\lambda$ implies that
\[
\mu(A_\lambda) >0.
\]
From the definition of $A_\lambda$, we have 
$\Pi_\lambda (\mu \restr A_\lambda) \ll \mathcal{H}^m$.
 Hence, condition (i) implies that there exists $0< \mu' \leq \mu \restr A_\lambda$ which disintegrates atomically with respect to $\Pi_\lambda$, namely, it holds
  \begin{equation}
   \label{e:disfin}
   \mu' ( B) = \int_{\mathbb{R}^m} \sum_{x \in \Pi^{-1}_\lambda(y)} m_y(x) \delta_x(B) \,  d \mathcal{H}^m(y), \ \ \text{$B \subset X$ Borel},
   \end{equation}
   where $(m_y)_{y \in \mathbb{R}^m}$ are suitable multiplicity functions $m_y \colon X \to [0,\infty)$ supported in the level set $\Pi_\lambda^{-1}(y)$ for $\mathcal{H}^m$-a.e. $y \in \mathbb{R}^m$. 

By a standard saturation argument (see, for instance, \cite[Subsection 3.2.14]{fed1}), we decompose the measure $\mu'$ as 
    \begin{equation}
    \label{e:saturation}
    0< \mu' = \mu_r \oplus \mu_u
    \end{equation}
    where $\mu_r = \mu' \restr R$ for some countably $m$-rectifiable set $R \subset A_\lambda$, while $\mu_u = \mu' \restr (A_\lambda \setminus R)$ and the set $A_\lambda \setminus R$ satisfies
    \begin{equation}
    \label{e:mumunr}
    \mu'(R' \cap (A_\lambda \setminus R))=0, \ \ \text{ for every countably $m$-rectifiable set $R' \subset X$.}
    \end{equation}

    We claim that $\mu_u = 0$. Indeed, assume by contradiction that $\mu_u \neq 0$. In order to reach a contradiction, we will prove that (i) implies the existence of a compact set $K \subset \mathbb{R}^m$ and of a continuous function $g \colon K \to X$ such that
\begin{equation}
\label{e:contra1}
    \mu_u(\text{Im}(g)) >0.
\end{equation}
   After verifying \eqref{e:contra1}, we can apply (ii) in combination with Lemma~\ref{l:graphical} to immediately conclude from the property of the set $A_\lambda \setminus R$ in \eqref{e:mumunr} that
   \[
   \mu_u(\text{Im}(g))= \mu'(\text{Im}(g) \cap (A_\lambda \setminus R))= \mu'\big(\text{Im}(g \restr \Pi_\lambda(A_\lambda \setminus R))\big)=0,
   \]
   where we also used that $\Pi_\lambda(A_\lambda \setminus R)$ is $\mathcal{H}^m$-measurable thanks to the measurable projection Theorem.
This contradicts \eqref{e:contra1} and will prove the claim.
   
  To prove \eqref{e:contra1}, assume for a moment that the family $(m_y)_{y \in \mathbb{R}^m}$ can be chosen in such a way that the following holds: the map $\Psi \colon \mathbb{R}^m \times X \to [0,\infty)$ defined as
   \begin{equation}
   \label{e:measPhi}
    \Psi(y,x) := m_y(x) 
   \end{equation}
   is Borel measurable. If this is the case, the set $\{\Psi >0\}$ is Borel. Therefore, by applying again the measurable projection Theorem we have that $\pi_{\mathbb{R}^m}(\{\Psi >0\})$ is $\mathcal{H}^m$-measurable, where $\pi_{\mathbb{R}^m} \colon \mathbb{R}^m \times X \to \mathbb{R}^m$ denotes the projection onto the first component. In addition, the multi-valued map
   \[
   \Gamma \colon \pi_{\mathbb{R}^m}\big(\{\Psi >0\}\big) \to 2^{X}, \qquad \Gamma(y):= \{\Psi >0\} \cap \pi_{\mathbb{R}^m}^{-1}(y),
   \]
   is well defined and $\Gamma(y) \neq \emptyset$ for every $y \in \pi_{\mathbb{R}^m}\big(\{\Psi >0\}\big)$. Then, the measurable selection Theorem gives us a $\mathcal{H}^m$-measurable map $\tilde{g} \colon \pi_{\mathbb{R}^m}\big(\{\Psi >0\}\big) \to X$ such that $\tilde{g}(y) \in \Gamma(y)$ for every $y \in \pi_{\mathbb{R}^m}\big(\{\Psi >0\}\big)$. By construction
   \[
   m_y(\tilde{g}(y)) >0 \ \ \text{and} \ \  \Pi_\lambda(\tilde{g}(y))=y, \ \ \text{for every $y \in \pi_{\mathbb{R}^m}\big(\{\Psi >0\}\big)$}. 
   \]
  Thanks to property (i) it follows that $\mathcal{H}^m\big(\pi_{\mathbb{R}^m}\big(\{\Psi >0\}\big)\big) >0$.
  
  By applying Lusin's theorem to the measurable function $\tilde{g}$, we find a compact set $K \subset \pi_{\mathbb{R}^m}(\{\Psi >0\})$ such that $g := \tilde{g} \restr K$ is continuous and $\mathcal{H}^m(K)>0$. Therefore, we can infer directly from the disintegration \eqref{e:disfin} the validity of $\eqref{e:contra1}$. 
  
 Consequently, to complete our contradiction argument, we are reduced to proving the following statement: let $\eta$ be a finite measure on $X$, then the multiplicity family $(m_y)_{y \in \mathbb{R}^m}$ of its atomic part with respect to the disintegration given by $\Pi_\lambda$, can be chosen in such a way that the map $\Psi$ in \eqref{e:measPhi} is Borel measurable.

To this purpose, assume that $\eta$ is a finite measure on $X$ that can be disintegrated as follows
\[
\eta(B) = \int_{\mathbb{R}^m} \eta_y (B \cap \Pi_\lambda^{-1}(y)) \, d\mathcal{H}^m(y), \ \ B \subset X \text{ Borel},
\]
for a family $(\eta_y)_{y \in \mathbb{R}^m}$ of finite measures in $X$. From \eqref{e:disintegration00}, it follows that $(\eta_y)_{y \in \mathbb{R}^m}$ can be chosen in such a way that
\[
y \longmapsto \eta_y (B \cap \Pi_\lambda^{-1}(y)) \text{ is Borel measurable for every $B \subset X$ Borel}.
\]
As a consequence, it is not difficult to verify that, for every $r>0$, the map
\[
 (y,x) \longmapsto \eta_y(B_r(x)) \text{ is a Borel measurable map on the product space $\mathbb{R}^m \times X$}.
\]
Therefore by defining the function 
\[
\Psi(y,x) = m_y(x) := \limsup_{i \to \infty} \eta_y(B_{1/i}(x)),
\]
we easily see that the family $(m_y(x))_{y}$ satisfies \eqref{e:disfin} and that $\Psi$ is a Borel representative.

Finally, our contradiction argument is concluded and consequently $\mu_u = 0$. This immediately gives that $0 <\mu_r(X) = \mu'(R)$.

In order to complete the proof of the theorem we just need to show that $\mu' \ll \mathcal{H}^{\alpha m}$ whenever the family $(\Pi_\lambda)_{\lambda}$ is made of $\alpha$-H\"older regular functions. To this purpose, we make use of \cite[Corollary 2.10.11]{fed1}, to infer
\begin{equation}
\label{e:formulah0}
\int_{\mathbb{R}^m} \mathcal{H}^0(B \cap \Pi_\lambda^{-1}(y)) \, d\mathcal{H}^m(y) \leq \|\Pi_\lambda\|_{C^{0,\alpha}}^m \mathcal{H}^{\alpha m}(B), \ \ \text{ for every Borel set $B \subset R$}.
\end{equation}
By defining the Borel regular measure $\phi$ as (see \cite[Theorem 2.10.10]{fed1})
\[
\phi(B) := \int_{\mathbb{R}^m} \mathcal{H}^0(B \cap \Pi_\lambda^{-1}(y)) \, d\mathcal{H}^m(y), \ \ \text{ for every Borel set $B \subset R$}, 
\]
we immediately see that, condition (i) together with the fact that $\Pi_\lambda \mu_r  = \Pi_\lambda \mu'  \ll \mathcal{H}^m$, imply $\mu' \ll \phi$. Eventually, applying formula \eqref{e:formulah0}, we finally infer that $\mu'  \ll \mathcal{H}^{\alpha m}$, and the proof is complete.
\end{proof}

Theorem \ref{t:final} immediately leads to the following Besicovitch-Federer type of results.

 \begin{theorem}
    \label{t:final2}
    Under the hypothesis of Theorem \ref{t:final} assume that $\Pi_\lambda \colon X \to \mathbb{R}^m$ is Lipschitz continuous for every $\lambda \in \Lambda$ and that
    \begin{equation}
    \label{e:atomicity147}
    \mu_{\lambda,y} \text{ is purely atomic for $\mathcal{H}^m$-a.e. $y \in \mathbb{R}^m$.}
    \end{equation}
    Then $\mu$ decomposes as
    \begin{equation}
         \mu = \mu_r \oplus \mu_u
     \end{equation}
     where, the measure $\mu_r$ is $m$-rectifiable, and the measure $\mu_u$ satisfies 
     \begin{equation}
         \label{e:final2.95}
         \Pi_\lambda\mu \, \bot \, \mathcal{H}^m, \ \ \text{for $\mathcal{L}^l$-a.e. $\lambda \in \Lambda$}.
     \end{equation}
\end{theorem}

\begin{proof}
 Exactly as in the decomposition \eqref{e:saturation}, we find a countably $m$-rectifiable set $R \subset X$ such that $\mu(R' \setminus R)=0$ for every countably $m$-rectifiable set $R' \subset X$. Therefore, by virtue of \eqref{e:atomicity147}, we make use of Theorem \ref{t:final} to infer the validity of \eqref{e:final2.95} for the measure $\mu \restr (X \setminus R)$. 
 
 We can further decompose $\mu \restr R = (\mu \restr R)^a \oplus (\mu \restr R)^s$, namely, its $\mathcal{H}^m$-absolutely continuous part and $\mathcal{H}^m$-singular part, respectively. Finally, since $(\mu \restr R)^s$ must be concentrated on a $\mathcal{H}^m$-negligible set, by setting $\mu_r := (\mu \restr R)^a$ and $\mu_u := \mu \restr (X \setminus R) + (\mu \restr R)^s$, the theorem immediately follows.
\end{proof}

Finally, combining Proposition~\ref{p:finite1} with Theorem~\ref{t:final}, we obtain the following characterization of pure unrectifiability, already stated in the introduction as Theorem~\ref{t:mainthm1} in terms of orthogonal projections in Euclidean space.

\begin{theorem}
\label{t:mainthm1.}
 Let $(X,d)$ be a separable locally compact metric space, let $(\Pi_\lambda)_{\lambda \in \Lambda}$ be a transversal family of Lipschitz maps, and let $\mu$ be a finite Borel measure on $X$. Assume that there exists a Borel set $\Sigma \subset \Lambda$ with $\mathcal{L}^\ell(\Sigma)>0$, such that for every $\lambda \in \Sigma$
\begin{equation}
\label{e:mainthm1.1.}
\mu_{\lambda,y} \text{ is purely atomic for $\mathcal{H}^m$-a.e.\ $y \in \mathbb{R}^m$}.
\end{equation}
Then, for $\mathcal{L}^l$-a.e. $\lambda \in \Sigma$,
\begin{equation}
\label{e:mainthm1.2.}
\Pi_\lambda \mu \perp \mathcal{H}^m
\quad \text{and} \quad
\Pi_\lambda \colon X \to \mathbb{R}^m
\text{ is injective on a set of full $\mu$-measure}
\end{equation}
if and only if $\mu$ is concentrated on a purely $m$-unrectifiable set.
\end{theorem}

\section{Rectifiability via slicing and connection with currents}

\label{s:recslice}

We start by proving a more general version of Theorem \ref{t:rectcri}, which is valid for separable locally compact metric spaces endowed with a transversal family of maps.

\begin{theorem}
   \label{t:rectcritr}
    Let $(X,d)$ be a separable locally compact metric space, let $(\Pi_\lambda)_{\lambda \in \Lambda}$ be a transversal family of Lipschitz maps, and let $(\mu_\lambda)_{\lambda \in \Lambda}$ be a family of finite Borel measures in $X$. Assume that there exists a Borel set $\Sigma \subset \Lambda$ and a finite Borel measure $\eta$ in $X$ such that
    \begin{equation}
    \label{e:rectcritr}
        \Pi_\lambda \mu_\lambda \ll \mathcal{H}^m \ \ \text{and} \ \ \mu_\lambda \ll \eta, \ \ \text{ for $\lambda \in \Sigma$}.
    \end{equation}
     Then, for almost every $\lambda \in \Sigma$ and for every Borel measurable function $f \colon \mathbb{R}^m \to X$ such that $\Pi_\lambda \circ f = \emph{Id}_{\mathbb{R}^m}$, the restricted measure $\mu_\lambda \restr \text{\emph{Im}}(f)$ is $m$-rectifiable.
\end{theorem}

\begin{proof}
   We may assume that $\mathcal{L}^\ell(\Sigma) >0$ and that $(\mu_\lambda)$ are not $\mathcal{L}^\ell$-almost all identically null. We argue similarly to the proof of Theorem \ref{t:final} by replacing the measure $\mu$ with $\eta$. 
   
   Let $A \subset X \times \Lambda$ be the set relative to the measure $\eta$ given by Proposition \ref{p:finite}. Thanks to the $(\eta \otimes \mathcal{L}^l)$-measurability of the set of points $(x,\lambda) \in X \times \Lambda$ satisfying condition \eqref{e:3limit1} (see Appendix \ref{a:measkeyalt}), we can apply Lemma \ref{l:keyalternative1} to find $\lambda \in \Sigma$ such that
\begin{enumerate}[(i)]
\item \eqref{e:rectcritr} and $\mu_\lambda>0$ are satisfied at $\lambda$.
\item \eqref{e:hyp1} is satisfied by $\eta$ for $\eta$-a.e. $z \in A_\lambda$.
\end{enumerate}
Notice that the condition (i) implies that
\[
\eta(A_\lambda) >0 \ \ \text{ and } \ \ \mu_\lambda(\mathbb{R}^n \setminus A_\lambda) =0.
\]

Thanks to the absolute continuity assumed in \eqref{e:rectcritr}, the desired conclusion will follow by verifying that, for every $\epsilon >0$ we find a Borel set $B \subset \mathbb{R}^n$ such that $\eta(\mathbb{R}^n \setminus B) \leq \epsilon$ and $\mu_\lambda \restr (\text{Im}(f) \cap B)$ is $m$-rectifiable. We further observe that, if $f \colon \mathbb{R}^m \to X$ satisfies the hypothesis of the theorem, then  
\begin{equation}
\label{e:graph=atom}
\begin{split}
\mu_\lambda \restr \text{Im}(f) &= \mu_\lambda \restr \big(\text{Im}(f) \cap \bigcup_{y \in \mathbb{R}^m} \{\text{atoms of $(\mu_{\lambda,y}$}\}\big) \\
&=  \mu_\lambda \restr \big(\text{Im}(f) \cap \bigcup_{y \in \mathbb{R}^m} \{\text{atoms of $(\mu_{\lambda,y}$}\} \cap A_\lambda\big).
\end{split}
\end{equation}

 Hypothesis \eqref{e:rectcritr} allows us to apply the Radon-Nykodym's Theorem to write $\mu_\lambda = m_\lambda \, \eta$ for $m_\lambda \in L^1(\mathbb{R}^n;\eta)$. By virtue of the integrability of $m_\lambda$ with respect to $\eta$, for every $\epsilon >0$ we find a sufficiently large $t >0$ such that $B:= \{m_\lambda \geq t\}$ satisfies $\eta(\mathbb{R}^n \setminus B) < \epsilon$. Therefore, we just need to prove that $\mu_\lambda \restr B$ is an $m$-rectifiable measure whenever restricted to $\text{Im}(f)$. But since $m_\lambda \leq t$ (almost everywhere) on $B$, the condition (ii) passes immediately to $\mu_\lambda \restr B$. More precisely, \eqref{e:hyp1} is satisfied by $\mu_\lambda \restr B$ for $\eta$-a.e. $z \in A_\lambda$, and hence, also for $(\mu_\lambda \restr B)$-a.e. $z \in A_\lambda$. 

Finally, thanks to the equality in \eqref{e:graph=atom}, we are in position to argue exactly as in the proof of Theorem \ref{t:final}, with $\mu$ replaced by $\mu_{\lambda} \restr B$, to infer that $\mu_\lambda \restr (\text{Im}(f) \cap B)$ is $m$-rectifiable. 
\end{proof}

\subsubsection*{A counterexample.} We now restrict our analysis to $\mathbb{R}^2$, endowed with the family of orthogonal projections. 
The following example shows that rectifiability cannot, in general, be deduced from condition~\eqref{e:rectcri} 
when it is required only for a \emph{dense set} of projections.

We claim that there exists a nonzero finite Borel measure $\mu$ on $\mathbb{R}^2$ such that $\mu$ satisfies~\eqref{e:rectcri} for a dense set of lines $\ell \in \mathrm{Gr}(2,1)$ and
\begin{equation}
    \mu(R) = 0 \quad \text{for every countably $1$-rectifiable set $R \subset \mathbb{R}^2$.}
\end{equation}
To illustrate this, consider the Sierpiński gasket $S \subset \mathbb{R}^2$ constructed by Kenyon in~\cite{ken}. Recall that $S$ is a self-similar $1$-set satisfying the following property: the projection of $S$ onto a line has positive $\mathcal{H}^1$-measure if and only if the projection contains an interval. This occurs precisely for all lines whose slope can be written as $p/q$ in lowest terms, where the integers $p, q$ satisfy $p+q \equiv 0 \pmod{3}$~\cite[Lemma~5]{ken}.  

In addition, the following general result holds for self-similar sets~\cite[Theorem~1.9]{far}: if a self-similar $1$-set $K \subset \mathbb{R}^2$ satisfies $\mathcal{H}^1(\pi(K)) > 0$ for some orthogonal projection $\pi$ onto a line $\ell$ through the origin, then
\[
\pi_\ell (\mathcal{H}^1 \restr K)
\simeq 
  \mathcal{H}^1 \restr \pi_\ell(K).
\]
Hence, by letting $\mu := \mathcal{H}^1 \restr S$, we find that $\mu$ satisfies~\eqref{e:rectcri} for a countable dense set of lines.

\subsection{Connection with currents}
\label{ss:conncurrents}
Here we recall the notion of integer-coefficient current given in Definition~\ref{d:Gconstrained}. 
Let $T$ be a classical Federer-Fleming $m$-current of finite mass in $\mathbb{R}^n$. We say that $T$ is an integer-coefficient current if, for every Borel set
$B\subset\mathbb R^n$ and every $V \in \text{Gr}(n,m)$ there exists a function $\theta_{\pi_V,B}\in L^1(\mathbb R^m;G)$ such that 
\begin{equation}
\label{e:G-valued}
\pi_{V\#}(T\restr B)
=
\llbracket\theta_{\pi_V,B}\rrbracket.
\end{equation}

We observe that, this definition was introduced in \cite[Definition 4.2]{ak1} for rectifiable metric currents. In the metric setting, the pushforward condition \eqref{e:G-valued} is required for every Lipschitz map $\varphi$ defined on the underlying metric space with values in $\mathbb{R}^m$. 

The definition is also compatible with the correspondence between
Euclidean metric currents and Federer-Fleming currents. More
precisely, let $T\in\mathbf M_m(\mathbb R^n)$ be a metric current and
let $\tilde T$ be its associated Federer-Fleming current as in
\cite[Theorem~11.1]{ak1}. Then $T$ satisfies
\eqref{e:G-valued} with $\pi_V$ replaced by any Lipschitz map if and only if, the same holds true for $\tilde{T}$. Indeed, restriction and Lipschitz pushforward are compatible with the
correspondence $T\mapsto\tilde T$.

\begin{proof}[Proof of Theorem \ref{t:Grect}]
  For each $V \in \text{Gr}(n,m)$, we fix a Borel choice of an orientation, represented by a unit simple $m$-vector spanning $V$. We denote by $d\pi_V$ the corresponding unit volume form on $V$ pulled back through the orthogonal projection $\pi_V$. Reversing the orientation replaces both $T \restr d\pi_V$ and the associated slices by their negatives. 
  
  Since $T$ is a $m$-current with finite mass, for every $m$-plane $V \subset \mathbb{R}^n$ 
the functional
\[
f\longmapsto
\langle T,f\,d\pi_V\rangle,
\qquad
f\in C_c^\infty(\mathbb R^n),
\]
extends uniquely to a finite signed Radon measure on $\mathbb R^n$,
which we denote by $\eta_V$.

\vspace{1mm}
\emph{Step 1: Rectifiability of the mass.} For every $V\in\text{Gr}(n,m)$, we claim that
\begin{equation}
\label{e:projected-muV-ac19}
\pi_V |\eta_V|
\ll
\mathcal H^m\restr V.
\end{equation}
Indeed, for every Borel set $B \subset \mathbb{R}^n$, by using the restriction operator $T \restr B$ for current with finite mass we can write for every Borel and bounded function $\psi \colon V \to \mathbb{R}$
\begin{align}
\int_V \psi(y) \, d\pi_V (\eta_V \restr B)(y) &= \int_B \psi(\pi_V(x)) \, d \eta_V(x) \\
&= \langle T, \mathbbm{1}_B \, \psi \circ \pi_V \, d\pi_V  \rangle\\
&=\langle T \restr B, \psi \circ \pi_V \, d\pi_V  \rangle \\
&=\langle \pi_{V\#}(T \restr B), \psi  \rangle\\
&= \int_V \psi(y) \, \theta_{\pi_V,B}(y) \, d\mathcal{H}^m(y),
\end{align}
where, in the last equality, we used \eqref{e:G-valued}. Thanks to the arbitrariness of $\psi$ we deduce
\begin{equation}
\label{e:localized-pushforward19}
\pi_V (\eta_V \restr B) = \theta_{\pi_V,B} \mathcal{H}^m, \ \ \text{ as measures in $V$.}
\end{equation}
Let $P_V,N_V\subset\mathbb R^n$ be Borel sets which form a Hahn decomposition for
$\eta_V$. Applying \eqref{e:localized-pushforward19} with
$B=P_V$ and $B=N_V$, we obtain \eqref{e:projected-muV-ac19}.

Hence, by setting $\mu_V := |\eta_V|$ for every $V \in \text{Gr}(n,m)$, we see that the family $(\mu_V)$ satisfies the hypothesis of Theorem \ref{t:mainthm2} with $\eta$ replaced by $\|T\|$. By applying the conclusion of Theorem \ref{t:mainthm2} we infer that, for $\gamma_{n,m}$-a.e. $V \in \text{Gr}(n,m)$
\begin{equation}
    \label{e:restgraphrect}
   f\colon V \to V^\bot \ \ \text{Borel measurable } \Rightarrow \mu\restr \text{graph}(f) \ \ \text{is $m$-rectifiable.}
\end{equation}

By a direct application of the disintegration theorem with map $\pi_V\colon \mathbb R^n\to V$, we find a unique (up to
$\mathcal H^m$-negligible sets) family of finite signed Radon measures $(\langle T,\pi_V,y\rangle)_{y\in V}$, where $\langle T,\pi_V,y\rangle$ is supported in $\pi^{-1}_{V}(y)$ for $\mathcal{H}^m$-almost every $y \in V$, 
such that:
\[
\int_{\mathbb{R}^n} \phi(x) \, \eta_V(x)
=
\int_V
\bigg(\int_{\mathbb{R}^n} \phi(x) \, d\langle T,\pi_V,y\rangle(x) \bigg)
\,d\mathcal H^m(y),
\]
whenever $\phi \colon \mathbb{R}^n \to \mathbb{R}$ is Borel measurable. By choosing $\phi= \mathbbm{1}_B\, \psi \circ \pi_V$, we can make use of \eqref{e:localized-pushforward19} to infer that 
\begin{align*}
 \int_V
\langle T,\pi_V,y\rangle(B) \, \psi(y)
\,d\mathcal H^m(y) &= \bigg(\int \phi(x) \, d\langle T ,\pi_V,y\rangle(x) \bigg) 
\,d\mathcal H^m(y) \\
&= \int_{\mathbb{R}^n} \phi(x) \, \eta_V(x) \\
&= \int_{\mathbb{R}^n} \psi(y) \, d\pi_V(\eta_V \restr B)(y)\\
&= \int_{\mathbb{R}^n} \psi(y) \, \theta_{\pi_V,B}(y) d\mathcal{H}^m(y),
\end{align*}
 whenever $\psi \colon V \to \mathbb{R}$ is Borel measurable. In particular, we deduce that for every Borel set $B \subset \mathbb{R}^n$
 \begin{equation}
 \label{e:newinformation}
   \langle T,\pi_V,y\rangle(B) = \theta_{\pi_V,B}(y), \ \ \text{ for $\mathcal{H}^m$-a.e. $y \in V$}. 
 \end{equation}
 By possibly passing to a countable basis $\mathcal{U}$ of open subsets of $\mathbb{R}^n$, we make use of \eqref{e:newinformation} and the assumption $\theta_{\pi_V,B} \in L^1(V;\mathbb{Z})$, to deduce that for $\mathcal{H}^m$-a.e. $y \in V$ 
 \begin{equation}
 \label{e:Gconstrained}
  \langle T,\pi_V,y\rangle(U) \in \mathbb{Z}, \ \ \text{ for every $U \in \mathcal{U}$}.
 \end{equation}
Since $G$ is discrete in $\mathbb{R}$, the condition \eqref{e:Gconstrained} forces $\langle T,\pi_V,y\rangle$ to be a linear $G$-combination of Dirac deltas for $\mathcal{H}^m$-a.e. $y \in V$.

The atomicity condition \eqref{e:restgraphrect} allows us to argue as in the proof of Theorem \ref{t:final} and infer that the measure $\mu_V$ is concentrated in a countable union of graphs of Borel measurable functions $f_i \colon V \to V^\bot$. This last information can be combined with property \eqref{e:restgraphrect} to deduce that $\mu_V$ is an $m$-rectifable measure for $\gamma_{n,m}$-a.e. $V \in \text{Gr}(n,m)$. 

Finally, since the set of admissible planes has full Grassmannian measure, one
can choose finitely many of them, $(V_1,\dotsc,V_N)$, whose associated
unit $m$-vectors span $\Lambda_m(\mathbb R^n)$. Consequently,
\[
\|T\|\leq C\sum_{i=1}^N |\mu_{V_i}|.
\]
Since each $\mu_{V_i}$ is $m$-rectifiable, the same is true of $\|T\|$. In particular, $\|T\| = \theta \, \mathcal{H}^m \restr R$ for some Borel measurable function $\theta \colon \mathbb{R}^n \to [0,\infty)$ and some countably $m$-rectifiable set $R \subset \mathbb{R}^n$.

\vspace{1mm}
\emph{Step 2: Alignment of the polar with the tangent.} In order to prove that $T$ is rectifiable $m$-current, it remains to verify that, by calling $\tau$ a Borel representative of the polar of $T$, then $\tau$ aligns with the approximate tangent of $\|T\|$ for $\|T\|$-a.e. $x \in \mathbb{R}^n$. Assume by contradiction that the $\|T\|$-measurable set
\[
N :=\big\{x \ | \ \text{Tan}(\|T\|,x) \text{ and } \tau(x) \text{ exist}, \ \text{span}\,\tau(x) \neq \text{Tan}(\|T\|,x) \big\},
\]
satisfies $\|T\|(N)>0$. Then, the $m$-current $T \restr N$ is non null and $\|T \restr N\|$ is $m$-rectifiable. From the very definition of $N$, we can find a unitary and simple Borel $m$-vector field $\tau' \colon N \to \Lambda_m(\mathbb{R}^n)$ such that 
\[
\text{span}\,\tau'(x) = \text{Tan}(\|T\|,x) \ \text{ and } \ \pm\tau'(x) \neq \tau(x), \ \ \text{ for $\|T\|$-a.e. $x \in N$}.
\]

For every $V \in \text{Gr}(n,m)$, we use the Area formula on $R$ together with the fact that $|\langle\tau',d\pi_V\rangle|$ must coincide with the tangential Jacobian of $\pi_V \restr R$ for $\mathcal{H}^m$-a.e. $x \in N$, to write
\begin{align*}
   \int_V \bigg(&\int_{\pi_V^{-1}(y) \cap N} f(x) \frac{\langle \tau(x),d\pi_V \rangle}{\langle \tau'(x),d\pi_V \rangle} \, \theta(x) \, d \mathcal{H}^0(x)\bigg)d\mathcal{H}^m(y) \\
   &=\int_N f(x) \, \langle \tau(x),d\pi_V(x) \rangle\, \theta(x) \, d\mathcal{H}^m(x)\\
   &= \int_V \bigg(\int_{\pi_V^{-1}(y)} f(x)  \, \theta(x) \, d \langle T,\pi_V,y \rangle(x)\bigg)d\mathcal{H}^m(y),
\end{align*}
  for every Borel function $f \colon \mathbb{R}^n \to \mathbb{R}$. Notice that the quotient $\langle \tau(x),d\pi_V \rangle/\langle \tau'(x),d\pi_V \rangle$ does not depend on the orientation of $V$ and that the above equality is well defined for those $V \in \text{Gr}(n,m)$ satisfying 
  \[
  \langle \tau'(x),d\pi_V \rangle \neq 0, \ \ \text{ for $\mathcal{H}^m$-a.e. $x \in N$},
  \]
  and this holds for $\gamma_{n,m}$-a.e. $V \in \text{Gr}(n,m)$. Hence, we obtain for $\gamma_{n,m}$-a.e. $V \in \text{Gr}(n,m)$
  \begin{equation}
  \label{e:nonso}
     \frac{\langle \tau,d\pi_V \rangle}{\langle \tau',d\pi_V \rangle} \,  \theta \, \mathcal{H}^0 \restr [\pi_V^{-1}(y) \cap N ]= \theta \, \langle T,d\pi_V,y\rangle, \ \ \text{ for $\mathcal{H}^m$-a.e. $y \in V$}.
  \end{equation}
By exploiting \eqref{e:Gconstrained} we infer from \eqref{e:nonso} that for  $\gamma_{n,m}$-a.e. $V \in \text{Gr}(n,m)$
\begin{equation}
\label{e:nonso1}
    \frac{\langle \tau(x),d\pi_V \rangle}{\langle \tau'(x),d\pi_V \rangle} \in \mathbb{Z}, \ \ \text{ for $\mathcal{H}^m$-a.e. $x \in N$}.
\end{equation}
 Since the condition in \eqref{e:nonso1} defines a Borel subsets in the product space $\mathbb{R}^n \times \text{Gr}(n,m)$, we can apply Fubini's theorem with product measure $(\mathcal{H}^m \restr R) \otimes \gamma_{n,m}$ to infer from \eqref{e:nonso1} that for $\mathcal{H}^m$-a.e. $x \in N$
 \begin{equation}
\label{e:nonso2}
   \frac{\langle \tau(x),d\pi_V \rangle}{\langle \tau'(x),d\pi_V \rangle} \in \mathbb{Z}, \ \ \text{ for $\gamma_{n,m}$-a.e. $V \in \text{Gr}(n,m)$}.
\end{equation}

In order to prove that \eqref{e:nonso2} leads to a contradiction, we define $F \colon U \to \mathbb{R}$ where $U:= \{ V\in\text{Gr}(n,m) \ | \  \langle\tau',d\pi_V\rangle\neq0 \}$ and
\[
F(V):= \frac{\langle \tau(x),d\pi_V \rangle}{\langle \tau'(x),d\pi_V \rangle}.
\]
We first observe that $U$ is a nonempty open subset of $\text{Gr}(n,m)$ with full $\gamma_{n,m}$-measure. Since a discrete subset of $\mathbb{R}$ is at most countable, the desired contradiction can be obtained by showing that every level set of $F$ has zero Grassmannian measure. 

Fix $c\in\mathbb R$. 
Then 
\[ 
F^{-1}(c) = \{ V\in U: \langle\tau-c\tau',d\pi_V\rangle=0 \}. 
\] 
Since $\tau$ and $\tau'$ are unit and $\tau\neq\pm\tau'$, one has 
\[ 
\tau-c\tau'\neq 0, \ \ \text{ for every }c\in\mathbb R. 
\] 
 For $V\in\text{Gr}(n,m)$, let $\xi_V$ be either of the two unit simple $m$-vectors spanning $V$, and consider 
\[ 
\Phi_c(V):= \bigl|\langle\tau-c\tau',\xi_V\rangle\bigr|^2.
\] 
This is a well-defined real-analytic function on $\text{Gr}(n,m)$. Indeed, by the Pl\"ucker embedding it is a quadratic polynomial in the Pl\"ucker coordinates. Moreover, $\Phi_c$ is not identically zero. Otherwise,
\[
\langle\tau-c\tau',\xi_V\rangle=0, \ \ \text{ for every }V\in\text{Gr}(n,m). 
\] 
Since the simple $m$-vectors span $\Lambda_m(\mathbb R^n)$, this would imply
\[
\tau-c\tau'=0,
\] 
which is impossible. The zero set of a nontrivial real-analytic function on a connected open set has zero $\gamma_{n,m}$-measure (see \cite{myt}). Since $\text{Gr}(n,m)$ is a continuous image of the special orthogonal group $\text{SO}(n)$, it is connected as well. Hence 
\[
\gamma_{n,m}(F^{-1}(c))=0, \ \ \text{ for every }c\in\mathbb R.
\] 
This concludes the proof.
\end{proof}

\vspace{3mm}

We now justify rigorously the construction of slices outlined in the introduction.

Let $T\in\mathbf M_m(\mathbb R^n)$, Fix a Lipschitz map
\[
\varphi \colon \mathbb R^n \to \mathbb{R}^m. 
\]

Since $T$ has finite mass,
the functional
\[
f\longmapsto
  T(f,\varphi_1, \dotsc, \varphi_m),
\qquad
f\in C_c^\infty(\mathbb R^n),
\]
extends uniquely to a finite signed Radon measure on $\mathbb R^n$,
which we denote by $\eta_\varphi$.

We have
\begin{equation}
\label{e:projected-muV-ac}
\varphi \, |\eta_\varphi|
\ll
\mathcal H^m\restr V.
\end{equation}
This absolute continuity property is a well-known consequence of the result of De Philippis and Rindler~\cite{dr}. Since we could not find a direct reference for the precise formulation needed here, we include a short argument.

Let $B\subset\mathbb R^n$ be a Borel set and consider the
top-dimensional metric current
\[
S_{\varphi,B}:=\varphi_{\#}(T\restr B)\in\mathbf M_m(\mathbb{R}^m).
\]
The result in \cite[Theorem~1.15]{dr} yields
\[
\|S_{\varphi,B}\|
\ll
\mathcal{L}^m.
\]
In particular, $S_{\varphi,B}=\llbracket g_{\varphi,B}\rrbracket$ for some $g_{\varphi,B}\in L^1(\mathbb{R}^m)$.
We next identify the coefficient measure of $S_{\varphi,B}$. For every
$\psi\in C_c^\infty(\mathbb{R}^m)$, we have
\begin{align*}
\int_{\mathbb{R}^m}\psi(y)g_{\varphi,B}(y)\,dy
&=(T\restr B)
  (\psi\circ\varphi ,\varphi^1,\dotsc,\varphi^m)
\\
&=
\int_B\psi(\varphi(x))\,d\eta_\varphi(x)
\\
&=
\int_{\mathbb{R}^m}\psi(y)\,
d\bigl[\varphi(\eta_\varphi\restr B)\bigr](y).
\end{align*}
It follows that
\begin{equation}
\label{e:localized-pushforward}
\varphi \, (\eta_\varphi\restr B)
\ll
\mathcal{L}^m
\qquad
\text{for every Borel set }B\subset\mathbb R^n.
\end{equation}

Let $P_\varphi,N_\varphi\subset\mathbb R^n$ be a Hahn decomposition for
$\eta_\varphi$. Applying \eqref{e:localized-pushforward} with
$A=P_\varphi$ and $A=N_\varphi$, we obtain \eqref{e:projected-muV-ac}.

By a direct application of the disintegration theorem with map $\varphi\colon \mathbb R^n\to \mathbb{R}^m$, we find a unique (up to
$\mathcal L^m$-negligible sets) family of finite signed Radon measures $(\langle T,\varphi,y\rangle)_{y\in \mathbb{R}^m}$, where $\langle T,\varphi,y\rangle$ is supported in $\varphi^{-1}(y)$ for $\mathcal{L}^m$-almost every $y \in \mathbb{R}^m$, 
such that:
\[
\int_{\mathbb{R}^n} \phi(x) \, d\eta_\varphi(x) 
=
\int_{\mathbb{R}^m}
\bigg( \int_{\mathbb{R}^n} \phi(x) \, d\langle T,\varphi,y\rangle(x) \bigg)
d\mathcal L^m(y),
\]
whenever $\phi \colon \mathbb{R}^n \to \mathbb{R}$ is a Borel measurable function. 
By recalling the definition of $\eta_\varphi$, we can thus write
\begin{equation}
\label{e:slicing-normal23}
T(\phi,\varphi_1, \dotsc, \varphi_m)
=
\int_{\mathbb{R}^m}
\bigg( \int_{\mathbb{R}^n} \phi(x) \, d\langle T,\varphi,y\rangle(x) \bigg)
d\mathcal L^m(y),
\end{equation}
whenever $\phi \colon \mathbb{R}^n \to \mathbb{R}$ is a Borel measurable function. 

The slicing map introduced above inherits from the disintegration
theorem the following properties. For $\mathcal L^m$-almost every $y\in\mathbb R^m$,
$\langle T,\varphi,y\rangle$ is a finite signed measure concentrated on
\[
\varphi^{-1}(y)\cap\text{spt}T.
\]
In particular, it can be canonically identified with a
$0$-dimensional metric current.

Moreover, 
\begin{equation}
\label{e:only2}
\int_{\mathbb R^m}
\mathbf M\bigl(\langle T,\varphi,y\rangle\bigr)\,
d\mathcal L^m(y)
\leq
\prod_{i=1}^m\text{Lip}(\varphi_i)\,\mathbf M(T).
\end{equation}

The construction is linear in $T$ and is compatible with
restrictions: for every bounded Borel function $\phi$,
\begin{equation}
\langle T\restr \phi,\varphi,y\rangle
=
\langle T,\varphi,y\rangle\restr \phi
\qquad\text{for $\mathcal L^m$-a.e. }y.
\end{equation}
It is also compatible with Lipschitz push-forwards: if
$F:\mathbb R^n\to\mathbb R^N$ and
$q:\mathbb R^N\to\mathbb R^m$ are Lipschitz, then
\begin{equation}
\langle F_\#T,q,y\rangle
=
F_\#\langle T,q\circ F,y\rangle
\qquad\text{for $\mathcal L^m$-a.e. }y.
\end{equation}

Finally, whenever $T$ is a normal metric current, the classical
Ambrosio-Kirchheim slices satisfy the same disintegration identity.
The uniqueness of disintegration therefore implies that the slicing
map introduced here coincides, up to an
$\mathcal L^m$-negligible set of parameters, with the classical one.

\begin{proof}[Proof of Theorem \ref{t:rectmetric}]
   By arguing as in the proof of Theorem \ref{t:Grect}, we fix for each $V \in \Sigma$ an orientation, so that the above slicing construction is well defined for the maps $\pi_V \colon \mathbb{R}^n \to V$. Notice once again that, reversing the orientation replaces both $T \restr d\pi_V$ and the associated slices by their negatives; hence properties such as atomicity are independent of this choice.
    
    We start by proving that the mass measure $\|T\|$ is $m$-rectifiable. For every $V \in \Sigma$ we  consider
    \[
    \mu_V := \int_{\mathbb{R}^m} |\langle T,\pi_V,y\rangle| \, d\mathcal{L}^m.
    \]
   By virtue of properties \eqref{e:slicing-normal23}-\eqref{e:only2}, $\mu_V$ is a finite positive Borel measure for
$\gamma_{n,m}$-almost every $V\in\Sigma$. Applying
Theorem~\ref{t:rectcri} with $\eta=\|T\|$, we infer that, for
$\gamma_{n,m}$-almost every $V\in\Sigma$, the restricted measure
$\mu_V\restr\text{graph}(f)$ is $m$-rectifiable whenever
$f\colon V\to V^\bot$ is Borel measurable. Hence, by exploiting the atomicity
assumption~\eqref{e:rectmetric} we can argue exactly as in the proof of Theorem \ref{t:Grect} to infer that $\|T\|$ is $m$-rectifiable. 

Finally, the parametric representation \cite[Theorem 4.5]{ak1} yields a sequence of compact sets $K_i$, functions $\theta_i \in L^1(\mathbb{R}^m)$ with $\text{supp} \, \theta_i \subset K_i$, and bi-lipschitz maps $f_i \colon K_i \to \mathbb{R}^n$, such that
\[
T = \sum_i f_{i\#} \llbracket \theta_i \rrbracket \quad \text{ and } \quad \mathbf{M}(T) = \sum_i \mathbf{M}(f_{i\#} \llbracket \theta_i \rrbracket).
\]

This representation immediately implies the corresponding representation for the associated Federer-Fleming current $\tilde{T}$ and hence $\tilde{T}$ is a rectifiable current \cite[Theorem 4.1.28]{fed1}.

\end{proof}

\appendix
\section{Proof of Proposition \ref{p:equivcone}}

To this purpose we need the following lemma (see \cite[Lemma 7.6]{per}).

\begin{lemma}
    \label{l:ps}
Let $U \subset \mathbb{R}^m$ be open and let $T \in C^1(U;\mathbb{R}^m)$ be such that 
$\|DT - \text{Id}\|_\infty \leq \tfrac{1}{2}$ on $B_r(\lambda_0) \subset U$ for some $r > 0$. 
Then 
\[
T : B_{r/3}(\lambda_0) \to T(B_{r/3}(\lambda_0))
\]
is a diffeomorphism, and 
\[
 B_{\rho/2}(T(\lambda_0)) \subset T(B_\rho(\lambda_0))
\quad \text{for every } \rho \in (0,r].
\]
\end{lemma}

We are now in position to prove Proposition \ref{p:equivcone}.

\begin{proof}[Proof of Proposition \ref{p:equivcone}]

 We start by proving the proposition with a weaker property than \eqref{e:equivcone1234}, namely, 
 \begin{equation}
 \label{e:equicone}
 \Phi_{V,\lambda}^{-1}\big(\mathbb{B}_{s/c}(\lambda)\big) \subset X(x,\lambda,s) \cap U(V) \subset  \Phi_{V,\lambda}^{-1}\big(\mathbb{B}_{cs}(\lambda)\big).
 \end{equation}
 Notice that the difference from the statement of the proposition is that, in condition \eqref{e:equicone}, we allow the map $\lambda \longmapsto \Phi_{V,\lambda}$ to depend fully on $\lambda \in \Lambda$, and not only on the orthogonal complement $V^\perp$ of $V$.

 \vspace{2mm}
 
\emph{-Step 1: Proof of \eqref{e:curvradpro1}}. We start by defining the Borel set $U(V)$. In view of property \eqref{e:h2}, we apply the Cauchy-Binet formula to compute the quantity $|J_\lambda T_{xx'}(\lambda_0)|$, to infer that for every $x' \in X(x,\lambda_0,C')$ there exists a coordinate $m$-plane $V$ such that
\begin{equation}
\label{e:opencon1}
|\text{det} D_\lambda (T_{xx'} \restr (V + \lambda_0))(\lambda_0)| > \frac{C'}{2c_{l,m}}
\end{equation}
for a dimensional constant $c_{l,m}>0$. Therefore, by setting
\[
U(V):= \{x' \in X(x,\lambda_0,C') \ | \ \eqref{e:opencon1} \text{ holds}  \}, \ \ \text{ for every $V \in \mathcal{V}_{l,m}$},
\]
we immediately deduce that 
\begin{equation}
\label{e:coverequicone}
X(x,\lambda_0,C') \subset \bigcup_{V \in \mathcal{V}_{l,m}} U(V).
\end{equation}
  The condition in \eqref{e:curvradpro1.97} is then verified by choosing $s_0 \leq C'/2L$, due to the following inclusion
  \begin{equation}
\label{e:cont9}
X(x, \lambda, s_0) \subset X(x, \lambda_0, C'), \ \ \text{for every $\lambda \in \mathbb{B}_{s_0}(\lambda_0)$}.
\end{equation}

In addition, we exploit \eqref{e:h3} to find $s'_0 \in (0,\min\{\delta/2,C'\})$ such that, for every $V \in \mathcal{V}_{l,m}$, and for every $x' \in X(x,\lambda_0, C') \cap U(V)$, it holds 
\begin{equation}
\label{e:opencon}
|\text{det} D_\lambda (T_{xx'} \restr (V + \lambda))(\lambda)| > \frac{C'}{4c_{l,m}}, \ \ \text{ for every $\lambda \in \mathbb{B}_{s'_0}(\lambda_0)$}.
\end{equation}

In order to construct the map $\Phi_{V,\lambda}$ fulfilling \eqref{e:curvradpro1}, fix $V \in \mathcal{V}_{l,m}$. We claim that there exist $s_0 \in (0,s'_0)$ and a constant $d >0$, such that for any fixed $\lambda \in \mathbb{B}_{s_0}(\lambda_0)$ the following statement holds true: each $x' \in X(x, \lambda, s_0/24d) \cap U(V)$ corresponds to a unique $\lambda_{x'} \in \mathbb{B}_{s_0}(\lambda) \cap (V + \lambda)$ satisfying
\[
T_{xx'}(\lambda_{x'}) = 0.
\]
To prove the claim, we verify the following assertion: there exists $s_0 \in (0,\min \{s'_0, C'/2L\})$ such that for every $\lambda \in \mathbb{B}_{s_0}(\lambda_0)$ and $x' \in X(x,\lambda,s_0/24d) \cap U(V)$ we have

\begin{equation}
\label{e:prop(1)}
\text{$T_{xx'} \restr [\mathbb{B}_{2s_0}(\lambda) \cap (V + \lambda)]$ is a $C^1$-diffeomorphism with its image}
\end{equation}
and
\begin{equation}
\label{e:prop(2)}
0 \in \text{Im}\big( T_{xx'} \restr [\mathbb{B}_{s_0}(\lambda) \cap (V + \lambda)]\big).
\end{equation}

\vspace{2mm}

Indeed, for every $\lambda \in \mathbb{B}_{\delta/2}(\lambda_0)$ and every $x' \in X(x,\lambda_0,C') \cap U(V)$, we can define the map $\tilde{T} \colon V+\lambda \to \mathbb{R}^m$ as 
\[
\tilde{T}(\lambda'):= A^{-1}_{xx'}(\lambda) \, T_{xx'}(\lambda'),
\]
where $A_{xx'}(\lambda):=D_\lambda (T_{xx'} \restr V + \lambda)(\lambda)$.
If 
\[
s \in (0, \min\{s'_0,C'/8c_{l,m}C''\})
\]
where $C''$ is the constant appearing in \eqref{e:h3}, then, for every $\lambda \in \mathbb{B}_{\delta/2}(\lambda_0)$, for every $\lambda' \in \mathbb{B}_s(\lambda) \cap (V + \lambda)$ and for every $x' \in X(x,\lambda_0,C') \cap U(V)$, it holds
\[
|D_\lambda \tilde{T}(\lambda') - \text{Id}| = |D_\lambda \tilde{T}(\lambda') - D_\lambda \tilde{T}(\lambda)| \leq \text{Lip}(\tilde{T})|\lambda - \lambda'| \leq \frac{1}{2},
\]
where we used $\text{Lip} (\tilde{T}) \leq 4c_{l,m}C''/C'$. Therefore, Lemma \ref{l:ps} tells us that, if we let 
\[
s''_0 := \frac{1}{2} \min\{s'_0,C'/8c_{l,m}C''\},
\]
then, for every $\lambda \in \mathbb{B}_{s''_0}(\lambda_0) \subset \mathbb{B}_{\delta/2}(\lambda_0)$ and for every $x' \in X(x,\lambda_0,C') \cap U(V)$, it holds
\begin{equation}
\label{e:ps1}
\tilde{T} \colon B_{s''_0/3}(\lambda) \to \tilde{T}( B_{s''_0/3}(\lambda))
\end{equation}
is a $C^1$-diffeomorphism and
\begin{equation}
\label{e:ps2}
B_{\rho/2}(\tilde{T}(\lambda)) \subset \tilde{T}(B_\rho(\lambda)), \ \ \text{ for every $\rho \in (0,s''_0]$}.
\end{equation}

 Moreover, by exploiting \eqref{e:ps2} and the very definition of the map $\tilde{T}$, we infer that, for every $\lambda' \in \mathbb{B}_{s''_0}(\lambda) \cap (V + \lambda)$ and every $x' \in X(x,\lambda_0,C') \cap U(V)$, it holds
\begin{equation}
    \label{e:ps3}
    A_{xx'}(\lambda)B_{\rho/2}(\tilde{T}(\lambda)) \subset T_{xx'}(B_\rho(\lambda)), \ \ \text{ for every $\rho \in (0,s''_0]$}.
\end{equation}
But since \eqref{e:h2} holds, we have $B_{d \rho}(T_{xx'}(\lambda)) \subset A_{xx'}(\lambda)B_{\rho/2}(\tilde{T}(\lambda))$ for every $\rho \in (0,s_0)$, for a constant $d$ that only depends on $C'$ and $C''$. Therefore, we infer that, for every $\lambda \in \mathbb{B}_{s''_0}(\lambda_0)$, for every $\lambda' \in \mathbb{B}_{s''_0}(\lambda) \cap (V + \lambda)$, and for every $x' \in X(x,\lambda_0,C') \cap U(V)$, it holds
\begin{equation}
    \label{e:ps4}
    B_{\rho/d}(T_{xx'}(\lambda)) \subset T_{xx'}(B_\rho(\lambda)), \ \ \text{ for every $\rho \in (0,s''_0]$}.
\end{equation}

Finally, by setting $s_0:= s''_0/6$, and by recalling that $s''_0 \leq s'_0 \leq C'/2L$, we can make use of \eqref{e:cont9} to infer, from property \eqref{e:ps1}, the validity of \eqref{e:prop(1)}. Moreover, property \eqref{e:prop(2)} follows by virtue of \eqref{e:ps4} (recall again \eqref{e:cont9}) with the choice $\rho= s_0:= s''_0/6$. The claim is thus proved

For any fixed $\lambda \in \mathbb{B}_{s_0}(\lambda_0)$ we now turn to define 
\[
\Phi_{V,\lambda} \colon X(x,\lambda,s_0/24d) \cap U(V) \to \mathbb{B}_{s_0}(\lambda) \cap (V + \lambda)  \quad \text{ as } \quad \Phi_{V,\lambda}(x'):= \lambda_{x'}.
\]
By construction \eqref{e:curvradpro1} is satisfied. 

\vspace{2mm}

\emph{-Step 2: Proof of \eqref{e:equicone}.} Let $\lambda \in \mathbb{B}_{s_0}(\lambda_0)$, let $V \in \mathcal{V}_{l,m}$, and let $s \in (0,s_0/24d)$. 

To prove the first inclusion assume that $x' \in \Phi_{V,\lambda}^{-1}(\mathbb{B}_{s/L}(\lambda))$. Clearly $x' \in U(V)$. Moreover, from the very definition of the map $\Phi_{V,\lambda}$ it follows that $T_{xx'}(\Phi_{V,\lambda}(x'))=0$. Hence, we can estimate
\[
|T_{xx'}(\lambda)| = |T_{xx'}(\lambda) - T_{xx'}(\Phi_{V,\lambda}(x'))| \leq L |\lambda - \Phi_{V,\lambda}(x')| < L \frac{s}{L} < s,
\]
which immediately implies that $x' \in X(x,\lambda,s)$.  This proves the first inclusion with $c = L$

To prove the second inclusion in \eqref{e:equicone}, assume that $x' \in  X(x,\lambda,s) \cap U(V)$. Then, we already know from \eqref{e:cont9} that, with the choices of $\lambda$ and $s$ at the beginning of Step 2 it holds $X(x,\lambda,s) \subset X(x,\lambda_0,C')$. Therefore, we have that $T_{xx'}(\Phi_{V,\lambda}(x'))=0$ and $\Phi_{V,\lambda}(x') \in \mathbb{B}_{s_0}(\lambda_0)$ by construction. Hence, we can make use of property \eqref{e:ps4} with $\rho = 24ds$ to infer that there exists $\tilde{\lambda} \in \mathbb{B}_{2ds}(\lambda) \cap (V + \lambda)$ such that $T_{xx'}(\tilde{\lambda})=0$. But recall that $T_{xx'} \restr [\mathbb{B}_{s_0}(\lambda) \cap (V + \lambda)]$ is a diffeomorphism with its image. Therefore, we deduce that $\Phi_{V,\lambda}(x') = \tilde{\lambda}$, and, in particular, $\Phi_{V,\lambda}(x') \in \mathbb{B}_{24ds}(\lambda) \cap (V + \lambda)$. This proves also the second inclusion with $c= 24d$. Condition \eqref{e:equicone} follows by setting
\begin{equation}
\label{e:defc}
c:= \max \{L,48d\}.
\end{equation}

\vspace{2mm}

\emph{-Step 3: Proof of \eqref{e:equivcone1234}.} Now we conclude the proof of the proposition. Notice that, if $\lambda \in \mathbb{B}_{s_0/cL}(\lambda_0)$ and $x' \in X(x,\lambda,s_0/c)$, then 
\[
|T_{xx'}(\lambda_0 + \lambda^\bot)| \leq L |\lambda - \lambda_0| + T_{xx'}(\lambda) <  \frac{2s_0}{c}.
\]
 As a consequence, for every $\lambda \in \mathbb{B}_{s_0/cL}(\lambda_0)$ we have
\[
X(x,\lambda,s_0/c) \subset X(x,\lambda_0 + \lambda^\bot,2s_0/c).
\]

Therefore, for any given $\lambda \in \mathbb{B}_{s_0/cL}(\lambda_0)$, the map $\Phi_{V,\lambda_0+\lambda^\bot}$ is well defined on the set $X(x,\lambda,s_0/c) \cap U(V)$ and 
\begin{equation}
\label{e:equivcone123456}
T_{xx'}(\Phi_{V,\lambda}(x'))=T_{xx'}(\Phi_{V,\lambda_0+\lambda^\bot}(x'))=0, \ \ \text{for every $x' \in X(x,\lambda,s_0/c) \cap U(V)$}.
\end{equation}
In addition recall that, by construction, $\Phi_{V,\lambda_0+\lambda^\bot}$ takes values in $\mathbb{B}_{s_0}(\lambda_0 + \lambda^\bot) \cap (V+\lambda)$ (we used that $V +\lambda = V + \lambda_0 + \lambda^\bot$). Moreover, by using property \eqref{e:ps4} (recall also that $c >d$ by \eqref{e:defc}), the condition $x' \in X(x,\lambda,s_0/c) \cap U(V)$ implies $\Phi_{V,\lambda}(x') \in \mathbb{B}_{s_0}(\lambda) \cap (V + \lambda)$. This means that, if $\lambda \in \mathbb{B}_{s_0/cL}(\lambda_0)$ then (recall that $L \geq 1$ and $c \geq 2$)
\[
\Phi_{V,\lambda}(x') \in \mathbb{B}_{s_0}(\lambda_0 + \lambda^\bot) \cap (V+\lambda), \ \ \text{ for every $x' \in X(x,\lambda,s_0/c) \cap U(V)$}.
\]
But since property \eqref{e:prop(1)} tells us that $T_{xx'} \restr [\mathbb{B}_{2s_0}(\lambda_0 + \lambda^\bot) \cap (V + \lambda)]$ is a diffeomorphism with its image, we finally deduce that, for every $\lambda \in \mathbb{B}_{s_0/cL}(\lambda_0)$, it holds
\begin{align}
\label{e:equicone10}
&\Phi_{V,\lambda}(x')=\Phi_{V,\lambda_0+\lambda^\bot}(x'), \ \ \text{for every $x' \in X(x,\lambda,s_0/c) \cap U(V)$}.
\end{align}
Hence, from \eqref{e:equicone} and \eqref{e:equicone10} we deduce for every $\lambda \in \mathbb{B}_{s_0/2cL}(\lambda_0)$ and every $s \in (0,s_0/2c)$ it holds
\begin{align}
  \Phi^{-1}_{V,\lambda_0 + \lambda^\bot}(\mathbb{B}_{s/c}(\lambda))  &=  X(x,\lambda,s_0/c) \cap \Phi^{-1}_{V,\lambda_0 + \lambda^\bot}(\mathbb{B}_{s/c}(\lambda))  \\
   &= X(x,\lambda,s_0/c)  \cap \Phi^{-1}_{V,\lambda}(\mathbb{B}_{s/c}(\lambda)) \\
   &\subset  \Phi^{-1}_{V,\lambda}(\mathbb{B}_{s/c}(\lambda)) \subset X(x,\lambda,s) \cap U(V),
\end{align}
where, in the first equality, we deduced from \eqref{e:equicone} that $\Phi^{-1}_{V,\lambda_0 + \lambda^\bot}(\mathbb{B}_{s/c}(\lambda)) \subset X(x,\lambda,s_0/c)$ provided $\lambda \in \mathbb{B}_{s_0/2cL}(\lambda_0)$ and $s \in (0,s_0/2c)$. Finally, from \eqref{e:equicone} and \eqref{e:equicone10} we also deduce that for every $\lambda \in \mathbb{B}_{s_0/2cL}(\lambda_0)$ and every $s \in (0,s_0/2c)$ it holds  
\begin{align}
    X(x,\lambda,s) \cap U(V) &\subset   X(x,\lambda,s_0/c) \cap \Phi^{-1}_{V,\lambda}(\mathbb{B}_{cs}(\lambda))\\ 
    &=  X(x,\lambda,s_0/c)  \cap \Phi^{-1}_{V,\lambda_0 + \lambda^\bot}(\mathbb{B}_{cs}(\lambda)) \\
    &\subset  \Phi^{-1}_{V,\lambda_0 + \lambda^\bot}(\mathbb{B}_{cs}(\lambda)).
\end{align}
The property \eqref{e:equivcone1234} is proved.

\vspace{2mm}

\emph{-Step 4: Proof of \eqref{e:contPsi}.} Assume that $(\lambda_i) \subset \mathbb{B}_{s_0}(\lambda_0)$ and $(x_i) \subset X$ are sequences such that $x_i \in X(x,\lambda_i,s_0/c) \cap U(V)$ for every $i=1,2,\dotsc$, and such that $(x_i,\lambda_i) \to (\tilde{x},\tilde{\lambda})$ for some $\tilde{\lambda} \in \mathbb{B}_{s_0}(\lambda_0)$ and $\tilde{x} \in X(x,\tilde{\lambda},s_0/c) \cap U(V)$ as $i \to \infty$. Since $\Phi_{V,\lambda_i}(x_i) \in B_{s_0}(\lambda_i) \cap (V + \lambda_i) \subset \mathbb{B}_{2s_0}(\lambda_0)$ for every $i=1,2,\dotsc$, up to pass to a not relabeled subsequence, we may assume that $\Phi_{V,\lambda_i}(x_i) \to \lambda' \in B_{2s_0}(\tilde{\lambda}) \cap (V + \tilde{\lambda})$ as $i \to \infty$. By using the continuity of the map $\Pi \colon X \times \Lambda \to \mathbb{R}^m$, we write
\[
0 = \lim_{i \to \infty} T_{xx_i}(\Phi_{V,\lambda_i}(x_i))= T_{x\tilde{x}}(\lambda'),
\]
from which we deduce $\tilde{\lambda}= \Phi_{V,\tilde{\lambda}}(\tilde{x})$ thanks to property \eqref{e:prop(1)} and \eqref{e:curvradpro1}. Property \eqref{e:contPsi} is proved.
\end{proof}

\section{Measurability}

\subsubsection{Measurability of the conditions in Proposition \ref{p:disretr1}.}
\label{a:measprod}
 The map
\begin{equation}
\label{e:appendix2}
(x,\lambda) \longmapsto \mu(X(x,\lambda,s))
\end{equation}
is $(\mu \otimes \mathcal{L}^l)$-measurable for every $s \in (0,1)$. Indeed, take $\varphi_\delta \in C^0_c((0,1)) $ such that $\varphi_\delta \equiv 1$ on $[-\delta,\delta]$ for some $\delta \in (0,1)$. Define the functions $f_\delta \colon X \times \Lambda \times X \setminus \{x=x'\}  \to [0,\infty)$ as
\[
f_\delta(x,\lambda,x'):= \varphi_\delta(T_{xx'}(\lambda)/s),
\]
where $T_{xx'}$ is defined in \eqref{e:deftxx'}.
Clearly, $f_\delta(x,\lambda,\cdot) =1$ on $X(x,\lambda,s\delta)$ and $f_\delta(x,\lambda,\cdot) =0$ on $X \setminus X(x,\lambda,s)$. while. Moreover, since the continuity assumption made on $\Pi \colon X \times \Lambda \to \mathbb{R}^m$ implies the continuity of $f_\delta$, we can apply Lebesgue's dominated convergence theorem to infer
\[
\lim_{i \to \infty} \int_X f_\delta(x_i,\lambda_i,x') \bigg[1 \wedge \frac{d(x',B_\delta(x_i))}{\delta}\bigg]\, d\mu(x') =  \int_X f_\delta(x,\lambda,x') \bigg[1 \wedge \frac{d(x',B_\delta(x))}{\delta}\bigg]\, d\mu(x'),
\]
whenever $(x_i,\lambda_i) \to (x,\lambda)$ as $i \to \infty$. Therefore, the monotonicity property of measures gives 
\[
\int_X f_\delta(x,\lambda,x') \bigg[1 \wedge \frac{d(x',B_\delta(x))}{\delta}\bigg]\, d\mu(x') \nearrow \mu(X(x,\lambda,s)), \ \ \text{ for every $(x,\lambda)$ as $\delta \searrow 0$},
\]
and we deduce the desired measurability of the map in \eqref{e:appendix2} for every $s \in (0,1)$.

In addition, we can make use once again of the monotonicity property of measures to infer that
\[
 \sup_{s \in (0,1)} s^{-m}\mu(X(x,\lambda,s))
= \sup_{s \in \mathbb{Q} \cap (0,1)} s^{-m}\mu(X(x,\lambda,s)).
\]
Therefore, the map
\[
\Psi(x,\lambda) := \sup_{s \in (0,1)} s^{-m}\mu(X(x,\lambda,s))
\]
is $(\mu \otimes \mathcal{L}^l)$-measurable. From the proof of Proposition~\ref{p:disretr1}, we know that every pair $(x,\lambda)$ with $\Psi(x,\lambda) < \infty$ also satisfies \eqref{e:locweak} and \eqref{e:equ21}. Moreover, Proposition~\ref{p:fcv1} ensures that $\Psi(x,\lambda) < \infty$ holds for $(\mu \otimes \mathcal{L}^l)$-a.e.\ $(x,\lambda)$. Consequently, the set of points $(x,\lambda)$ satisfying \eqref{e:locweak} and \eqref{e:equ21} is $(\mu \otimes \mathcal{L}^l)$-measurable.

\subsubsection{Measurability of the condition in Lemma \ref{l:keyalternative1}.}
\label{a:measkeyalt}

To prove the $(\mu \otimes \mathcal{L}^l)$-measurability of the sets of points $(x,\lambda) \in X \times \Lambda$ satisfying condition \eqref{e:3limit1}, we observe that the map
    \[
    \rho \longmapsto \mu(X(x,\lambda,s,\rho)) 
    \]
    is right continuous, for every $(x,\lambda,s)$, and that, thanks to the continuity of $\Pi \colon X \times \Lambda \to \mathbb{R}^m$, the map
    \[
    s \longmapsto \mu(X(x,\lambda,s,\rho))
    \]
    is lower semi-continuous, for every $(x,\lambda,\rho)$. This implies that
    \[
   (x,\lambda) \longmapsto \sup_{0<\rho<\delta} (s\rho)^{-m}\mu(X(x,\lambda,s,\rho)) =  \sup_{\rho \in (0,\delta) \cap \mathbb{Q}} (s\rho)^{-m}\mu(X(x,\lambda,s,\rho))
    \]
    is $(\mu \otimes \mathcal{L}^l)$-measurable, for every $(s,\delta)$, and that the map
   \[
   s \longmapsto \sup_{0<\rho<\delta} (s\rho)^{-m}\mu(X(x,\lambda,s,\rho)) 
    \] 
    is lower semi-continuous, for every $(x,\lambda,\delta)$. Hence, by exploiting the lower semi-continuity, given $s \in (0,1)$, we find $(q_k) \subset \mathbb{Q} \cap (0,1) $ such that $q_k \to s$ and hence 
    \begin{align*}
       \liminf_{k} \sup_{0<\rho<\delta} (q_k\rho)^{-m}\mu(X(x,\lambda,q_k,\rho)) \geq  \sup_{0<\rho<\delta} (s\rho)^{-m}\mu(X(x,\lambda,s,\rho)).
    \end{align*}
    But this immediately implies, for every $k=1,2,\dotsc$., that 
    \[
   \sup_{0 < s < 1/k} \sup_{0<\rho<\delta} (s\rho)^{-m}\mu(X(x,\lambda,s,\rho)) =\sup_{s \in (0,1/k) \cap \mathbb{Q}} \sup_{0<\rho<\delta} (s\rho)^{-m}\mu(X(x,\lambda,s,\rho)),
    \]
    for every couple of points $(x,\lambda)$. Therefore, by virtue of Appendix \ref{a:measprod}, we deduce that
    \[
    \Psi_{k,\delta}(x,\lambda) := \sup_{0 < s < 1/k} \sup_{0<\rho<\delta} (s\rho)^{-m}\mu(X(x,\lambda,s,\rho)),
    \]
    is $(\mu \otimes \mathcal{L}^l)$-measurable for every $k=1,2,\dotsc$ and every $\delta>0$.

    Finally, the desired measurability follows by observing that the map 
    \begin{align*}
     \Psi(x,\lambda):= \lim_{j \to \infty}  \lim_{k \to \infty} \Psi_{k,1/j}(x,\lambda) &= \lim_{j \to \infty} \lim_{k \to \infty} \sup_{0 <s < 1/k} \sup_{0 < \rho <1/j} \frac{\mu(X(x,\lambda,s,\rho) )}{(s\rho)^m} \\
     & =\lim_{j \to \infty} \limsup_{s \to 0+} \sup_{0 < \rho <1/j} \frac{\mu(X(x,\lambda,s,\rho) )}{(s\rho)^m} \\
     &  =\lim_{\delta \to 0^+} \limsup_{s \to 0+} \sup_{0 < \rho <\delta} \frac{\mu(X(x,\lambda,s,\rho) )}{(s\rho)^m},
    \end{align*}
    is $(\mu \otimes \mathcal{L}^l)$-measurable.

\subsubsection{Measurability of the set $\{\{(y,\lambda) \in \mathbb{R}^m \times \Lambda \ | \ y \in Y_\lambda \}\}$.}
\label{a:measYlambda}
To see that the set $\{(y,\lambda) \in \mathbb{R}^m \times \Lambda \ | \ y \in Y_\lambda \}$ is Borel measurable, notice that, the continuity of the map $\Pi \colon X \times \Lambda \to \mathbb{R}^m$ tells us that the map
\[
(y,\lambda) \longmapsto (\Pi_\lambda\mu)(B^m_r(y)), 
\]
is Borel measurable for every $r >0$. Hence, the map $(y,\lambda) \longmapsto f_\lambda(y)$ defined as
\[
f_\lambda(y) :=
\begin{cases}
    \lim_{r \to 0^+} r^{-m}(\Pi_\lambda\mu)(B^m_r(y)), &\text{ if the limit exists}\\
    0 &\text{ otherwise},
\end{cases}
\]
is Borel measurable. We can hence mollify the function $f(y,\lambda)$, to obtain a family $f_\epsilon$ of continuous functions such that $f_\epsilon(y,\lambda) \to f(y,\lambda)$ for $(\mathcal{H}^m \otimes \mathcal{L}^l)$-a.e. $(y,\lambda)$ as $\epsilon \to 0^+$. Since one easily verify that 
\[
(y,\lambda) \longmapsto \int_{B^m_r(y)} |f_\epsilon(y,\lambda)-a|\, d\mathcal{H}^m
\]
is Borel measurable for every $\epsilon >0$ every $r >0$, and every $r>0$, as well as
\[
\lim_{\epsilon \to 0^+} \int_{B^m_r(y)} |f_\epsilon(y,\lambda)-a|\, d\mathcal{H}^m = \int_{B^m_r(y)} |f_\lambda(y)-a|\, d\mathcal{H}^m, 
\]
it follows that the map
\[
(y,\lambda) \longmapsto \int_{B^m_r(y)} |f_\lambda(y)-a|\, d\mathcal{H}^m
\]
is Borel measurable for every $r >0 $ and every $a \in \mathbb{R}$.

Moreover, by using the monotonicity property for measure, we have that
\[
\limsup_{r \to 0+ } \frac{1}{r^m}\int_{B^m_r(y)} |f_\lambda(y)-a|\, d\mathcal{H}^m=\limsup_{\substack{q \to 0+ \\ q \in \mathbb{Q}}}  \frac{1}{q^m}\int_{B^m_q(y)} |f_\lambda(y)-a|\, d\mathcal{H}^m, 
\]
for every couple $(y,\lambda)$. Therefore, the set $A(a,n)$ made of those couple $(y,\lambda)$ such that
\[
\limsup_{r \to 0+ } \frac{1}{r^m}\int_{B^m_r(y)} |f_\lambda(y)-a|\, d\mathcal{H}^m< \frac{1}{n},  
\]
is Borel measurable. Eventually, we notice that 
\[\{(y,\lambda) \in \mathbb{R}^m \times \Lambda \ | \ y \in Y_\lambda \} = \bigcap_{n \geq 1} \bigcup_{a \in \mathbb{Q}} A(n,a),
\]
from which the desired property follows.

\section*{Acknowledgements}
I would like to express my sincere gratitude to Andrea Marchese for his generous investment of time and insightful discussions, which greatly enriched this work. His encouragement and enthusiasm were invaluable in motivating me to pursue this research.

This paper has been supported by the Austrian Science Fund (\textbf{FWF}) projects \textbf{Y1292} and \textbf{P35359}.


\begin{thebibliography}{99}




\bibitem{albmar} Alberti, G., and A, Marchese. \emph{On the differentiability of Lipschitz functions with respect to measures in the Euclidean space}. Geometric and Functional Analysis 26.1 (2016): 1-66.


\bibitem{almtas} Almi, S., and E. Tasso. \emph{A general criterion for jump set slicing and applications}. arXiv preprint


\bibitem{ags} Ambrosio, Luigi, Nicola Gigli, and Giuseppe Savar\'e. \emph{Gradient flows: in metric spaces and in the space of probability measures}. Springer Science \& Business Media, 2005.

\bibitem{ak1} Ambrosio, L., and B. Kirchheim. \emph{Currents in metric spaces}. Acta Mathematica 185.1 (2000): 1-80.


\bibitem{bal} Balogh, Z. M., and A. Iseli. \emph{Marstrand type projection theorems for normed spaces}. Journal of Fractal Geometry 6.4 (2019): 367-392.


\bibitem{bara} Barański, K., Y. Gutman, and A. Śpiewak. \emph{A probabilistic Takens theorem}. Nonlinearity 33.9 (2020): 4940.


\bibitem{bate1} Bate, David. \emph{Purely unrectifiable metric spaces and perturbations of Lipschitz functions}. Acta Math 224 (2020): 1-65.


\bibitem{bate2} Bate, D., M. Csörnyei, and B. Wilson. \emph{The Besicovitch-Federer projection theorem is false in every infinite-dimensional Banach space}. Israel Journal of Mathematics 220 (2017): 175-188.



\bibitem{chatol} Chang, A., and X. Tolsa. \emph{Analytic capacity and projections}. Journal of the European Mathematical Society 22.12 (2020): 4121-4159.



\bibitem{depa2} De Pauw, Thierry. \emph{An example pertaining to the failure of the Besicovitch-Federer structure theorem in Hilbert space}. Publicacions Matemàtiques, vol. 61, no. 1, 2017, pp. 153-73.


\bibitem{dr} De Philippis, Guido, and Filip Rindler. \emph{On the structure of $\mathscr{A}$-free measures and applications\emph}." Annals of Mathematics (2016): 1017-1039.



\bibitem{far} Farkas, \'Abel. \emph{Interval projections of self-similar sets.} Ergodic Theory and Dynamical Systems 40.1 (2020): 194-212.


\bibitem{fed1}
Federer, Herbert. \emph{Geometric measure theory}. Springer, 1996.


\bibitem{gal} Galeski, Jacek. \emph{Besicovitch-Federer projection theorem for continuously differentiable mappings having constant rank of the Jacobian matrix.} Mathematische Zeitschrift 289.3 (2018): 995-1010.



\bibitem{hov1} Hovila, Risto. \emph{Transversality of isotropic projections, unrectifiability, and Heisenberg groups}. Revista Matematica Iberoamericana 30.2 (2014): 463-476.



\bibitem{hov} Hovila, R., J\"arvenp\"a\"a, E., J\"arvenp\"a\"a, M., and Ledrappier, F. \emph{Besicovitch-Federer projection theorem and geodesic flows on Riemann surfaces}. Geometriae Dedicata 161.1 (2012): 51-61.


\bibitem{HW41} Hurewicz, Witold, and Henry Wallman. \emph{Dimension theory}. Vol. 4. Princeton university press, 2015.


\bibitem{ise} Iseli, Annina, and Anton Lukyanenko. \emph{Projection theorems for linear-fractional families of projections}. Mathematical Proceedings of the Cambridge Philosophical Society. Vol. 175. No. 3. Cambridge University Press, 2023.


\bibitem{joymor} Joyce, H., and P. Mörters. \emph{A set with finite curvature and projections of zero length}. J.
Math. Anal. Appl. 247 (1999): 126–135.


\bibitem{ken} Kenyon, Richard. \emph{Projecting the one-dimensional Sierpinski gasket}. Israel Journal of Mathematics 97 (1997): 221-238.



\bibitem{Mn81} Mañé, Ricardo. \emph{On the dimension of the compact invariant sets of certain non-linear maps}. Dynamical Systems and Turbulence, Warwick 1980: Proceedings of a Symposium Held at the University of Warwick 1979/80. Berlin, Heidelberg: Springer Berlin Heidelberg, 2006.


\bibitem{marc} Marchese, Andrea. \emph{A PDE perspective on the flat chain conjecture}. arXiv preprint arXiv:2511.06822 (2025).


\bibitem{mar} Marstrand, John M. \emph{Some fundamental geometrical properties of plane sets of fractional dimensions}. Proceedings of the London Mathematical Society 3.1 (1954): 257-302.


\bibitem{mat4} Mattila, Pertti. \emph{Rectifiability; a survey}. arXiv preprint arXiv:2112.00540 (2021).

\bibitem{mat7} Mattila, Pertti. \emph{Hausdorff dimension, projections, intersections, and Besicovitch sets}. New Trends in Applied Harmonic Analysis, Volume 2: Harmonic Analysis, Geometric Measure Theory, and Applications. Cham: Springer International Publishing, 2019. 129-157.

\bibitem{mat5} Mattila, Pertti. \emph{Hausdorff dimension, projections, and the Fourier transform}. Publicacions matematiques (2004): 3-48.


\bibitem{mat8} Mattila, Pertti. \emph{Smooth maps, null-sets for integralgeometric measure and analytic capacity}. Annals of Mathematics 123.2 (1986): 303-309.


\bibitem{mat2} Mattila, Pertti. \emph{An example illustrating integralgeometric measures}. American Journal of Mathematics (1986): 693-702.


\bibitem{myt} Mityagin, B.S. \emph{The Zero Set of a Real Analytic Function.} Math Notes 107, 529–530 (2020). https://doi.org/10.1134/S0001434620030189


\bibitem{per1} Peres, Yuval, and Boris Solomyak. \emph{The sharp Hausdorff measure condition for length of projections}. Proceedings of the American Mathematical Society 133.11 (2005): 3371-3379.


\bibitem{per} Peres, Yuval, and Wilhelm Schlag. \emph{Smoothness of projections, Bernoulli convolutions, and the dimension of exceptions.} Duke Math. J. 102(2) (2000): 193-251.


\bibitem{pug} Pugh, Harrison. \emph{A localized Besicovitch-Federer projection theorem}. arXiv preprint arXiv:1607.01758 (2016).



\bibitem{Rob11} Robinson, James C. \emph{Dimensions, embeddings, and attractors.} Vol. 186. Cambridge University Press, 2010.


\bibitem{schioppa} Schioppa, Andrea. \emph{Metric currents and Alberti representations}. Journal of Functional Analysis 271.11 (2016): 3007-3081.


\bibitem{simsol} Simon, Károly, and Boris Solomyak. \emph{Visibility for self-similar sets of dimension one in the plane}. (2006): 67-78.


\bibitem{SYC91} Sauer, Tim, James A. Yorke, and Martin Casdagli. \emph{Embedology}. Journal of statistical Physics 65.3 (1991): 579-616.


\bibitem{tas1} Tasso, Emanuele. \emph{Rectifiability of a class of integralgeometric measures and applications}. Courant Journal of Pure and Applied Matehmatics (CPAM). Accepted Paper (2026)


\bibitem{whi1} White, Brian. \emph{Rectifiability of flat chains}. Annals of Mathematics (1999): 165-184.


\bibitem{whi2} White, Brian. \emph{A new proof of Federer’s structure theorem for $k$-dimensional subsets of $\mathbb{R}^{n}$}. Journal of the American Mathematical Society 11.3 (1998): 693-701.



\bibitem{Whi36} Whitney, Hassler. \emph{Differentiable manifolds}. Annals of Mathematics 37.3 (1936): 645-680

\end{thebibliography}
\end{document}